\documentclass[imsarx]{imsart}
\RequirePackage{amsthm,amsmath,amsfonts,amssymb}
\RequirePackage[numbers]{natbib}
\RequirePackage[utf8]{inputenc}
\usepackage[T1]{fontenc}
\usepackage{bbm} 
\usepackage{todonotes}
\usepackage{accents}
\usepackage{comment}
\usepackage{multicol}
\usepackage[normalem]{ulem}
\usepackage{graphicx,wrapfig} 
\usepackage{framed} 
\usepackage{caption}  
\usepackage{enumerate} 
\usepackage{enumitem}
\usepackage{pgfplots}
\pgfplotsset{compat=1.14}
\usepackage{mathrsfs}
\usetikzlibrary{arrows}
\usepackage{caption}  
\usepackage[titletoc, title]{appendix} 
\usepackage{hyperref} 
\RequirePackage[colorlinks,citecolor=blue,urlcolor=blue]{hyperref}

\startlocaldefs
\numberwithin{equation}{section}
\newtheorem{Lemma}{Lemma}[section]

\newtheorem{Proposition}[Lemma]{Proposition}

\newtheorem{Theorem}[Lemma]{Theorem}

\newtheorem{Remark}[Lemma]{Remark}

\newtheorem{Corollary}[Lemma]{Corollary}
\newtheorem{Definition}[Lemma]{Definition}

\newtheorem{remark}[Lemma]{Remark}

\newtheorem{Claim}[Lemma]{Claim}
\newcommand{\sub}[1]{\boldsymbol{#1}}
\newcommand{\R}{\mathbb{R}}

\newcommand{\N}{\mathbb{N}}

\newcommand{\pr}{\mathbb{P}}
\newcommand{\E}{\mathbb{E}}

\newcommand{\prob}{\mathbb{P}}

\newcommand{\emp}{\varnothing}

\newcommand{\eqn}[1]{\begin{equation} #1 \end{equation}}
\newcommand{\eqan}[1]{\begin{align} #1 \end{align}}
\newcommand{\e}{\mathrm{e}}

\newcommand{\one}{\mathbbm{1}}
\newcommand{\var}{\rm{Var}}
\newcommand{\PAM}{\mathrm{PAM}}

\newcommand{\dir}[1]{\overset{\rightharpoonup}{#1}}

\newcommand{\PA}{\mathrm{PA}}
\newcommand{\PArs}{\mathrm{PA}^{\rm{\sss(A)}}}
\newcommand{\PArt}{\mathrm{PA}^{\rm{\sss(B)}}}
\newcommand{\PAri}{\mathrm{PA}^{\rm{\sss(D)}}}
\newcommand{\PArf}{\mathrm{PA}^{\rm{\sss(E)}}}
\newcommand{\RPPT}{\mathrm{RPPT}}
\newcommand{\sss}{\scriptscriptstyle}
\newcommand{\PU}{\mathrm{PU}}
\newcommand{\CPU}{\mathrm{CPU}}
\newcommand{\tree}{\mathrm{t}}
\newcommand{\din}{d^{\sss(\mathrm{in})}}

\newcommand{\old}{\scriptscriptstyle {\sf O}}
\newcommand{\Old}{{\scriptstyle {\sf O}}}
\newcommand{\young}{\scriptscriptstyle {\sf Y}}
\newcommand{\Young}{\scriptstyle {\sf Y}}

\newcommand{\infsupp}{\operatorname{inf~supp}}

\newcommand{\btr}{$\triangleright$}

\newcommand{\FC}[1]{#1}
\newcommand{\arxiversion}[1]{#1}
\newcommand{\journalversion}[1]{}

\newcommand{\Beta}{{\sf Beta}}

\newcommand{\nn}{\nonumber}

\definecolor{darkred}{rgb}{1,0,0}
\definecolor{darkgreen}{rgb}{0,0.6,0}
\definecolor{darkblue}{rgb}{0,0,1}
\definecolor{darkagenta}{rgb}{0.6,0,0.6}

\definecolor{ffzzqq}{rgb}{1,0.6,0}
\definecolor{qqzzqq}{rgb}{0,0.6,0}
\definecolor{xdxdff}{rgb}{0.49019607843137253,0.49019607843137253,1}
\definecolor{ttttff}{rgb}{0.2,0.2,1}
\definecolor{ududff}{rgb}{0.30196078431372547,0.30196078431372547,1}

\endlocaldefs

\setlist[enumerate]{itemsep = -0.2em}
\begin{document}


\begin{frontmatter}
\title{Universality of the local limit\\ of preferential attachment models }
\runtitle{Local convergence of preferential attachment models}

\begin{aug}
\author[A]{\fnms{Alessandro}~\snm{Garavaglia}\ead[label=e1]{ale.garavaglia@gmail.com}},
\author[B]{\fnms{Rajat Subhra}~\snm{Hazra}\ead[label=e2]{r.s.hazra@math.leidenuniv.nl}},
\author[A]{\fnms{Remco}~\snm{van der Hofstad}\ead[label=e3]{r.w.v.d.hofstad@TUE.nl}}
\and
\author[C]{\fnms{Rounak}~\snm{Ray}\ead[label=e4]{rounak\textunderscore ray@brown.edu}}
\address[B]{University of Leiden, Niels Bohrweg 1, 2333 CA, Leiden, The Netherlands\printead[presep={,\ }]{e2}}

\address[A]{Department of Mathematics and
	Computer Science, Eindhoven University of Technology, 5600 MB Eindhoven, The Netherlands \printead[presep={,\ }]{e1,e3}}
	
\address[C]{Division of Applied Mathematics, Brown University, 182 George Street, Providence, RI 02906, United States of America \printead[presep={,\ }]{e4}}
\end{aug}

\begin{abstract}
We study preferential attachment models where vertices enter the network with i.i.d.\ random numbers of edges that we call the {\em out-degree}. We identify {the} local limit of such models, substantially extending the work of {Berger et al.\ \cite{BergerBorgs}.} The degree distribution of this limiting random graph, which we call the {\em random P\'{o}lya point tree}, has a surprising size-biasing {property}.  

Many of the existing preferential attachment models can be viewed as special cases of our preferential {attachment} model with i.i.d.\ {out-degrees}. Additionally, our models incorporates negative values of {the preferential attachment fitness} parameter, which allows us to consider {preferential attachment models} with infinite-variance degrees.

{Our proof of local convergence} consists of two main steps: a P\'olya urn description of our graphs, and {an} explicit identification of the neighbourhoods in them. We provide a novel {and explicit proof} to establish a coupling between the preferential attachment model and the P\'{o}lya urn graph. Our result proves a density convergence result, for fixed ages of vertices in the local limit.
\end{abstract}



\end{frontmatter}



\section{Introduction}
\label{sec:intro:polya}
\subsection{{Real-world networks and preferential} attachment models}

Empirical studies on real-life networks reveal that most of these networks {(a)} grow with time; {(b)} are small worlds, meaning that typical distances in the network are small; and {(c)} have power-law degree sequences. 
The Barab\'asi-Albert model of \cite{ABrB,albert1999} is {the most popular random graph model for such real-life networks} due to the fact that, through a simple dynamic, its properties resemble real-world networks. This model has been generalized in many different ways, creating a wide variety of {\em preferential attachment models} (PAMs).

By PAMs we denote a class of random graphs with a common dynamics: At every time step, a new vertex appears in the graph, and it connects {to} $m\geq1$ existing vertices, with probability proportional to a function of the degrees {of the vertices}. In other words, when vertex $n\in\N$ appears, it connects to vertex $i\leq n$ with probability
\eqn{
	\label{for:heur:PAfunction}
	\pr(n\rightsquigarrow  i \mid \mathrm{PA}_{n-1}) \propto f(d_i(n-1)),
}	
where {$\PA_{n}$ is the preferential attachment graph with $n$ vertices,} $d_i(n-1)$ denotes the degree of vertex $i$ in the graph $\PA_{n-1}$, and $f$ is some {preferential attachment} function. We {thus} in fact {consider a whole} PA class of random graphs since every function $f$ defines a different model.
The original Barab\'asi-Albert model is retrieved with $f(k) = k$. The literature often considers the so-called {\em affine} PAM, where $f(k) = k+\delta$, for some constant $\delta>-m$.

The constant $\delta$ allows for flexibility in the graph structure. In fact, the power-law exponent of the degree distribution is given by $\tau= 3+\delta/m$ \cite{BRST01, Der2009,vdH1}, and in general, for $m\geq 2$, the typical distance and the diameter are of order $\log\log n$ when $\tau\in(2,3)$, while they are of order $\log n$ when $\tau>3$. When $\tau=3$, distances and diameter are of order $\log n/\log\log n$ {instead} \cite{BolRio04b,CarGarHof,DSMCM,DSvdH}.

In {PAMs}, the degree of a vertex increases over time and higher degree vertices are prone to attract the edges incident to new vertices, increasing their degrees even further. In literature, this is sometimes referred {to as} \textit{rich-get-richer} effect. {Models} where the vertex degrees are determined by a \textit{weight} associated to it are sometimes called \textit{rich-by-birth} {models}. 

In \cite{Dei}, the authors have considered a model incorporating both of these effects. The model is a {PAM} with \emph{random} out-degrees, i.e., every vertex joins the existing network with a random number of edges that it connects to the existing vertices preferentially to their degrees at that time, as done in usual preferential attachment models. Here, the authors have shown that this system also shows the power-law degree distribution. {Jordan \cite{jordan06} considered} a special case of this particular model to analyze their results on degree distribution of the random graph. 
Cooper and Frieze \cite{cooper2003} have shown  that a further generalised version of this PAM also {has a} power-law degree distribution. Since most real-life networks are dynamic in nature and have power-law degrees, the PAM is often used to model them. However, in such networks it is never the case that every vertex joins with exactly the same out-degree, and {thus} i.i.d.\ out-degrees are more realistic. Gao and van der Vaart \cite{gao-avdvaart} have shown that the maximum likelihood estimator of the fitness parameter $\delta$ of PAM with random out-degrees is asymptotically normal and has asymptotically minimal variance.
There have been further generalisations by allowing younger vertices to have higher degrees in PAMs through a random {\em fitness} parameter. More precisely, individual factors can be assigned  to vertices, obtaining that the probabilities in \eqref{for:heur:PAfunction} are proportional to $\eta_i f(d_i(n-1))$, thus obtaining PAMs {\em with multiplicative fitness}, or proportional to $d_i(n-1)+\eta_i$, {giving rise to} {\em additive fitness}, {where $(\eta_i)_{i\geq 1}$ indicate the i.i.d.\ vertex fitnesses} \cite{Bhamidi,Bianconi,Borgs,der16,der2014}.

{
	Similar to many often other random graph models, such as the configuration model, PAMs are called {\em locally {tree-like} graphs}, meaning that the neighbourhood of the majority of vertices is structured as a tree (up to a certain distance). This idea can be formalized using the notion of local convergence (i.e., the Benjamini-Schramm limit), introduced in \cite{Aldous2004,benjamini}. {Local} convergence turns out to be an extremely {versatile} tool to understand the geometry of the graph. The basic idea is to explore the neighbourhood of a uniformly chosen vertex of the graph {up to a finite distance}, to understand its distributional and geometric properties. We refer to, e.g.,  \cite[{Chapter 2}]{vdH2} for {an overview of} the theory and applications of the above concept.
	
	Berger et al.\ initiated the study of local convergence of PAMs {in \cite{BergerBorgs}}. They showed that the finite neighbourhood of the graph converges to the corresponding neighbourhood of the P\'olya point tree (see {the description of the P\'olya point tree} in Section~\ref{sec:RPPT}). The proof {uses} a P\'olya urn representation introduced in \cite{BBCS05} to study the spread of viral infections on networks. 
	\smallskip\noindent
	\paragraph{ Main results and innovation of this paper}
	The main aim of this article is to {extend} the local convergence proof in \cite{BergerBorgs} {to a more general class of PAMs (including random out-degrees and related dynamics).} {We achieve this by explicitly computing the} {asymptotic} density of neighbourhoods of the PAMs. This helps us to extend the result by \cite{BergerBorgs} to models where one can accommodate negative fitness parameters and random {out-degrees}. The limiting random tree is an extension of the P\'olya point tree described in \cite{BergerBorgs}, which we call the {\em random P\'olya point tree}.
	
	The randomness of the {out-degrees} provides a surprising {size-biasing effect} in the limiting random tree. We show that there is {a} universal description of the limit by considering many possible affine variants of the PAMs. Additionally, we study {the} {\em vertex-marked local convergence in probability} of the PAMs, which is an extension to the local convergence shown in \cite{BergerBorgs}. Here, the marks denote the {\em ages} of the vertices in the tree, and we prove convergence of the joint \emph{densities} of these ages.
	
	In the next section we provide details of the various preferential attachment models considered in this article. We also provide a formal definition of the vertex-marked local limit and the formulation of our main results.}
\subsection{The {models}}\label{sec:model}
{Several versions} of preferential attachment models are available in the literature. We generalize these definitions to the case of {a} random initial number of edges that connect to the already available graph. We refer to these edge numbers as {\em out-degrees}, even though we consider our model to be {\em undirected.}
{Let $(m_i)_{i\ge 3}$ be {a sequence of i.i.d.\ copies of an $\N$-valued random variable $M$} with finite $p$-th moment for some $p>1$, and $\delta >-{\infsupp}(M)$ be a fixed real number, {where $\operatorname{supp}(M)$ denotes the support of the random variable $M$}. }
In our models, every new vertex $v$ joins the graph with $m_v$ many edges incident to it. {In the classical preferential attachment models, instead,} every new vertex comes with a {\em fixed} number of edges incident to it. {In case of self-loops, the degree of the new vertex increases by $2$.} {Thus,} we can consider the existing models as a degenerate case of ours. Define $m_1=m_2=1$, $\boldsymbol{m}=(m_1,m_2,m_3,\ldots)$ and
\[
m_{[l]} = \sum\limits_{i\leq l} m_i.
\]
To describe the {edge-connection probabilities} and to simplify our calculations, we {frequently work with the conditional law given $\boldsymbol{m}=(m_i)_{i\ge 1}$}. The conditional measure is denoted by $\prob_m$, i.e., for any event $\mathcal{E}$,
\eqn{\label{eq:conditional_measure:M} \prob_m(\mathcal{E})= \prob\left( \left. \mathcal{E} \right| \boldsymbol{m} \right).}
{Conditionally} on $\boldsymbol{m}$, we define some {versions of the preferential attachment models,} special cases of which {are} equivalent to the models~(a-g) described in \cite{GarThe}. We consider the initial graph to be $G_0$ with $2$ vertices with degree $a_1$ and $a_2$, respectively. {Throughout the remainder of this paper, we fix}  $\delta>-\infsupp(M)$.
\smallskip
\paragraph{ Model (A)} {This is} the generalized version of \cite[Model (a)]{vdH1}. {In this model, conditionally on the existing graph, every new vertex joins the graph and, with all its out-edges, connects to the existing vertices in the graph (\emph{including itself}) with probability proportional to the degree of the receiving vertices at that time. In particular, the degrees are updated while the out-edges are being connected, a feature which we call {\em intermediate updating}.} Conditionally on $\boldsymbol{m}$, for $v\in\N$ and $j=1,\ldots,m_v$, the attachment probabilities are given by
\eqn{
	\label{eq:edge_connecting:model(rs):2}
	\prob_m\left(\left. v\overset{j}{\rightsquigarrow} u\right| \PArs_{v,j-1}(\boldsymbol{m},\delta) \right) = \begin{cases}
		\frac{d_u(v,j-1)+\delta}{c_{v,j}}\qquad\qquad \mbox{ for $v>u$},\\
		\frac{d_u(v,j-1)+1+\frac{j}{m_u}\delta}{c_{v,j}}\qquad \mbox{ for $v=u$},
	\end{cases}
}
where $v\overset{j}{\rightsquigarrow} u$ denotes that vertex $v$ connects to $u$ with its \(j\)-th edge, $\PArs_{v,j}(\boldsymbol{m},\delta)$ denotes the graph with $v$ vertices, {with} the $v$th vertex {already having connected its first} $j$ out-edges, and $d_u(v,j)$ denotes the degree of vertex $u$ in $\PArs_{v,j}(\boldsymbol{m},\delta)$. We identify $\PArs_{v+1,0}(\boldsymbol{m},\delta)$ with $\PArs_{v,m_v}(\boldsymbol{m},\delta)$. The normalizing constant $c_{v,j}$ in \eqref{eq:edge_connecting:model(rs):2} equals
\eqn{\label{eq:normali}
	c_{v,j} := a_{[2]}+2\left( m_{[v-1]}+j-2 \right) - 1+ (v-1)\delta + \frac{j}{m_v}\delta \, ,
}
where $a_{[2]}=a_1+a_2$. We denote the above model by $\PArs_v(\boldsymbol{m},\delta)$. If we consider $M$ as a degenerate distribution {that is equal to $m$ a.s.}, then $\PArs_v(\boldsymbol{m},\delta)$ is essentially equivalent to $\PA_v^{\sss ({m},\delta)}(a)$ defined in \cite{vdH2}. $\PA_v^{({m,0})}(a)$ was {informally introduced by \cite{albert1999} and first studied rigorously in \cite{BolRio04b}.}
Later the model for general $\delta$ was described in \cite{ABrB}.
\smallskip
\paragraph{ Model (B)} {This is} the generalized version of \cite[Model (b)]{vdH1}. {This model is a variant of the model (A), where the first edge of the new vertices cannot form a self-loop.} Conditionally on $\boldsymbol{m}$, the attachment probabilities are given by
\eqn{\label{eq:edge_connecting:model(rt):2}
	\prob_m\left(\left. v\overset{j}{\rightsquigarrow} u\right| \PArt_{v,j-1}(\boldsymbol{m},\delta) \right) = \begin{cases}
		\frac{d_u(v,j-1)+\delta}{c_{v,j}}\qquad\qquad \mbox{ for $v>u$},\\
		\frac{d_u(v,j-1)+\frac{(j-1)}{m_u}\delta}{c_{v,j}}\qquad \mbox{ for $v=u$},
\end{cases}}
where again $\PArt_{v,j}(\boldsymbol{m},\delta)$ denotes the graph with $v$ vertices, {with} the $v$th vertex {already having connected its first} $j$ out-edges, and $d_u(v,j)$ denotes the degree of vertex $u$ in $\PArt_{v,j}(\boldsymbol{m},\delta)$. We identify $\PArt_{v+1,0}(\boldsymbol{m},\delta)$ with $\PArt_{v,m_v}(\boldsymbol{m},\delta)$. The normalizing constant $c_{v,j}$ in \eqref{eq:edge_connecting:model(rt):2} now equals
\[
c_{v,j} = a_{[2]}+2\left( m_{[v-1]}+j-3 \right) +(v-1)\delta + \frac{(j-1)}{m_v}\delta \,.
\]
We denote the above model by $\PArt_v(\boldsymbol{m},\delta)$. For $M$ a degenerate distribution {which is equal to $m$ a.s.}, we obtain $\PA_v^{\sss ({m},\delta)}(b)$ described in \cite{vdH2} from $\PArt_v(\boldsymbol{m},\delta)$.
{\begin{remark}[Difference between Models (A) and (B)]
	\rm{From the definition of the models (A) and (B), the models are different in that the first edge from every new vertex can create a self-loop in model (A) but not in model (B). Note that the edge probabilities are different for all $j$.}\hfill$\blacksquare$
\end{remark}}
\paragraph{ Model (D)} This model is the generalized {version} of the sequential model described in \cite{BergerBorgs}. {In this model, every new vertex connects all its out-edges to the existing vertices in the graph (\emph{excluding itself}) with probability proportional to their degree at that time, with intermediate updating.}

We start with a fixed graph $G_0$ of size $2$ and degrees $a_1$ and $a_2$, respectively. Denote the graph {by} $\PAri_{v}(\boldsymbol{m},\delta)$ when {the graph has $v$ vertices.} 
For $v\geq 3$, vertex $v$ enters the system with $m_v$ many out-edges {whose} edge-connection probabilities are given by
\eqn{\label{sec:model:eq:1}
	\prob_m\left(\left. v\overset{j}{\rightsquigarrow} u \right| \PAri_{v,j-1} \right) = 
	\frac{d_u(v,j-1)+\delta}{c_{v,j}}\quad \text{for}\quad v>u,
}
where $\PAri_{v,j}(\boldsymbol{m},\delta)$ denotes the graph with $v$ vertices, {with} the $v$th vertex {already having connected its first} $j$ out-edges, and $d_u(v,j)$ is the degree of the vertex $u$ in the graph $\PAri_{v,j}$, and
\begin{equation*}
	c_{v,j}= a_{[2]}+2\left( m_{[v-1]}-2 \right)+(j-1) + (v-1)\delta\,.
\end{equation*}
Again we identify $\PAri_{v+1,0}(\boldsymbol{m},\delta)$ with $\PAri_{v,m_v}(\boldsymbol{m},\delta)$. For $M$ a degenerate random variable {that is equal to $m$ a.s.}, we obtain $\PA_v^{\sss ({m},\delta)}(d)$, {as} described in \cite{vdH2}.
\smallskip
{\paragraph{ Model (E)} This model is the independent model proposed {by {Berger et al.\ in \cite{BergerBorgs}}}. {In this model, conditionally on the existing graph, every new vertex, independently across all its out-edges, connects to the existing vertices (excluding itself) with probability proportional to their degree at the time. In particular, the degrees of the vertices are not updated as the edges are connected.}

We start with the same initial graph $G_0$. We name the graph $\PArf_{v}(\boldsymbol{m},\delta)$ when there are $v$ vertices in the graph. Conditionally on $\boldsymbol{m},$ every new vertex $v$ is added to the graph with $m_v$ many {out-edges}. Every {out-}edge $j\in[m_v]$ from $v$ connects to one of the vertex $u\in[v-1]$ independently with the {probabilities}
	\eqn{\label{eq:model:f:1}
		\prob_m\left( \left.v\overset{j}{\rightsquigarrow}u\right|\PArf_{v-1}(\boldsymbol{m},\delta) \right) = \frac{d_u(v-1)+\delta}{a_{[2]}+2(m_{[v-1]}-2)+(v-1)\delta},}
	where $d_{u}(v-1)$ is the degree of the vertex $u$ in $\PArf_{v-1}$. Note that the intermediate degree updates are not there in this graph. Similarly as {in} model (D), no self-loops are allowed here.}
\smallskip
\paragraph{Model (F)} The independent model does not allow for self-loops, but multiple {edges} between two vertices {may still occur}. {Such multigraphs are, in many applications, unrealistic, and we next propose a model where the graphs obtained are {\em simple}, i.e., without self-loops and multiple edges.} Model (F) is a variation of the independent model {(i.e., without intermediate updating),} where the new vertices connect to the existing vertices in the graph independently {but {\em without replacement}} and the edge-connection probabilities are {similar to} \eqref{eq:model:f:1}. {It can be thought of as an extension of the \emph{conditional model} described in \cite{BergerBorgs}}. Indeed, for $j\geq 2$, the normalization factor changes a little due to the fact that vertices that have already appeared as neighbours are now forbidden. {Every {out-}edge $j\in[m_v]$ from $v$ connects to one of the vertex $u\in[v-1]$ independently with the {probabilities}
	\eqn{\label{eq:model:f:2}
		\prob_m\left( \left.v\overset{j}{\rightsquigarrow}u\right|\PA_{v,j-1}^{\rm{\sss (F)}}(\boldsymbol{m},\delta) \right) = \frac{(d_u(v-1)+\delta)}{c_{v,j}}\Big(1-\sum\limits_{i=1}^{j-1}\one_{\{v\overset{{i}}{\rightsquigarrow}u\}}\Big),}
	where $d_{u}(v-1)$ is the degree of the vertex $u$ in $\PA_{v,j-1}^{\rm{\sss (F)}}$, and \(c_{v,j}\) updates as
\[
		c_{v,j}= a_{[2]}+ 2( m_{[v-1]}-2)+ (v-1)\delta- \sum_{i=1}^{j-1} \sum_{w\in [v-1]} (d_w(v-1)+ \delta)\one_{[v\overset{{i}}{\rightsquigarrow} w]}~.
\]
We identify $\PA_{v+1,0}^{\rm{\sss (F)}}(\boldsymbol{m},\delta)$ with $\PA_{v,m_v}^{\rm{\sss (F)}}(\boldsymbol{m},\delta)$.}
The degenerate case of our PAM, i.e., with fixed out-degree, is studied in \cite{basu2014} for analyzing the largest connected component in the strong friendship subgraph of evolving online social networks.

\begin{remark}[{The missing model (c)}]
	{\rm There is a description of model (c) in \cite[Section 8.2]{vdH1}, which {reduces to model (a), so we refrain from discussing it further in this article}. }\hfill$\blacksquare$ 
\end{remark}
\subsection{The space of rooted vertex-marked graphs and marked local convergence}
Local convergence of rooted graphs was first introduced by Benjamini and Schramm in \cite{benjamini} {and Aldous and Steele in \cite{Aldous2004}.} {We now give a brief introduction.} 

A graph $G=(V(G), E(G))$ ({possibly} infinite) is called \textit{locally finite} if every vertex $v\in V(G)$ has finite degree (not necessarily uniformly bounded). A pair $(G,o)$ is called a {{\em rooted graph}, where $G$ is} rooted at $o\in V(G)$. For any two vertices $u,v\in V(G),~d_G(u,v)$ is defined as the length of the shortest path from $u$ to $v$ in $G$, {i.e., the minimal number of edges needed to connect $u$ and $v$}. 

For a rooted graph $(G,o)$, the {\em $r$-neighbourhood} of the root $o$, {denoted by $B_r^{\sss(G)}(o)$, is defined as the graph rooted at $o$ with vertex and edge sets given by}
\eqn{\label{eq:def:nbhd:root}
	\begin{split}
		V\big( B_r^{\sss(G)}(o) \big) =& \{ v\in V(G)\colon d_G(o,v)\leq r \},\quad\text{and}\\
		E\big( B_r^{\sss(G)}(o) \big) =& \left\{ \{ u,v \}\in E(G)\colon u,v\in V\big( B_r^{\sss(G)}(o) \big) \right\}\,.
\end{split}}

Let $(G_1,o_1)$ and $(G_2,o_2)$ be two rooted locally finite graphs with $G_1=(V(G_1),E(G_1))$ and $G_2=(V(G_2),E(G_2))$. Then we say that $(G_1,o_1)$ is {\em{rooted isomorphic}} to $(G_2,o_2)$, {which we denote as $(G_1,o_1)\simeq (G_2,o_2)$}, {when} there exists a graph isomorphism between $G_1$ and $G_2$ that maps $o_1$ to $o_2$, i.e., {when} there exists a bijection $\phi\colon V(G_1)\mapsto V(G_2)$ such that 
\eqn{\label{def:rooted:isomorphism}
	\begin{split}
		&\phi(o_1)=o_2,\qquad\text{and} \qquad\{u,v\}\in E(G_1) \iff \{\phi(u),\phi(v)\}\in E(G_2).
\end{split}} 

{\em Marks} are generally images of an injective function $\mathcal{M}$ acting on the vertices of a graph $G,$ as well as {on} the edges. 
{These} marks take values in a complete separable metric space $(\Xi,d_{\Xi})$. Rooted graphs with only vertices having marks are called \textit{vertex-marked} rooted graphs, {and are denoted} by $(G,o,\mathcal{M}(G))= (V(G),E(G),o,\mathcal{M}(V(G)))$. We say that two graphs $(G_1,o_1,\mathcal{M}_1(G_1))$ and $(G_2,o_2,\mathcal{M}_2(G_2))$ are vertex-marked-isomorphic, {which we denote as $(G_1,o_1,\mathcal{M}_1(G_1))\overset{\star}{\simeq}(G_2,o_2,\mathcal{M}_2(G_2))$,} {when} there exists a bijective function $\phi\colon V(G_1)\mapsto V(G_2)$ such that
\begin{itemize}
	\item[$\rhd$] $\phi(o_1)=o_2$;
	\item[$\rhd$] $(v_1,v_2)\in E(G_1)$ implies that $(\phi(v_1),\phi(v_2))\in E(G_2)$;
	\item[$\rhd$] $\mathcal{M}_1(v)=\mathcal{M}_2(\phi(v))$ for all $v\in V(G_1)$.
\end{itemize}

We {let} $\mathcal{G}_\star$ be the vertex-marked isomorphism invariant class of rooted graphs. Similarly as {in} {\cite[Definition 2.11]{vdH2}}, {one} can define a metric $d_{\mathcal{G}_\star}$ as
\eqn{\label{for:def:distance:1}
	d_{\mathcal{G}_\star}\left( (G_1,o_1,\mathcal{M}_1(G_1)),(G_2,o_2,\mathcal{M}_2(G_2)) \right) = \frac{1}{1+R^\star}\, ,}
where
\eqan{\label{for:def:distance:2}
		R^\star = \sup\{ r\geq 0: B_r^{(G_1)}(o_1)\simeq B_r^{(G_2)}(o_2),\mbox{and there exists a bijective function $\phi$ from}&\nn\\
		\mbox{ $V(G_1)$ to $V(G_2)$ such that }d_{\Xi}(\mathcal{M}_1(u),\mathcal{M}_2(\phi(u)))\leq \frac{1}{r},\forall u\in V(B_r^{(G_1)}(o_1))&\},}
which makes $\left(\mathcal{G}_\star,d_{\mathcal{G}_\star}\right)$ a Polish space.
We {next} define the notion of marked local convergence of vertex-marked random graphs on this space. {\cite[Theorem~2.14]{vdH2} describes various} notions of local convergence.
For the next two definitions of local convergence, we consider that $\{ G_n \}_{n\geq 1}$ is a sequence of (possibly random) vertex-marked graphs with $G_n=\left( V(G_n),E(G_n),\mathcal{M}_n(V(G_n)) \right)$ that are finite (almost surely). Conditionally on $G_n,$ let $o_{n}$ be a randomly chosen vertex from $V(G_n)$. Note that $\{ (G_n,o_n) \}_{n\geq 1}$ is a sequence of random variables defined on $\mathcal{G}_\star$. Then {vertex-marked local convergence is defined as follows:}
\begin{Definition}[Vertex marked local  convergence]\label{def:vertex:marked:local:convergence}
	\smallskip\noindent
	{ \begin{itemize}
			\item[(a)] {\bf (Vertex-marked local weak convergence)} The sequence $\{(G_n,o_n)\}_{n\geq 1}$ is said to converge {{{vertex-marked locally weakly}}} to a random element $(G,o,\mathcal{M}(V(G)))\in \mathcal{G}_\star$ having probability law $\mu_\star$, if, for every $r>0$, and every $(H_\star,\mathcal{M}_{H_\star}(H_\star))\in\mathcal{G}_\star$, as $n\to\infty$,
			\eqan{\label{eq:def:mark:local:weak:convergence}
					&\prob\left( B_r^{\sss(G_n)}({o_n})\simeq H_\star,~d_{\mathcal{G}_\star}\big( (B_r^{\sss(G_n)}({o_n}),o_n,\mathcal{M}(V(B_r^{\sss(G_n)}({o_n}))),(H_\star,\mathcal{M}_{H_\star}(H_\star)) \big)\leq \frac{1}{r} \right)\nn\\
					&\hspace{1cm}\to \mu_\star\left( B_r^{\sss(G)}({o})\simeq H_\star,~~d_{\mathcal{G}_\star}\big( (B_r^{\sss(G)}({o}),o,\mathcal{M}(V(B_r^{\sss(G)}({o}))),(H_\star,\mathcal{M}_{H_\star}(H_\star)) \big)\leq \frac{1}{r} \right).}
			\smallskip\noindent
			\item[(b)] {\bf (Vertex-marked local convergence)}
			The sequence $\{ G_n \}_{n\geq 1}$ is said to converge {{{vertex-marked locally}}} in probability to a random element $(G,o,\mathcal{M}(V(G)))\in \mathcal{G}_\star$ having probability law $\mu_\star$, when for every $r>0$, and for every $(H_\star,\mathcal{M}_{H_\star}(H_\star))\in\mathcal{G}_\star$, as $n\to\infty$,
			\eqan{\label{eq:def:mark:local:convergence}
					&\frac{1}{n}\sum\limits_{\omega\in[n]}\one_{\left\{ B_r^{\sss(G_n)}({\omega})\simeq H_\star,d_{\mathcal{G}_\star}\big( (B_r^{\sss(G_n)}({\omega}),\omega,\mathcal{M}(V(B_r^{\sss(G_n)}({\omega}))),(H_\star,\mathcal{M}_{H_\star}(H_\star)) \big)\leq \frac{1}{r} \right\}}\nn\\
					&\hspace{1cm}\overset{\prob}{\to} \mu_\star\left( B_r^{\sss(G)}({o})\simeq H_\star,~d_{\mathcal{G}_\star}\big( (B_r^{\sss(G)}({o}),o,\mathcal{M}(V(B_r^{\sss(G)}({o}))),(H_\star,\mathcal{M}_{H_\star}(H_\star)) \big)\leq \frac{1}{r} \right).}
	\end{itemize}}
	\hfill$\blacksquare$
\end{Definition}
\subsection{Definition of {the} random P\'{o}lya point tree}\label{sec:RPPT}
In this section, {we define the random P\'olya point tree (RPPT) that will act as the vertex-marked local limit of our preferential attachment graphs. We start by defining the vertex set of this RPPT:}

\begin{Definition}[Ulam-Harris set and its ordering]\label{def:ulam}
	{\rm Let $\N_0=\N\cup\{0\}$. {The {\em Ulam-Harris set}} is
		\begin{equation*}
			\mathcal{U}=\bigcup\limits_{n\in\N_0}\N^n.
		\end{equation*}
		For $x = x_1\cdots x_n\in\N^n$ and $k\in\N$ we denote {the element $x_1\cdots x_nk$ by $xk\in\N^{n+1}$.} The~{\em root} of the Ulam-Harris set is denoted by $\emp\in\N^0$.
		
		For any $x\in \mathcal{U}$, we say that $x$ has length $n$ if $x\in\N^n$. {The lexicographic \textit{ordering}} between the elements of the Ulam-Harris set {is as follows:}
		\begin{itemize}
			\item[(a)] for any two elements $x,y\in\mathcal{U}$, {$x>y$ when the length of $x$ is more than that of $y$;}
			\item[(b)] if $x,y\in \N^n$ for some $n$, then $x>y$ if there exists $i\leq n,$ such that $x_j=~y_j~\forall j<i$ and $x_i>~y_i$.\hfill$\blacksquare$
	\end{itemize}}
\end{Definition}

We use the elements of the Ulam-Harris set to identify nodes in a rooted tree, since the notation in Definition~\ref{def:ulam} allows us to denote the relationships between children and parents, where for $x\in\mathcal{U}$, we denote the $k$th child of $x$ by the element $xk$.
\smallskip

\paragraph{Random P\'olya point tree (RPPT)}
The $\RPPT(M,\delta)$ is an {\em infinite multi-type rooted random tree} where $M$ is an $\N$-valued random variable and $\delta>-\infsupp(M)$. {It is a multi-type branching process, with a mixed continuous and discrete type space. We now describe its properties one by one.}

\paragraph{ \bf Descriptions of the distributions and parameters used}
\begin{enumerate}
	\item[{\btr}] {{Define} $i = \frac{\E[M]+\delta}{2\E[M]+\delta}.$}
	\smallskip
	\item[{\btr}] {Let} $\Gamma_{\rm in}(m)$ denote a Gamma distribution with parameters $m+\delta$ and $1$, where $m\in \N$. 
	\smallskip
	\item[{\btr}] {For $\delta>-\infsupp(M),$ let}  $M^{(\delta)}$ be an $\N$-valued random variable such that $$\prob\left(M^{(\delta)}=m\right) = \frac{m+\delta}{\E[M]+\delta}\prob(M=m),$$ i.e., $M^{(\delta)}+\delta$ is a size-biased version of $M+\delta$. In particular, $M^{(0)}$ is the size-biased distribution of $M$.
\end{enumerate}
\smallskip

\paragraph{Feature of the vertices of {the} RPPT}
Below, to avoid confusion, we use `node' for a vertex in the RPPT and `vertex' for a vertex in the PAM. We now discuss the properties of the nodes in $\RPPT(M,\delta)$. Every node except the root in the $\RPPT$ has
\begin{enumerate}
	\item[{\btr}] a {\em label} $\omega$ in the Ulam-Harris set {$\mathcal{U}$} (recall Definition~\ref{def:ulam});
	\smallskip
	\item[{\btr}] an {\em age} $A_{\omega}\in[0,1]$;
	\smallskip
	\item[{\btr}] a positive number $\Gamma_{\omega}$ called its {\em strength};
	\smallskip
	\item[{\btr}] a {label in $\{{\Old},{\Young}\}$} depending on the age of the {node} and its parent, {with $\Young$ denoting that the node is younger than its parent, and $\Old$ denoting that the node is older than its parent.}
\end{enumerate}
{Note that \emph{age} stands for the birth-time of the given node. Hence, $\omega$ is labeled $\Young$ when $A_{\omega}$ is larger than its parent's age in the tree.}
Based on its type being $\Old$ or $\Young$, every {node} $\omega$ has an independent $\N$-valued random variable $m_{-}(\omega)$ associated to it. {If $\omega$ has type $\Old$, then $m_{-}(\omega)$ is distributed as $M^{(\delta)}$, while if $\omega$ has type $\Young$, then $m_{-}(\omega)$ is distributed as $M^{(0)}-1$. For all nodes with label \(\omega\),
given $m_{-}(\omega), \Gamma_{\omega}$ is distributed as $\Gamma_{\rm in}\left( m_{-}(\omega)+1 \right)$.}

\smallskip
\begingroup
	\allowdisplaybreaks
	\paragraph{Construction of the RPPT} We next use the above definitions to construct the RPPT using an {\em exploration process}. {For any node with label \(\omega\), we perform the {\em neighbourhood exploration} as follows:
	\begin{enumerate}
		\item Sample $m_{-}(\omega)$ i.i.d.\ random variables $U_{\omega1},\ldots,U_{\omega m_{-}(\omega)}$ independently from all the previous steps and from each other, uniformly on $[0,1]$. To nodes $\omega 1,\ldots,\omega m_{-}(\omega)$ assign the ages $U_{\omega1}^{1/i}A_\omega,\ldots,U_{\omega m_{-}(\omega)}^{{1/i}}A_\omega$ and type $\Old$ and set them unexplored;
		\item Let $A_{\omega(m_{-}(\omega)+1)},\ldots,A_{\omega(m_{-}(\omega)+\din_\omega)}$ be the random $\din_\omega$ points given by a conditionally independent Poisson process on $[A_{\omega},1]$ with intensity
		\eqn{
			\label{for:pointgraph:poisson}
			\rho_{\omega}(x) = {(1-i)}{\Gamma_\omega}\frac{x^{-i}}{A_\omega^{1-i}}.
		}
		Assign these ages to $\omega(m_{-}(\omega)+1),\ldots,\omega(m_{-}(\omega)+\din_\omega)$, {assign them type $\Young$ } and set them unexplored;
		\item Draw an edge between $\omega$ and each one of the nodes $\omega 1,\ldots,\omega(m_{-}(\omega)+\din_\omega)$;
		\item Set $\omega$ as explored.
	\end{enumerate}
	}
	The root is special in the tree. It has label $\emp$ and its age $A_\emp$ is an independent uniform random variable in $[0,1]$. Since the root $\emp$ has no type, $m_{-}(\emp)$ is distributed as $M$, {and $\Gamma_\emp$ is distributed as $\Gamma_{\rm in}(m_{-}(\emp))$}. {We start by setting the root unexplored, and start the exploration process. Recursively over the elements in the set of unexplored nodes, we perform the breadth-first neighbourhood exploration, starting from the smallest vertex (in the lexicographic ordering of the Ulam-Harris set) in the set of currently unexplored vertices.}
\endgroup

	We call the resulting tree the {\em random P\'olya point tree with parameters $M$ and $\delta$,} and denote it by $\RPPT(M,\delta)$. Occasionally we drop $M$ and $\delta$ while mentioning $\RPPT(M,\delta)$. When $M$ is degenerate {and equal to $m$ a.s.,} we {call this the P\'olya point tree with parameters $m$ and  $\delta$.} This coincides with the definition in \cite{BergerBorgs}. 

{Figure~\ref{fig:old-neighbor} visualises the {structure of the} \(\Old\)-labelled children of a node \((a_{\omega},t_{\omega})\) {with $t_{\omega}\in \{\Young,\Old\}$}. The \((a_{\omega i})_{i=1}^{m_{-}(\omega)}\)'s do not need be ordered. {To simplify the visualisation}, we have portrayed it as ordered.}

\begin{figure}[h!]
\centering
\begin{tikzpicture}[line cap=round,line join=round,>=triangle 45,x=1cm,y=1cm]
\clip(-12.2,-3.5) rectangle (3,1.8);
\draw [line width=2pt] (-11.92,-1.26)-- (-4.78795,-1.26);
\draw [color=ttttff](-4.8,-0.7) node[anchor=north west] {$\mathbf{a_{\omega}}$};
\draw [color=ttttff](-11.4,-0.4755091198217508) node[anchor=north west] {$\mathbf{U_1^{1/\chi}a_{\omega}}$};
\draw [color=ttttff](-9.8,-0.4755091198217508) node[anchor=north west] {$\mathbf{U_2^{1/\chi}a_{\omega}}$};
\draw [color=ttttff](-7.3,-0.40565458262796117) node[anchor=north west] {$\mathbf{U_{m_{-}(\omega)}^{1/\chi}a_{\omega}}$};
\draw [color=ffzzqq](-10.9,-1.3137635661472267) node[anchor=north west] {$\mathbf{a_{\omega 1}}$};
\draw [color=ffzzqq](-9.45,-1.3137635661472267) node[anchor=north west] {$\mathbf{a_{\omega 2}}$};
\draw [color=ffzzqq](-7.1,-1.3137635661472267) node[anchor=north west] {$\mathbf{a_{\omega m_{-}(\omega)}}$};
\draw [line width=2pt] (0.3638270641432047,0.7644089153680154)-- (-3.48564,-2.90295);
\draw [line width=2pt] (0.3638270641432047,0.7644089153680154)-- (-2.150936274833223,-2.868027018709046);
\draw [line width=2pt] (0.3638270641432047,0.7644089153680154)-- (-0.22994,-2.9);
\draw [color=ffzzqq](-4.4,-2.8505633844105986) node[anchor=north west] {$\mathbf{(a_{\omega 1}, \mathbf{O})}$};
\draw [color=ffzzqq](-2.8,-2.815636115813704) node[anchor=north west] {$\mathbf{(a_{\omega 2},\mathbf{O})}$};
\draw [color=ffzzqq](-0.9983364111356936,-2.8330997501121513) node[anchor=north west] {$\mathbf{(a_{\omega m_{-}(\omega)},\mathbf{O})}$};
\draw [color=ttttff](-0.1251546962133229,1.4804179216043591) node[anchor=north west] {$\mathbf{(a_{\omega},t_{\omega})}$};
\begin{scriptsize}
\draw [fill=ududff] (-11.92,-1.26) circle (2.5pt);
\draw[color=ududff] (-12.06154873920213,-0.9295636115813835) node {0};
\draw [fill=ududff] (-4.78795,-1.26) circle (2.5pt);
\draw [fill=xdxdff] (-10.518625757153531,-1.26) circle (2.5pt);
\draw [fill=xdxdff] (-9.092182521765015,-1.26) circle (2.5pt);
\draw [fill=xdxdff] (-6.5146346270281805,-1.26) circle (2.5pt);
\draw [fill=ududff] (0.3638270641432047,0.7644089153680154) circle (2.5pt);
\draw [fill=ffzzqq] (-3.48564,-2.90295) circle (2.5pt);
\draw [fill=ffzzqq] (-2.150936274833223,-2.868027018709046) circle (2.5pt);
\draw [fill=ffzzqq] (-0.22994,-2.9) circle (2.5pt);
\end{scriptsize}
\end{tikzpicture}
\caption{Creation of \(\Old\)-labelled children of \((a_{\omega},t_{\omega})\)}
\label{fig:old-neighbor}
\end{figure}
	
{Figure~\ref{fig:young-neighbor} visualises the {structure of the} \(\Young\)-labelled children of a node \((a_{\omega},t_{\omega})\) {with $t_{\omega}\in \{\Young,\Old\}$}. In this case the \(\big(a_{\omega (m_{-}(\omega)+i)}\big)_{i=1}^{d_{\omega}^{\rm{\sss (in)}}}\)'s are ordered, as they are the arrival times of a Cox process, i.e., a Poisson process with random intensity.}

\begin{figure}[h!]
\centering
\begin{tikzpicture}[line cap=round,line join=round,>=triangle 45,x=1cm,y=1cm]
\clip(-12.5,-7) rectangle (1,0.5);
\draw [line width=2pt] (-11.92,-1.26)-- (-1.51217,-1.26);
\draw [color=ttttff](-12.148202333633952,-0.5971958131363562) node[anchor=north west] {$\mathbf{a_{\omega}}$};
\draw [color=qqzzqq](-11,-0.7124915066857319) node[anchor=north west] {$\mathbf{a_{\omega (m_{-}(\omega)+1)}}$};
\draw [color=qqzzqq](-8.75,-0.6980795449920599) node[anchor=north west] {$\mathbf{a_{\omega (m_{-}(\omega)+ 2)}}$};
\draw [color=qqzzqq](-5.5,-0.5971958131363562) node[anchor=north west] {$\mathbf{a_{\omega \big(m_{-}(\omega)+d_{\varnothing}^{\mathrm{(in)}}\big)}}$};
\draw [line width=2pt] (-11.384368363869331,-3.2778206881593417)-- (-9.438753535223615,-5.511674750678496);
\draw [line width=2pt] (-11.384368363869331,-3.2778206881593417)-- (-7.73814,-5.5);
\draw [line width=2pt] (-11.384368363869331,-3.2778206881593417)-- (-1.93012,-5.5);
\draw [color=qqzzqq](-11.5,-5.497262788984824) node[anchor=north west] {$\mathbf{(a_{\omega (m_{-}(\omega)+1)}, \mathbf{Y})}$};
\draw [color=qqzzqq](-8.75,-5.4540269039038085) node[anchor=north west] {$\mathbf{(a_{\omega (m_{-}(\omega)+2)},\mathbf{Y})}$};
\draw [color=qqzzqq](-4,-5.396379057129121) node[anchor=north west] {$\mathbf{(a_{\omega (m_{-}(\omega)+d_{\varnothing}^{\mathrm{(in)}})},\mathbf{Y})}$};
\draw [color=ttttff](-12.2,-2.5) node[anchor=north west] {$\mathbf{(a_{\omega},t_{\omega})}$};
\draw [color=ttttff](-1.75,-0.6) node[anchor=north west] {$\mathbf{1}$};
\begin{scriptsize}
\draw [fill=ududff] (-11.92,-1.26) circle (2.5pt);
\draw [fill=ududff] (-1.51217,-1.26) circle (2.5pt);
\draw [fill=xdxdff] (-9.874968783740332,-1.26) circle (2.5pt);
\draw [fill=xdxdff] (-7.793354283200702,-1.26) circle (2.5pt);
\draw [fill=xdxdff] (-4.031928002498959,-1.26) circle (2.5pt);
\draw [fill=ududff] (-11.384368363869331,-3.2778206881593417) circle (2.5pt);
\draw [fill=qqzzqq] (-9.438753535223615,-5.511674750678496) circle (2.5pt);
\draw [fill=qqzzqq] (-7.73814,-5.5) circle (2.5pt);
\draw [fill=qqzzqq] (-1.93012,-5.5) circle (2.5pt);
\end{scriptsize}
\end{tikzpicture}
\caption{Creation of \(\Young\)-labelled children of \((a_{\omega},t_{\omega})\)}
\label{fig:young-neighbor}
\end{figure}

\subsection{Main result and discussions}
\label{sec-main+overview}

We now have all the ingredients to state the main result of this article:

\begin{Theorem}[Local convergence theorem for PA models]\label{theorem:PArs}
	Let $M$ be an $\N$-valued random variable with finite $p$-th moment for some $p>1$ and $\delta>-\infsupp(M)$. Then {preferential attachment models} {$(A), (B), (D)$ and $(E)$} converge vertex-marked locally {in probability} to the random P\'olya point tree with parameters $M$ and $\delta$. 
\end{Theorem}


\paragraph{ Observations}We make some remarks about the above result:
\begin{enumerate}
	\item Our proof uses {the finiteness} of the $p$-th moment for the proofs of some of the concentration bounds around the mean. It would be interesting to identify the precise necessary condition for the local limit result to hold.
	
	\item Berger et al.\ in \cite{BergerBorgs} assumed the fitness parameter $\delta$ to be non-negative, but here {we} allow for negative $\delta$.
	Note that this accommodates infinite-variance degree distributions used in \cite{jordan06,mori2005maximum,waclaw-sokolov}, and suggested in many applied works, see e.g.\ \cite{dorogovtsev2008,dorogovtsev2000,voitalov2019} and the references therein.

    \item {The \emph{vertex-mark} of the vertex $k$ in a preferential attachment model of size $n$ is $k/n$, the limit of which represents the \emph{age} of nodes in the limiting random P\'{o}lya point tree.}
	
	\item {Berger et al.\ in \cite{BergerBorgs}} have shown that $\PA_n^{\sss ({m},\delta)}(d)$ converges locally {in probability} to {the} P\'olya point tree with parameters $m$ and $\delta$. 
	{Restricting to degenerate distributions,} our result {can} be viewed as an extension of \cite{BergerBorgs} to all preferential attachment models. Moreover our model considers the case where every vertex comes with an i.i.d.\ number of out-edges which has only finite $p$-th moment for some $p>1$ and we have considered any general starting graph $G_0$ of size $2$. If we do not assume that the initial graph is of size $2,$ then it increases the computational complexity, and hence we avoid this complication.
	
	\item We provide a proof in detail for Model (A). {The proof for} the models (B) and (D) is very similar and we {only} indicate the {necessary} changes. The fact that all these models have the same local limit is a sign of \emph{universality}.
	
	\item We prove a \emph{local density} result {(see Section~\ref{sec:local_convergence})}, which is stronger than in our main result, for models (A), (B) and {(D)}, and interesting in its own right.
\end{enumerate}
\smallskip

\paragraph{ Consequences of our main result}
We next discuss some consequences of our main result in Theorem \ref{theorem:PArs}, focusing on degree distributions. Denote
\eqn{
	\label{lambda-def}
	\lambda(x)= \frac{1-x^{1-i}}{x^{1-i}}.
}
It immediately follows from the definition {of local convergence in probability, since $h(G,o)=\one_{\{d^{\sss(G)}_o=k\}}$ is a bounded and continuous function,} that
\eqn{
	\label{degree-convergence}
	\frac{1}{n}\sum\limits_{u}{\one_{\{d_u(n)=k\}}}\overset{\prob}{\to} p_k,
}
where
\eqn{
	\label{asymptotic-degree}
	p_k=\prob(M+Y(M)=k)=L^{\sss (\emp)}(k) k^{-\tau},}
for some slowly varying function $L^{\sss(\emp)}(\cdot)$, with $\tau=\min\{3+\delta/\E[M], \tau_{\sss M}\}>{1},$ $\tau_{\sss M}(\geq p)$ is the power-law exponent of the out-degree distribution (with $\tau_{\sss M}=+\infty$ when $M$ is light-tailed), and $Y(M)$ is a mixed-Poisson random variable with mixing distribution $\Gamma(M)\lambda(U_\emp),$ where $\Gamma(M)$ is a Gamma variable with parameters $M+\delta$ and $1$ and $U_\emp\sim Unif(0,1)$ and $Unif(0,1)$ is a uniform random variable on $(0,1)$.
This was previously established, under related assumptions, for model (E) in \cite{Dei}. We next extend this result to the convergence of older and younger neighbours of a random vertex:

\begin{Corollary}[Asymptotic degree distribution for older and younger neighbours]
	\smallskip\noindent
	\begin{itemize}
		\item[(a)]\label{cor-older-younger-neighbours}
		As $n\rightarrow \infty$, the degree of a uniform older neighbour of a uniform vertex satisfies
		\eqn{\label{eq:deg:old}
			\frac{1}{n}\sum\limits_{\substack{u,v,j:\\u\overset{j}{\rightsquigarrow} v}}\frac{\one_{\{d_v(n)=k\}}}{m_u}\overset{\prob}{\to} \Tilde{p}_k^{\sss({\old})} = \prob\Big( 1+ M^{(\delta)}+Y\big( M^{(\delta)}+1,A_{\old} \big) = k\Big),~\quad\mbox{for }k\geq 1}
		where $Y\big( M^{(\delta)}+1,A_{\old} \big)$ is a mixed Poisson random variable with mixing distribution $\Gamma_{\rm in}\big( M^{(\delta)}+1 \big)\lambda\big( A_{\old} \big)$ and $A_{\old}$ is distributed as $U_\emp U_1^{1/i}$, where $U_\emp, U_1$ are independent $Unif(0,1)$ random variables.\\
		Similarly, as $n\rightarrow \infty$, the degree of a uniform younger neighbour of a uniform vertex satisfies
		\eqn{\label{eq:deg:young}
			\begin{split}
				\frac{1}{\sum_v \one_{\{d_v(n)>m_v\}}}\sum\limits_{\substack{u,v,j:\\u\overset{j}{\rightsquigarrow} v}}&\frac{\one_{\{d_v(n)>m_v\}}\one_{\{d_u(n)=k\}}}{d_v(n)-m_v}\\
				&\overset{\prob}{\to} \Tilde{p}_k^{\sss({\young})} =\prob\Big( M^{(0)}+Y\big( M^{(0)},A_{\young} \big) =k \Big),~\quad\mbox{for }k\geq 1
		\end{split}}
		where $Y\big( M^{(0)},A_{\young} \big)$ is a {mixed} Poisson random variable with mixing distribution $\Gamma_{\rm in}\big( M^{(0)} \big)\lambda\big( A_{\young} \big)$ and $A_{\young}$ has conditional density
		\eqn{\label{eq:density:A} f_{U_\emp}(x)=\frac{\rho_\emp(x)}{\int_{U_\emp}^1 \rho_\emp(s)\,ds}=\frac{x^{-i}}{\int_{U_\emp}^1 s^{-i}\,ds}~,}
		where $U_\emp\sim {\sf Unif}(0,1)$ and $\rho_\emp(s)$ is as defined in \eqref{for:pointgraph:poisson}, {and we take the expectation w.r.t.\  \(U_{\emp}\) to obtain the density function \(f_{\young}\) of \(A_{\young}\).}
		\smallskip\noindent
		\item[(b)]\label{cor-older-younger-neighbours-tail} There exist slowly varying functions $L^{\sss({\old})}(k)$ and $L^{\sss({\young})}(k)$, such that
		\eqn{\label{tail:exponent}
			\Tilde{p}_k^{\sss({\old})}\sim L^{\sss({\old})}(k)k^{-\tau_{\sss({\old})}},
			\qquad
			\Tilde{p}_k^{\sss({\young})}\sim \Theta(L^{\sss({\young})}(k)) k^{-\tau_{\sss({\young})}} \text{ as $k\to\infty$},
		}
		where $\tau_{\sss({\old})}=\min\{2+\delta/\E[M], \tau_{\sss M}-1\}$ and $\tau_{\sss({\young})}=\min\{ 4+\delta/\E[M], \tau_{\sss M}-1 \}$.
	\end{itemize}
\end{Corollary}
The convergence is a direct consequence of Theorem \ref{theorem:PArs}, so what is left is the identification of the limiting probabilities for the $\RPPT$. 
Thus, the degree distribution of random neighbours is asymptotically {\em size-biased} compared to the original degree distribution in the graph, as in \eqref{degree-convergence}--\eqref{asymptotic-degree}. The proof of the corollary is postponed to Section~\ref{sec-proof-cor-older-younger}.
This result is somewhat surprising in the sense that in \eqref{tail:exponent}, both the tail exponents $\tau_{({\old})}$ and $\tau_{({\young})}$ contains the $\tau_M-1$, but the dependence on the $\PAM$ power-law exponent is either one larger or one smaller than for the degree distribution of the root.
{
\paragraph{Idea of proof of Theorem~\ref{theorem:PArs}}
	
	For any vertex-marked finite graph $\left(\tree,(a_\omega)_{\omega\in V(\tree)}\right)$ and $r\in \N,$ define 
	\eqan{
	N_{r,n}\left(\tree,(a_\omega)_{\omega\in V(\tree)}\right)=&\sum\limits_{v\in[n]}\one_{\left\{ B_r^{\sss(G_n)}(v)~\simeq~ \tree, ~|v_\omega/n-a_\omega|\leq 1/r~\forall\omega~\in~V(t) \right\}}\nn\\
    =&{\sum\limits_{v\in[n]}\one_{\big\{ d_{\mathcal{G}_\star}\big( \big(B_r^{\sss(G_n)}(v),v,u\mapsto u/n\big),\big(\tree,\emp,(a_\omega)_{\omega\in V(\tree)}\big)\big)\leq 1/r \big\}},}
    }
	where $G_n$ is taken as $\PA_n$ and $v_\omega$ is the vertex in $G_n$ corresponding to $\omega \in V(\tree)$. Then, {to prove} Theorem~\ref{theorem:PArs}, {by} Definition~\ref{def:vertex:marked:local:convergence} it is enough to show that ${N_{r,n}\left(\tree,(a_\omega)_{\omega\in V(\tree)}\right)/n}$ converges in probability to $\mu\left( B_r^{\sss(G)}(\emp)\simeq~\tree, ~|A_\omega-a_\omega|\leq 1/r~\forall\omega~\in~V(t) \right)$, where $\mu$ is the law of the limiting $\RPPT$ graph, and $A_\omega$ is the mark of the vertex in $\RPPT$ corresponding to $\omega\in V(\tree)$ and $\tree$ is a tree. We aim to prove this convergence using a second moment method, i.e., {we will prove that}
	\eqan{\label{eq:second:moment:explain}
			\frac{1}{n}\E\left[N_{r,n}\left(\tree,(a_\omega)_{\omega\in V(\tree)}\right) \right] \to& ~\mu\left( B_r^{\sss(G)}(\emp)\simeq~\tree, ~|A_\omega-a_\omega|\leq 1/{(r+1)}~\forall\omega~\in~V(t) \right),\\
			\text{and}~~\frac{1}{n^2}\E\left[ N_{r,n}\left(\tree,(a_\omega)_{\omega\in V(\tree)}\right)^2\right] \to&~ \mu\left( B_r^{\sss(G)}(\emp)\simeq~\tree, ~|A_\omega-a_\omega|\leq 1/{(r+1)}~\forall\omega~\in~V(t) \right)^2.\nn}
	For proving the first, we show that the {\em joint density} of the {ages of the vertices in the} $r$-neighbourhood of a uniformly chosen vertex in $\PAM$ converges to that of $\RPPT$. {Calculating this joint density explicitly is somewhat} {involved} because of {the} dependence structure in the edge-connection probabilities of $\PAM$ models. In \cite{BBCS05}, the authors provide a P\'olya urn representation of model (D), {which made the proof in \cite{BergerBorgs} possible} since this representation {implies that the edges are {\em conditionally independent}}. {An essential step in our proof is thus to} construct similar P\'olya urn descriptions for models (A), (B) and (D) and show that the models are equal in distribution to the corresponding P\'olya urn representations. S\'{e}nizergues \cite{Delphin} proved a similar result to the first part of the proof for preferential attachment trees and extended the result for preferential attachment graph with random out-degree. We tend to use a P\'olya urn description with a different set of $\Beta$ random variables. {We do this} by explicitly calculating the graph probabilities. Our result works for i.i.d.\ random $\delta$ also. With this distributional equivalence in hand, we {can now compute the above joint density,} and show that {it} converges to that of the $\RPPT$.
	
	{For the second moment,} we first expand the {square of the sum arising in the} numerator. From the expansion we observe that along with some vanishing terms, we obtain the joint density of the $r$-neighbourhoods of {\em two} uniformly chosen vertices. Next we prove that the $r$-neighbourhoods of two uniformly chosen vertices are disjoint with high probability (whp). Again, for the joint density calculation, we use the P\'olya urn description of {our} models. Since the edge-connection {events} are conditionally independent {by the P\'olya urn representations}, the two neighbourhoods are {\em conditionally independent} when they are disjoint. Therefore the joint density {factorizes} and we obtain the required result.
	
	Though our main steps for proving the theorem {are the} same as {those in} \cite{BergerBorgs}, our proof techniques differ {significantly}. For example, we avoid the induction argument in the neighbourhood size. We use a coupling of the preferential attachment model with the P\'olya urn graphs through an explicit density computation. To summarize, we have two crucial steps in proving the main theorem:
	\begin{itemize}
		\item[(a)] \textbf{Equivalence: }there exists P\'olya urn representations for models (A), (B) and (D);
		\item[(b)] \textbf{Convergence: }the {joint} density of the {ages of the vertices in the} $r$-neighbourhood of the P\'olya urn graphs converges to that of the RPPT.
	\end{itemize}
	
	For model (E), the edges are connected independently and the degrees of the older vertices are updated only after the new vertex is included in the graph with {\em all} its out-edges. This edge-connection procedure is different from other models and we do not have a P\'olya urn representation for this graph. We first prove the convergence of {the models having a P\'olya urn representation.} Then we couple {models} (D) and (E), such that {the} probabilities of observing the $r$-neighbourhoods of a uniformly chosen vertex in the graphs are {asymptotically equal.} Since for model (D), we have proved the local convergence, convergence for model (E) follows immediately. {Note, however, that this does {\em not} imply the convergence of the joint density of the ages of the vertices in the $r$-neighbourhood for model (E).} }
\smallskip\noindent
\paragraph{Relation to the literature} P\'olya urn representations of the preferential attachment models were previously studied by Berger et al.\ \cite{BBCS05} and S\'{e}nizergues \cite{Delphin}. Local convergence of models (D) and (E) {was first proved} by Berger et al.\ \cite{BergerBorgs} for fixed out-degrees and non-negative $\delta$. Results on local convergence of related PAMs have been {established} by Y.\ Y.\ Lo \cite{tiffany2021}, who analysed the local limit of the preferential attachment trees but i.i.d.\ random fitness parameter $\delta$. Rudas, T\'oth and Valk\'o \cite{rudas-toth-valko} proved local convergence almost surely for general preferential attachment trees, based on a continuous-time embedding, which gives a continuous-time branching process {as studied} by Jagers and Nerman \cite{JN84} in full generality.

After uploading our paper to arXiv, we were informed about the work of Banerjee, Deka and Olvera-Cravioto \cite{BDC23} which also proves local convergence in probability, {but instead of the in-components} of general directed preferential attachment models with i.i.d.\ out-degrees, {where the edges are directed from young to old}. The paper assumes {a} finite first moment of the random out-degree and does not rely on the Pólya urn representation of the preferential attachment model, but rather on an extension of the work of the first and third authors that relates preferential attachment models to continuous-time branching processes \cite{GarvdH2}. Further, in addition to the local convergence, our paper establishes density convergence of the preferential attachment model, and identified the degree distribution of the nodes of the limiting tree. Although the proof techniques are different, both techniques are interesting in their own right.
\smallskip

\paragraph{Open problems} {While we have established the local limits for models (A), (B), (D), and (E), we conjecture that the same limit also holds for model (F), although our current proof techniques are not sufficient to verify this rigorously. In addition, we believe that our results can be an essential ingredient to the proof of various related properties of PAMs, such as the behaviour of percolation on, and graph distances in, them substantially extending works of \cite{DuGongLiZhu25,HazvdHofRay23} on fixed-out-degree preferential attachment models.} It would further be interesting to extend the work to random fitness distributions, as well as to the model with conditionally independent edges as studied by Dereich and M\"orters \cite{Der2009, Der2011, Der2013}.
\smallskip

\paragraph{Structure of the rest of the article}
In Section~\ref{sec:affineurn}, we {extend} the definition of P\'olya urn graphs to collapsed P\'olya urn graphs (for equivalence with models (A) and model (B)). In Section~\ref{sec:equivalence} we {prove} that the preferential attachment models (A), (B) and (D) are equal in distribution {to their} respective P\'olya urn representations. Section~\ref{sec:prelim:result} deals with some results that are necessary to show the $r$-neighbourhood density convergence.
In Section~\ref{sec:local_convergence} we show the convergence part of the proof {for model (A). Using the similarity of the models in Section~\ref{sec:equivalence}, the proofs for the models (B) and (D) follow in a same way}. In Section~\ref{sec:model(F):coupling}, we {show} the coupling between models (D) and (E). Section~\ref{sec-proof-cor-older-younger} provides a proof of Corollary~\ref{cor-older-younger-neighbours}. \arxiversion{The details of the proofs of several results are deferred to Appendix \ref{sec:appendix}, and are minor adaptations of the corresponding proofs for degenerate  out-degrees. In Appendix~\ref{app:preliminary:result}, we provide the proof of equivalence of models (B) and (D) and their respective P\'olya urn representations. These proofs follow the same line as those in Section~\ref{sec:equivalence} for the equivalence of model (A) and its P\'olya urn representation. 
The proof of Corollary~\ref{cor-older-younger-neighbours}(b) is deferred to Appendix~\ref{app:degree} due to similar proof techniques used in the literature.}\journalversion{{Here we have used several results for degenerate out-degrees adapted to our settings of i.i.d out-degrees, proved in \cite[{Appendix A}]{GHvdHR22}.}}

\section{P\'olya urn graph description}
\label{sec:affineurn}
The first and foremost difficulty in dealing with preferential attachment models {is} that the edge-connection {events} are not independent. {We now give a representation of our PAMs where these events are \emph{conditionally independent}, extending the results
	of Berger et al.\ in \cite{BBCS05} that apply to $\PA_n^{\sss ({m},\delta)}(d)$.} Intuitively, every new edge-connection in the preferential attachment model with intermediate degree updates can be viewed as drawing a ball uniformly from a P\'olya urn with balls {having multiple colours corresponding to the various vertices in the graph.} This is the place where the intermediate degree update comes as a blessing in disguise to us. If the degrees are not updated {after every edge-connection step} ({such as in} models (E) and (F)), {then} the model may seem simpler, but {in fact it is} harder to work with since we do not get a P\'olya urn description for it. We handle {model~(E)} in Section~\ref{sec:model(F):coupling} {using a {\em coupling approach} instead.}

\subsection{P\'olya urn graph}
The sequential preferential attachment models {(A), (B) and (D)} in Section~\ref{sec:model} can be interpreted as an experiment with $n$ urns {corresponding to the vertices in the graph}, where the number of balls in each urn represents the degree of the corresponding vertex in the graph.  First, we introduce a new class of random graphs {that we later prove (in Section~\ref{sec:equivalence}) to have the same distribution as our PAMs:}

\begin{Definition}[P\'olya urn graphs]
	\label{def:PU}
	{\rm Define the following {constructs:}
		\begin{itemize}
			\item[{\btr}] {Let $\left(U_{k} \right)_{k\geq 1}$ be} a sequence of i.i.d.\ uniform random variables in $[0,1]$;
			\item[{\btr}] {Given $\boldsymbol{m}=(1,1,m_3,m_4,\ldots)$}, {let $\boldsymbol{\psi}=( \psi_k )_{k\in\N}$ be} a sequence of conditionally independent $\Beta$ random variables with support in $[0,1]$, such that $\prob(\psi_k=1)=0$ for all $k\geq 2$ and $\psi_1 = 1$ point-wise;
			\item[{\btr}] Let $G_0$ be the initial graph of size $2$ and $n\in\N$ be the size of the graph.
		\end{itemize}
		Define $\mathcal{S}_0^{\sss(n)}=0$ and $\mathcal{S}_n^{\sss(n)}=1$ and, for $k\in [n-1]$, define
		\eqn{\begin{split}
				\mathcal{S}_k^{\sss(n)} = (1-\psi)_{(k,n]},\qquad\text{and}\qquad
				\mathcal{I}_k^{\sss(n)} = \left[\mathcal{S}_{k-1}^{\sss(n)}, \mathcal{S}_{k}^{\sss(n)}\right),
		\end{split}}
		where, for $A\subseteq[n]$,
		\eqn{(1-\psi)_A = \prod\limits_{a\in A}(1-\psi_a).}
		Then, {$\PU^{\sss\rm{(SL)}}_n(\boldsymbol{m},\boldsymbol{\psi})$ and $\PU^{\sss\rm{(NSL)}}_n(\boldsymbol{m},\boldsymbol{\psi})$, the {\em P\'olya urn graph} of size $n$ with and without self-loops, respectively, are} defined as follows: 
		\begin{itemize}
			\item[$\rhd$]
			{{In} $\PU^{\sss\rm{(NSL)}}_n(\boldsymbol{m},\boldsymbol{\psi})$, the} \(j\)-th edge from vertex $k$ is attached to vertex $u\in[k]$ {precisely when}
			\eqn{
				\label{for:PU:NSL}
				U_{m_{[k-1]}+j}\mathcal{S}^{\sss(n)}_{k-1}\in \mathcal{I}_u^{\sss (n)}~,
			}
            where recall $m_{[k]}=\sum_{i\leq k}m_i~$.
			Observe that self-loops are absent here since  $\left(0,\mathcal{S}^{\sss(n)}_{k-1}\right)$ and $\mathcal{I}_k^{\sss(n)}$ are two disjoint sets{;}
			\item[$\rhd$] for $\PU^{\sss\rm{(SL)}}_n(\boldsymbol{m},\boldsymbol{\psi})$,
			the condition \eqref{for:PU:NSL} is replaced by
			\eqn{
				\label{for:PU:SL}
				U_{m_{[k-1]}+j}\mathcal{S}^{\sss(n)}_{k}\in \mathcal{I}_u^{\sss (n)}.}
		\end{itemize}
		\vspace{-0.65cm}\hfill$\blacksquare$
	}
\end{Definition}

\
To specify a P\'olya urn graph $\PU_n(\boldsymbol{m,\psi})$, we need to specify the out-edge distribution $M$, and the parameters of the $\Beta$ variables $\boldsymbol{\psi}$.
{Berger et al.\ in \cite{BergerBorgs}} have shown that $\PA_n^{\sss ({m},\delta)}{{(d)}}$ is equal in distribution to $\PU^{\sss\rm{(NSL)}}_n(\boldsymbol{m}^{(1)},\boldsymbol{\psi})$, where $\boldsymbol{\psi}$ is taken as a sequence of {independent} $\Beta$ random variables with certain parameters and $M$ is degenerate at $m$, i.e.\ $\boldsymbol{m}^{(1)} = (1,1,m,m,\ldots)$. Since $\PAri_n(\boldsymbol{m},\delta)$ is a generalized version of $\PA^{{(m,\delta)}}_n(d)$ with i.i.d.\ random out-{degrees}, {we show that }{this model} also has a P\'olya urn description. We sketch an outline in Section~\ref{sec:equivalence} {[{see} Theorem~\ref{thm:equiv:PU:PAri}]} that conditionally on $\boldsymbol{m}$, $\PAri_n(\boldsymbol{m},\delta)$ and $\PU^{\sss\rm{(NSL)}}_n(\boldsymbol{m},\boldsymbol{\psi})$ are also equal in distribution. 
On the other hand, it is evident that {the} deterministic versions of model (B) and (D) are equivalent for $m=1$. From the edge-connection probabilities of model (A) and (B), we can observe that they are different only in {whether they give rise to self-loops or not}. {van der Hofstad in \cite[Chapter~5]{vdH2} shows} that model (A) and (B) can be obtained by a collapsing procedure for degenerate $M$. Therefore, for obtaining a P\'olya urn graph equivalence for models (A) and (B), a generalisation of this collapsing procedure is helpful, {and we continue by describing such collapsing procedures.}

\subsection{Collapsing operator}\label{subsec:collapsing}
Here we {generalize} the collapsing {procedure} discussed in \cite{vdH2}. Let $\boldsymbol{r} = (r_1,r_2,\ldots)\in\N^\N$ and {let} $\mathcal{H}$ be the set of all finite vertex-labelled {graphs}. Then the collapsing operator $\mathcal{C}_{\boldsymbol{r}}$, {acting} on $\mathcal{H}$ {and collapsing groups of vertices of size $r_i$ into vertex $v_i$,} is defined as follows:
\begin{itemize}
	\item[{\btr}] Let $G\in\mathcal{H}$ be a vertex-labelled graph of size $n$ and $n\in \left(r_{[k-1]}, r_{[k]}\right]$ for some $k\in\N$, where 
	\[
	r_{[l]}=\sum\limits_{i\leq l} r_i~;
	\]
	\item[{\btr}] {g}roup the vertices $\left\{ r_{[i-1]}+1,\ldots, r_{[i]} \right\}$ and name the groups $v_i$ for all $i< k$, while the group $v_k$ contains the vertices $\left\{r_{[k-1]}+1, \ldots , n\right\}$;
	\item[{\btr}] {c}ollapse the vertices of each of these $v_k$ groups {into one vertex,} and name the new vertex $k$;
	\item[{\btr}] {e}dges originating and ending in the same group form self-loops in the new graph;
	\item[{\btr}] {e}dges between two different groups form edges between the respective vertices in the new graph.
\end{itemize}
Note that if we fix $r_1=r_2=1$ and $r_i=m$ for all $i\geq 3$ and suitable $G_0$, then we get back the collapsing procedure used first in \cite{BolRio04b}, and further discussed in \cite{vdH2}. 
We now discuss the construction of models (A) and (B) through this collapsing operator, {conditionally on the i.i.d.\ out-degrees described by $\boldsymbol{m}$.}
\smallskip
\paragraph{{PAM construction by collapsing}}

{We start with a vertex-labelled graph $G_0$ of size $2$ and degrees $a_1$ and $a_2$, respectively. First, we explain the construction of model (A) using collapsing.
	
	Every $v\geq 3$ comes with exactly one edge. Given $\boldsymbol{m}=(1,1,m_3,\ldots)$, the incoming vertex $v=m_{[i-1]}+j$ for some $i\in[3,n)$ and $j\leq m_i$, connects to one of $u\in [v]$ with probability
	\eqn{\label{eq:edge_connecting:model(rs):1}
		\prob_m\left( v\rightsquigarrow u\mid \PArs_{v-1}(\boldsymbol{m},1,\delta) \right) = \begin{cases}
			\frac{d_u(v-1)+\delta(u)}{c_{v,j}}&\mbox{ when $v>u$,}\\
			\frac{1+\delta(u)}{c_{v,j}}&\mbox{ when $v=u$}.
	\end{cases}}
	Here $\delta(u) = {\delta/m_k}$ when $u\in\left( m_{[k-1]},m_{[k]} \right],$
	\[
	c_{{v},j} = a_{[2]}+2\left( m_{[{v}-1]}+j-2 \right) - 1+ ({v}-1)\delta + \frac{j}{m_{{v}}}\delta~{,}
	\]
	{and }$\PArs_{v-1}(\boldsymbol{m},1,\delta)$ denotes the graph formed after the $(v-1)$-st vertex is added, with $d_u(v-1)$ denoting the degree of the vertex $u$ in $\PArs_{v-1}(\boldsymbol{m},1,\delta)$. Continue the process until the $m_{[n]}$th vertex is added. 
	We obtain $\PArs_n(\boldsymbol{m},\delta)$ by {applying} $\mathcal{C}_{\boldsymbol{m}}$ on $\PArs_{m_{[n]}}(\boldsymbol{m},1,\delta)$, {where \(\mathcal{C}_{\boldsymbol{m}}\) is as defined in Section~\ref{subsec:collapsing}}.
	Therefore, for model (A), the conditional edge-connection probabilities are given by \eqref{eq:edge_connecting:model(rs):2}, {as required.}
	
	{To construct $\PArt_n(\boldsymbol{m},\delta)$ by a collapsing procedure}, we do not allow {for} self-loops for ${\PArt_n(\boldsymbol{m},1,\delta)}$ as we did in \eqref{eq:edge_connecting:model(rs):1}.
	The rest of the process remains the same. Starting from the same initial graph $G_0$, conditionally on $\boldsymbol{m}$, every $v\geq 3$ comes with exactly one edge. The incoming vertex $v=m_{[i-1]}+j$ for some $i\in[3,n]$ and $j\leq m_i$, connects to $u\in [v-1]$ {with the same probability as in \eqref{eq:edge_connecting:model(rs):1}. Since here we do not allow for any self-loop, the normalising constant in the denominator now is} 
	\[
	c_{{v},j} = a_{[2]}+2\left( m_{[{v}-1]}+j-3 \right) +({v}-1)\delta + \frac{(j-1)}{m_{{v}}}\delta~{,}
	\]
	{and} $\PArt_{v}(\boldsymbol{m},1,\delta)$ denotes the graph formed after the $v$th vertex is added, with $d_u(v)$ denoting the degree of the vertex $u$ in $\PArt_{v}(\boldsymbol{m},1,\delta)$. Continue the process until the $m_{[n]}$th vertex is added. We obtain $\PArt_n(\boldsymbol{m},\delta)$ by applying $\mathcal{C}_{\boldsymbol{m}}$ on $\PArt_{m_{[n]}}(\boldsymbol{m},1,\delta)$.
	Therefore, for model (B), the edge-connection probabilities are given by \eqref{eq:edge_connecting:model(rt):2}. 
	\begin{remark}[Initial graph is preserved in collapsing]
		{\rm Observe that we always choose $r_1=r_2=1$ while performing the collapsing operator on the pre-collapsed preferential attachment graphs. This is done intentionally to preserve the structure of the initial graph $G_0$. If we start with an initial graph of size $\ell\geq 1$, then we can choose $r_1=r_2=\ldots=r_\ell=1$ to preserve the initial graph structure in both collapsed and pre-collapsed preferential attachment graphs.}
	\end{remark}
	We will use this collapsing operator to introduce an extension of the P\'olya urn graph, {which we call} the collapsed P\'olya urn graph.}

\subsection{Collapsed P\'{o}lya urn graphs}

$\PArt_n(\boldsymbol{m},1,\delta)$ is essentially the same as model (D) when every vertex comes with exactly one out-edge, but $\delta$ is different for every vertex. {Therefore,} we expect that {the} P\'olya urn graph {extends to this graph}. {Similarly to the construction of model (B) through a collapsing procedure}, we collapse the P\'olya urn graph.

{Conditionally on $\boldsymbol{m},$ we construct the collapsed P\'olya {urn} graph by using our collapsing construction on the {P\'olya urn graph} defined in Definition~\ref{def:PU} as follows:}

\begin{Definition}[Collapsed Pólya urn graph]
	\label{def:CPU}
	{\rm We first construct $\PU_{m_{[n]}}(\boldsymbol{1},\boldsymbol{\psi})$ with every new vertex having exactly one out-edge and initial graph $G_0$ of size $2$. {Conditionally} on $\boldsymbol{m}=(1,1,m_3,m_4,\ldots),$ the graph $\CPU_n(\boldsymbol{m},\boldsymbol{\psi})$ is defined as $\mathcal{C}_{\boldsymbol{m}}\big( \PU_{m_{[n]}}(\boldsymbol{1},\boldsymbol{\psi}) \big)$. The label SL or NSL for the $\CPU_n$ will be determined by the label of the $\PU_{m_{[n]}}$. We denote the two $\CPU$'s with and without self-loops as $\CPU_n^{\sss\rm{(SL)}}$ and $\CPU_n^{\sss\rm{(NSL)}}$ respectively.}\hfill$\blacksquare$
\end{Definition}

\begin{remark}[{Self-loops for $\CPU_n^{\sss (\mathrm{NSL})}$}]
	{\rm $\CPU_n^{\sss (\mathrm{NSL})}$ may contain self-loops because of {the collapsing procedure.}}\hfill$\blacksquare$
\end{remark}
We end this section by deriving the connection probabilities for $\CPU_n^{\rm{\sss{(SL)}}}$ and $\CPU_n^{\rm{\sss{(NSL)}}}$.
For $k\geq 1$ and $j\in [m_{k}]$, define
\begin{align}\label{eq:def:position:CPU}
	\mathcal{S}_{k,j}^{\sss(n)} = \prod\limits_{l=m_{[k-1]}+j+1}^{m_{[n]}} (1-\psi_l),\qquad\text{and}\qquad
	\mathcal{S}_k^{\sss(n)} =  \mathcal{S}_{k,m_k}^{\sss(n)};
\end{align}
and the intervals $\mathcal{I}_k^{\sss(n)}=\left[\mathcal{S}_{k-1}^{\sss(n)}, \mathcal{S}_{k}^{\sss(n)}\right)$. {Let $\prob_m^{\sss\rm{(SL)}}$ and $\prob_m^{\sss\rm{(NSL)}}$ denote the conditional law given $\boldsymbol{m}$ of $\CPU_n^{\sss\rm{(SL)}}$ and $\CPU_n^{\sss\rm{(NSL)}}$, respectively.} Then, from the construction of the $\CPU_n^{\sss\rm{(SL)}}$, it follows that, for $u\geq 3$,
\eqn{\label{eq:edge-connecting_probability:CPU:SL}
	\prob_m^{\sss\rm{(SL)}}\left( \left.u\overset{j}{\rightsquigarrow}v\right|\left( \psi_k \right)_{k\in\N} \right) = \begin{cases}
		\frac{\mathcal{S}_{v}^{\sss(n)}-\mathcal{S}_{v-1}^{\sss(n)}}{\mathcal{S}_{u,j}^{\sss(n)}}& \text{for }u>v,\\
		\frac{\mathcal{S}_{u,j}^{\sss(n)}-\mathcal{S}_{u-1}^{\sss(n)}}{\mathcal{S}_{u,j}^{\sss(n)}}& \text{for }u=v.
\end{cases}}
This probability is for {$\CPU_n^{\sss\rm{(SL)}}$. For $\CPU_n^{\sss\rm{(NSL)}}$ instead}, the expression in \eqref{eq:edge-connecting_probability:CPU:SL} becomes
\eqn{\label{eq:edge-connecting_probability:CPU:NSL}
	\prob_m^{\sss\rm{(NSL)}}\left( \left.u\overset{j}{\rightsquigarrow}v\right|\left( \psi_k \right)_{k\in\N} \right) = \begin{cases}
		\frac{\mathcal{S}_{v}^{\sss(n)}-\mathcal{S}_{v-1}^{\sss(n)}}{\mathcal{S}_{u,j-1}^{\sss(n)}}& \text{for }u>v,\\
		\frac{\mathcal{S}_{u,j-1}^{\sss(n)}-\mathcal{S}_{u-1}^{\sss(n)}}{\mathcal{S}_{u,j-1}^{\sss(n)}}& \text{for }u=v.
\end{cases}}
{Now that we have introduced the relevant random graph models, in the next section, we will show that the collapsed versions of the P\'olya graph models have the same distribution as our preferential attachment models.}
\section{Equivalence of preferential attachment models}
\label{sec:equivalence}
{As explained in Section \ref{sec-main+overview},} there are two key steps in the proof of {the local limit result in} Theorem~\ref{theorem:PArs}. In the first step, we show that the $\PAM$ with random out-{degrees} is equal in distribution to collapsed P\'olya urn graph $(\CPU)$ or P\'olya urn graph $(\PU)$, with the $\psi$ random variables defined appropriately.
{We bring the notion of} $\CPU$ in between since {it implies} the much-appreciated independence structure of the edge-connection events, which is not {a priori} valid in {the} $\PAM$. We provide explicit calculations for model (A), and state the modifications required for other models.

\subsection{{Equivalence of model} (A) and $\CPU^{\sss\rm{(SL)}}$}\label{subsec:equivalence:(A)}
{Recall that model (A) has i.i.d.\ out-degrees for every vertex. Thus, in order to couple it with our $\CPU$, it must have the same out-degrees.} Given $\boldsymbol{m}= (m_1,m_2,m_3,\ldots)$, with $m_1=m_2=1$, we {thus aim} to couple $\PArs_n(\boldsymbol{m},\delta)$ and $\CPU^{\sss\rm{(SL)}}$.

For $i\in[2],$ {$a_i$ denotes the} degree of vertex $i$ in the initial graph $G_0$, {while for $i>2$ define $a_i\equiv 1$.} Let $\mathcal{H}_n$ be the set of all finite vertex-labelled graphs $G$ of size $n$ and {let} $\mathcal{H}_{\boldsymbol{m}}(G)= \left\{ G_e\in {\mathcal{H}}_{m_{[n]}} \colon \mathcal{C}_{\boldsymbol{m}}(G_e) {\overset{\star}{\simeq}} G \right\}$ {denote the set of graphs that are mapped to $G$ by the collapsing operator $\mathcal{C}_{\boldsymbol{m}}$.} 

For $k\geq 3$ and $l\in\left[ m_{k} \right]$, {define}
\eqn{\label{def:psi:1}
	\psi_{m_{[k-1]}+l} \sim {\Beta}\left( 1+\frac{\delta}{m_k},\, a_{[2]} + 2\left( m_{[k-1]}+l-3 \right)+(k-1)\delta+\frac{(l-1)}{m_k}\delta \right),
}
where $a_{[2]}=a_1+a_2$ and, for $k\leq 2$,
\eqn{\label{def:psi:2}
	\begin{split}
		\psi_1 \equiv~1,~
		\mbox{and}\quad \psi_{2} \sim {\Beta}\left( a_2+\delta,a_1+\delta \right).
\end{split}}
{Observe that the second parameters of the $\Beta$ random variables in \eqref{def:psi:1} are equal to $c_{k,l-1}{+}1$ as defined in \eqref{eq:normali}.}
{We abbreviate $\boldsymbol{\psi} = (\psi_i)_{i\geq 1}$ for the collection of $\Beta$ variables, where we emphasize that these variables are {\em conditionally independent} given the random out-degrees $\boldsymbol{m}$. Our main result concerning the relation between collapsed P\'olya graphs and model (A) is as follows:}
\begin{Theorem}[Equivalence of model (A) and $\CPU^{\sss\rm{(SL)}}$]\label{thm:equiv:CPU:PArs}
	For any graph $G\in\mathcal{H}_{n}$,
	\eqn{\label{eq:thm:equiv:PA:CPU}
		\prob_m\left( \PArs_n(\boldsymbol{m},\delta) \overset{\star}{\simeq} G \right) = \prob_m\left( \CPU_{n}^{\sss\mathrm{(SL)}}(\boldsymbol{m},\boldsymbol{\psi}) \overset{\star}{\simeq} G \right).}
\end{Theorem}

{We emphasize that Theorem \ref{thm:equiv:CPU:PArs} describes a {\em conditional} result given $\boldsymbol{m}$, i.e., it is conditionally on $\boldsymbol{m}$.
}

{We prove Theorem \ref{thm:equiv:CPU:PArs} by proving it for the pre-collapsed version of both graphs, and equating their conditional distributions.} The following proposition helps us in equating the conditional probabilities of the pre-collapsed graphs:
\begin{Proposition}[Equivalence of pre-collapsed model (A) and $\PU^{\sss\rm{(SL)}}$]\label{prop:equiv:CPU:PArs}
	Conditionally on $\boldsymbol{m}$, for any graph $H\in\mathcal{H}_{m_{[n]}}$,
	\eqn{\label{prop:eq:equiv:CPU:PArs}
		\prob_m\left( \PArs_n(\boldsymbol{m},1,\delta) \overset{\star}{\simeq} H \right) = \prob_m\left( \PU_{m_{[n]}}^{\sss\mathrm{(SL)}}(\boldsymbol{1},\boldsymbol{\psi}) \overset{\star}{\simeq} H \right).}
\end{Proposition}
\begin{proof}[{Proof of Theorem~\ref{thm:equiv:CPU:PArs} subject to Proposition~\ref{prop:equiv:CPU:PArs}.}]
	Proposition~\ref{prop:equiv:CPU:PArs} essentially provides us with the pre-collapsing equivalence of the graphs. Since, conditionally on $\boldsymbol{m},~ \left\{\PArs_{m_{[n]}}(\boldsymbol{m},1,\delta) \overset{\star}{\simeq} H\right\} $ are disjoint events for $H\in\mathcal{H}_{\boldsymbol{m}}(G)$, the probability on the RHS ({right hand side}) of \eqref{eq:thm:equiv:PA:CPU} for model (A) can be written in terms of the pre-collapsed graphs as
	\eqn{\label{for:equiv:CPU:PArs:1}
		\begin{split}
			\prob_m\left( \PArs_n(\boldsymbol{m},\delta) \overset{\star}{\simeq} G \right) =& \prob_m\left( \bigcup\limits_{H\in\mathcal{H}_{\boldsymbol{m}}(G)}\left\{ \PArs_{m_{[n]}}(\boldsymbol{m},1,\delta) \overset{\star}{\simeq} H \right\} \right)\\
			=& \sum\limits_{H\in\mathcal{H}_{\boldsymbol{m}}(G)}\prob_m\left(  \PArs_{m_{[n]}}(\boldsymbol{m},1,\delta) \overset{\star}{\simeq} H \right).
	\end{split}}
	Similarly the probability on the LHS ({left hand side}) of \eqref{eq:thm:equiv:PA:CPU} for $\CPU_n^{\sss\rm{(SL)}}$ can be written in terms of $\PU_{m_{[n]}}^{\sss\rm{(SL)}}$ as
	\eqn{\label{for:equiv:CPU:PArs:2}
		\begin{split}
			\prob_m\left( \CPU_n^{\sss\mathrm{(SL)}}(\boldsymbol{m},\boldsymbol{\psi}) \overset{\star}{\simeq} G \right) =& \prob_m\left( \bigcup\limits_{H\in\mathcal{H}_{\boldsymbol{m}}(G)}\left\{\PU_{m_{[n]}}^{\sss\mathrm{(SL)}}(\boldsymbol{1},\boldsymbol{\psi}) \overset{\star}{\simeq} H\right\} \right)\\
			=& \sum\limits_{H\in\mathcal{H}_{\boldsymbol{m}}(G)}\prob_m\left( \PU_{m_{[n]}}^{\sss\mathrm{(SL)}}(\boldsymbol{1},\boldsymbol{\psi}) \overset{\star}{\simeq} H \right).
	\end{split}}
	Now from Proposition~\ref{prop:equiv:CPU:PArs} it follows that the summands in \eqref{for:equiv:CPU:PArs:1} and \eqref{for:equiv:CPU:PArs:2} are equal. Hence, conditionally on $\boldsymbol{m},~\PArs_n$ and $\CPU_n^{\sss\rm{(SL)}}$ are equal in distribution.
\end{proof}

We now move towards the proof of Proposition~\ref{prop:equiv:CPU:PArs}. Berger et al.\ \cite{BergerBorgs} have proved a version of Theorem~\ref{thm:equiv:CPU:PArs} for model (D) and degenerate {out-degrees} using an extension to multiple urns of the P\'olya urn {characterization in terms of conditionally independent events by de Finetti's Theorem. We could adapt this proof. Instead, we} prove Proposition~\ref{prop:equiv:CPU:PArs} by explicitly calculating the graph probabilities of both random graphs and equating them term by term, {which we now show.} This proof is interesting in its own right.

Let $v(u)$ denote the vertex to which the out-edge from $u$ connects in $H$. 
Then from \eqref{eq:edge_connecting:model(rs):1},
\eqn{\label{for:equiv:CPU:PArs:4}
	\begin{split}
		&\prob_m\left(  \PArs_{m_{[n]}}(\boldsymbol{m},1,\delta) \overset{\star}{\simeq} H \right)\\
		&\qquad =\prod\limits_{u\in[3,n]}\prod\limits_{j\in[m_{u}]} \frac{d_{v(m_{[u-1]}+j)}(m_{[u-1]}+j-1)+\delta(v(m_{[u-1]}+j))}{a_{[2]}+2(m_{[u-1]}+j-2)+\left( (u-1)+\frac{j}{m_u} \right)\delta-1}~.
\end{split}}
The following lemma simplifies and rearranges the factors in the numerator of \eqref{for:equiv:CPU:PArs:4}:
\begin{Lemma}[Rearrangement of the numerator of \eqref{for:equiv:CPU:PArs:4}]\label{lem:equiv:num:rearrange}
	The numerator of \eqref{for:equiv:CPU:PArs:4} can be rearranged as
	\eqn{\label{eq:lem:num:rearrange:equiv}
		\begin{split}
			&\prod\limits_{u\in[3,n]}\prod\limits_{j\in[m_{u}]}\left( d_{v(m_{[u-1]}+j)}(m_{[u-1]}+j-1)+\delta(v(m_{[u-1]}+j)) \right)\\
			&\qquad=\prod\limits_{k\in[n]}\prod\limits_{l\in[m_{u}{]}} \prod\limits_{i=a_{m_{[k-1]}+l}}^{d_{m_{[k-1]}+l}(H)-1}\left( i+\frac{\delta}{m_k} \right),
	\end{split}}
	where $d_v(H)$ denotes the degree of the vertex $v$ in the graph $H$.
\end{Lemma}
\begin{proof}
	{Observe that the factors in the numerator of RHS of \eqref{for:equiv:CPU:PArs:4} depend on the receiver's degree. Since the edges from the new vertices connect to one of the existing vertices (or itself), the product in the numerator of \eqref{for:equiv:CPU:PArs:4} can be rewritten as
		\eqn{\label{for:equiv:CPU:PArs:4-1}
			\begin{split}
				&\prod\limits_{u\in[3,n]}\prod\limits_{j\in[m_{u}]}\left( d_{v(m_{[u-1]}+j)}(m_{[u-1]}+j-1)+\delta(v(m_{[u-1]}+j)) \right)\\
				=& \prod\limits_{s\in[1,m_{[n]}]} \prod\limits_{\substack{u\in[3,n],\, j\in[m_u]\\v(m_{[u-1]}+j)=s}} \big( d_s(m_{[u-1]}+j-1)+\delta(s) \big)~.
		\end{split}}
		For the very first incoming edge to $s\in[m_{[n]}]$, we have the factor of $(a_s+\delta(s))$ and for the remaining ones, we have a factor in the RHS of \eqref{for:equiv:CPU:PArs:4-1} with an increment of $1$. On the other hand, for the last incoming edge to $s$ in $H$, we have the factor $(d_s(H)-1+\delta(s))$. For any $s\in[m_{[n]}]$, there exists {a unique} $k\in[n]$ and $l\in[m_u]$ such that $s=m_{[k-1]}+l$ and $\delta(s)=\delta/m_k$. Therefore the LHS of \eqref{for:equiv:CPU:PArs:4-1} can be further simplified as
		\eqan{\label{for:equiv:CPU:PArs:5}
				&\prod\limits_{s\in[1,m_{[n]}]} \prod\limits_{\substack{u\in[3,n],\, j\in[m_u]\\v(m_{[u-1]}+j)=s}} \big( d_s(m_{[u-1]}+j-1)+\delta(s) \big)\nn\\
				=&\prod\limits_{k\in[n]}\prod\limits_{l\in[{m_{u}]}} \prod\limits_{i=a_{m_{[k-1]}+l}}^{d_{m_{[k-1]}+l}(H)-1}\left( i+\frac{\delta}{m_k} \right).}
		}
\end{proof}

By Lemma~\ref{lem:equiv:num:rearrange} and a rearrangement of numerator, the graph probability in \eqref{for:equiv:CPU:PArs:4} can be written as
\begin{align}
	\label{for:equiv:CPU:PArs:6}
	&\prob_m\left(  \PArs_{m_{[n]}}(\boldsymbol{m},1,\delta) \overset{\star}{\simeq} H \right)\nn\\
	&\qquad= \prod\limits_{u\in[n]}\prod\limits_{j\in{[m_{u}]}} \prod\limits_{i=a_{m_{[u-1]}+j}}^{d_{m_{[u-1]}+j}(H)-1}\left( i+\frac{\delta}{m_u} \right)\\
	&\hspace{1.5cm}\times\prod\limits_{u\in[3,n]}\prod\limits_{j\in[{m_{u}]}} \frac{1}{a_{[2]}+2(m_{[u-1]}+j-2)+\left( (u-1)+\frac{j}{m_u} \right)\delta-1}.\nonumber
\end{align}
Next, we calculate the graph probabilities of  $\PU_{m_{[n]}}^{\sss\rm{(SL)}}(\boldsymbol{1},\boldsymbol{\psi})$ {and show that these agree}.
To calculate $\prob_m\left( \PU_{m_{[n]}}^{\sss\mathrm{(SL)}}(\boldsymbol{1},\boldsymbol{\psi})\overset{\star}{\simeq} H \right)$, we condition on the $\Beta$ random variables {as well}. We denote the conditional measure by $\prob_{m,\psi}$, i.e., for every event $\mathcal{E}$,
\eqn{
	\label{for:equiv:CPU:PAr:7}
	\prob_{m,\psi}(\mathcal{E}) = \prob\left(\mathcal{E} \mid (m_k)_{k\geq 3},(\psi_k)_{k\geq 1} \right).
}
{Under this conditioning, the edges of $\PU_{m_{[n]}}^{\sss\rm{(SL)}}$ are independent.} First we calculate the conditional edge-connection probabilities for $\PU_{m_{[n]}}^{\sss\rm{(SL)}}(\boldsymbol{1},\boldsymbol{\psi})$:
\begin{Lemma}[Conditional edge-connection probability of $\PU^{\sss\rm{(SL)}}$]\label{lem:edge Probability:Polya Urn graph}
	Conditionally on $\boldsymbol{m}$ and $(\psi_k)_{k\geq 1}$ defined in \eqref{def:psi:1} and \eqref{def:psi:2}, the probability of connecting the edge from {$u$ to $v$} in $\PU_{m_{[n]}}^{\sss\rm{(SL)}}(\boldsymbol{1},\boldsymbol{\psi})$ is given by $\psi_v (1-\psi)_{(v,u]}$.
\end{Lemma}
\begin{proof} By the construction of $\PU_{m_{[n]}}^{\sss\rm{(SL)}}(\boldsymbol{1},\boldsymbol{\psi})$,
	\begin{align*}
		\mathcal{S}_k^{\sss(m_{[n]})} &= (1-\psi)_{(k,m_{[n]}]},\qquad\text{and}\qquad 
		|I_k^{\sss(m_{[n]})}|
		=\psi_k(1-\psi)_{(k,m_{[n]}]}.
	\end{align*}
	Taking the ratio {gives}
	\begin{equation*}
		{\prob_{m, \psi}}\left(u\rightsquigarrow v\right) = \frac{|I_v^{\sss(m_{[n]})}|}{\mathcal{S}_u^{\sss(m_{[n]})}} = \psi_v\frac{(1-\psi)_{(v,m_{[n]}]}}{(1-\psi)_{(u,m_{[n]}]}}=\psi_v(1-\psi)_{(v,u]}.
	\end{equation*}
\end{proof}
In P\'olya urn graphs, conditionally on the $\Beta$ random variables, the edges are added independently, leading to the following lemma:
\begin{Lemma}[Conditional density of $\PU^{\sss\rm{(SL)}}$]\label{lem:PU:conditional:graph_probability}
	For any graph $H\in\mathcal{H}_{m_{[n]}}$,
	\begin{equation}
		\label{eq:PU:conditional:graph_probabilit}
		\prob_{m,\psi}\left( \PU_{m_{[n]}}^{\rm\sss(SL)}(\boldsymbol{1},\boldsymbol{\psi}) \overset{\star}{\simeq} H \right) =  \prod\limits_{s\in[2,{m_{[n]}}]} \psi_s^{p_s}(1-\psi_s)^{q_s},
	\end{equation}
	where
	\begin{align}\label{def:p:q}
		p_s=&~ d_s(H)-a_s,\qquad\text{and}\qquad 
		q_s=\sum\limits_{u \in\left( m_{[2]}, m_{[n]}\right]}\one_{\{s\in(v(u),u]\}}.
	\end{align}
\end{Lemma}
\begin{proof}
	For every {incoming edge} of the vertex $u$, there is a vertex $w\ge u$ such that $v(w)=u$. Conditionally on {$\boldsymbol{m}$ and $\boldsymbol{\psi}$}, the {edge-connection events} are independent. {We then note that} there are $p_s$ many incoming edges for vertex $s$ and the factor $\psi_s^{p_s}$ comes from that. {Every $u$ such that $s\in(v(u),u]$ gives rise to one factor $1-\psi_s$, giving rise to} $q_s$ many factors $1-\psi_s$.
\end{proof}
{Note that by definition of $p_s$ and $q_s$,
\eqan{
	p_{m_{[n]}}&=\one_{\{m_{[n]}~\mbox{creates a self-loop}\}},\label{eq:major:1:1}\\
	\mbox{and}\qquad q_{m_{[n]}}&=\one_{\{m_{[n]}~\mbox{does not create a self-loop}\}}\label{eq:major:1:2}.
}}
{Next we aim to take the expectation w.r.t.\ $\boldsymbol{\psi}$, and for this, we compute the expectation of powers of $\Beta$ variables:}
\begin{Lemma}[{Integer moments of $\Beta$ distribution}]\label{lemma:Beta Expectation}
	For all $a,b\in \N$ and $\psi\sim\mathrm{Beta}(\alpha,\beta)$,
	\begin{equation}\label{eq:Beta Expectation}
		\E\left[ \psi^a(1-\psi)^b \right] = \frac{(\alpha+a-1)_a(\beta+b-1)_b}{(\alpha+\beta+a+b-1)_{a+b}},
	\end{equation}
	where $(n)_k=\prod\limits_{i=0}^{k-1}(n-i)$ for $k\geq 1$.
\end{Lemma}

\begin{proof}
	A direct calculation for $\Beta$ random variables shows that
	\begin{align*}
		\E\left[ \psi^a(1-\psi)^b \right]=&\frac{B(\alpha+a,\beta+b)}{B(\alpha,\beta)} = \frac{\Gamma(\alpha+\beta)\Gamma(\alpha+a)\Gamma(\beta+b)}{\Gamma(\alpha)\Gamma(\beta)\Gamma(\alpha+\beta+a+b)}\\
		=&\frac{(\alpha+a-1)_a(\beta+b-1)_b}{(\alpha+\beta+a+b-1)_{a+b}}~,
	\end{align*}
	{as required.}
\end{proof}
As a consequence of {Lemmas}~\ref{lem:PU:conditional:graph_probability} and \ref{lemma:Beta Expectation}, we can calculate $\prob_{m}\left( \PU_{m_{[n]}}^{\rm\sss(SL)}(\boldsymbol{1},\boldsymbol{\psi}) \overset{\star}{\simeq} H \right)$:
\begin{Lemma}\label{lem:PU:probability_calculation}
	For	$H\in\mathcal{H}_{\boldsymbol{m}}(G)$,
	\eqn{\label{eq:lem:PU:probability:0}
		\prob_{m}\left( \PU_{m_{[n]}}^{\rm\sss(SL)}(\boldsymbol{1},\boldsymbol{\psi}) \overset{\star}{\simeq} H \right) = \prod\limits_{s\in {\left[ 2  ,m_{[n]} \right]}} \frac{(\alpha_s+p_s-1)_{p_s}(\beta_s+q_s-1)_{q_s}}{(\alpha_s+\beta_s+p_s+q_s-1)_{p_s+q_s}},
	}
	where $\alpha_s$ and $\beta_s$ are the first and second parameters of the $\Beta$ random variables defined in ~\eqref{def:psi:1} and ~\eqref{def:psi:2}, $p_s, q_s$ are {defined} in Lemma~\ref{lem:PU:conditional:graph_probability} and $\infsupp(M)$ is {the minimum of the support} of the random variable $M$.
\end{Lemma}

\begin{proof} 
	Lemma~\ref{lem:PU:conditional:graph_probability} {describes} the conditional probability distribution of the P\'olya {urn} graphs given the independent $\Beta$ random variables. To obtain the unconditional probability, we take {the} expectation on the RHS of \eqref{eq:PU:conditional:graph_probabilit} {w.r.t.\ the $\Beta$ random variables}. Since $\psi_{m_{[u-1]}+j}$ are {independent $\Beta$ random variables with parameters $\alpha_{m_{[u-1]}+j}$ and $\beta_{m_{[u-1]}+j}$ respectively},
	\begin{align*}
		\prob_m\left( \PU_{m_{[n]}}^{\rm\sss(SL)}(\boldsymbol{1},\boldsymbol{\psi}) \overset{\star}{\simeq} H \right) = \E_m \Big[ \prod\limits_{s\in\left[2, m_{[n]} \right]} \psi_s^{p_s}(1-\psi_s)^{q_s} \Big]
		=\prod\limits_{s\in\left[2, m_{[n]} \right]} \E_m\left[\psi_s^{p_s}(1-\psi_s)^{q_s}\right] \\
		= \frac{(\alpha_2+p_2-1)_{p_2}(\beta_2+q_2-1)_{q_2}}{(\alpha_2+\beta_2+p_2+q_2-1)_{p_2+q_2}}\prod\limits_{s\in\left[ 3,m_{[n]} \right]} \frac{(\alpha_s+p_s-1)_{p_s}(\beta_s+q_s-1)_{q_s}}{(\alpha_s+\beta_s+p_s+q_s-1)_{p_s+q_s}}.
	\end{align*}
	{This matches the RHS of \eqref{eq:lem:PU:probability:0}.}
\end{proof}
In the next lemma, we derive an alternative expression for $q_s$ defined in \eqref{def:p:q} which helps in understanding the RHS of \eqref{eq:lem:PU:probability:0}:
\begin{Lemma}
	The $q_s$ defined in \eqref{def:p:q} can be represented as
	\eqn{\label{eq:lemma:p:q}
		q_s=\begin{cases}
			0 &\mbox{ if }s=1,\\
			d_1(H)-a_1 &\mbox{ if }s=2,\\
			d_{[s-1]}(H) - a_{[2]} - 2\left(s-3\right) &\mbox{ if }s\geq 3,
		\end{cases}}
	where
	\[
	d_{[s-1]}(H) = \sum\limits_{v\in [s-1]}d_v(H), \qquad \mbox{and} \qquad a_{[2]}= a_1+a_2.
	\]
\end{Lemma}

\proof
	For $s=1$, we observe that $q_1$ is $0$ by definition and, for $s=2$,
	\eqn{\label{for:equiv:CPU:PAr:8}
		\begin{split}
			q_{2} = \sum\limits_{u\in \left( 2,m_{[n]} \right]} \one_{[2\in(v(u),u]]} = \sum\limits_{u\in \left( 2,m_{[n]} \right]} \one_{[v(u)=1]} = d_1(H) - a_1~.
	\end{split}}
	Simplifying $q_s$ for $s\geq 3$ {gives}
	\begingroup
	\allowdisplaybreaks	
	\eqan{\label{for:equiv:CPU:PAr:9}
			q_s&= \sum\limits_{u\in \left(m_{[2]},m_{[n]}\right]} \one_{\left[ s\in (v(u),u] \right]}
			=\sum\limits_{u\in \left(2,m_{[n]}\right]} \one_{\left[ s\in \left(v(u),m_{[n]}\right] \right]} - \sum\limits_{u\in \left(2,m_{[n]}\right]} \one_{\left[ s\in \left(u,m_{[n]}\right] \right]}\nn\\
			&=\sum\limits_{u\in \left(2,m_{[n]}\right]} \one_{\left[ v(u)\in [s-1] \right]} - \sum\limits_{u\in \left(2,m_{[n]}\right]} \one_{\left[ u\in [s-1] \right]} \\
			& =\Big( \sum\limits_{u\in \left[ m_{[n]}\right]} \one_{\left[ v(u)\in [s-1] \right]} - \sum\limits_{u\in \left[2\right]} \one_{\left[ v(u)\in [s-1] \right]} \Big) -\Big( \sum\limits_{u\in \left[m_{[n]}\right]} \one_{\left[ u\in [s-1] \right]} - \sum\limits_{u\in \left[2\right]} \one_{\left[ u\in [s-1] \right]} \Big)\nn\\
			& =\sum\limits_{v\in[s-1]}\left( d_v^{\sss\rm{(in)}}(H) - d_v^{\sss\rm{(in)}}(G_0) \right) - \left((s-1)-2\right),\nn
	}
	\endgroup
	where $d_v^{\sss\rm{(in)}}(G)$ is the in-degree of vertex $v$ in the graph $G$ and $G_0$ is the initial graph we started with. Let $d_v^{\sss\rm{(out)}}(G)$ denote the out-degree of vertex $v$ in the graph $G$, {so that}
	\[
	d_v^{\sss\rm{(in)}}(G) + d_v^{\sss\rm{(out)}}(G) = d_v(G).
	\]
	{Note that} the vertices in the initial graph $G_0$ do not have out-edges directed to any new incoming vertices. Therefore, $d_v^{\sss\rm{(out)}}(H) = d_v^{\sss\rm{(out)}}(G_0)$ for all $v\in \left[ 2 \right]$.
	
	On the other hand, the new incoming vertices have exactly one out-edge each. Furthermore they are not {part of} the initial graph. {Hence,} by definition $d_v(G_0) = 0$ for all $v>2=m_{[2]}$. Therefore, both $d_v^{\sss\rm{(out)}}(G_0)$ and $d_v^{\sss\rm{(in)}}(G_0)$ are zero {for all $v>2=m_{[2]}$.}
	Hence, for $s\geq 3$,
	\eqan{\label{for:equiv:CPU:PAr:10}
			q_s&= \sum\limits_{v\in[s-1]}\left( d_v^{\sss\rm{(in)}}(H) - d_v^{\sss\rm{(in)}}(G_0) \right) - \left((s-1)-2\right)\nn\\
			&= \sum\limits_{v\in[s-1]}\left( d_v(H) - d_v(G_0) \right) - \sum\limits_{v\in[s-1]}\left( d_v^{\sss\rm{(out)}}(H) - d_v^{\sss\rm{(out)}}(G_0) \right) - \left(s-3\right)\nn\\
			&= \sum\limits_{v\in[s-1]}\left( d_v(H) - d_v(G_0) \right) - \sum\limits_{v\in\left[ 2 \right]}\left( d_v^{\sss\rm{(out)}}(H) - d_v^{\sss\rm{(out)}}(G_0) \right)\nn\\
			&\hspace{7cm} - \sum\limits_{v\in \left( 2,s-1 \right]} d_v^{\sss\rm{(out)}}(H)- \left(s-3\right)\\
			&= \sum\limits_{v\in[s-1]}\left( d_v(H) - d_v(G_0) \right) - 2\left(s-3\right)\nn\\
			&= d_{[s-1]}(H) - a_{[2]} - 2\left(s-3\right).\nn}
\vspace{-0.9cm}

~\hfill\qed
\vspace{0.5cm}

Now, we have all the tools to prove Proposition~\ref{prop:equiv:CPU:PArs}:
\begin{proof}[{Proof of Proposition~\ref{prop:equiv:CPU:PArs}}]
	We start by simplifying the RHS of \eqref{eq:lem:PU:probability:0} and equating term by term. First, we consider the term corresponding to $s=2$ in the RHS of \eqref{eq:lem:PU:probability:0}.
	Using ~\eqref{for:equiv:CPU:PAr:8} and substituting the values of $\alpha_2$ and $\beta_2$,
	\eqn{\label{for:prop:equiv:PU:PAr:1}
		\begin{split}
			&\frac{(\alpha_2+p_2-1)_{p_2}(\beta_2+q_2-1)_{q_2}}{(\alpha_2+\beta_2+p_2+q_2-1)_{p_2+q_2}}\\
			&\qquad\qquad= \frac{(\delta+d_2(H)-1)_{d_2(H)-a_2}(d_1(H)+\delta-1)_{d_1(G)-a_1}}{(d_{[2]}(G)+2\delta-1)_{d_{[2]}(G)-a_{[2]}}}\\
			&\qquad\qquad= \frac{1}{\left( d_{[2]}(H)+2\delta-1 \right)_{d_{[2]}(H)-a_{[2]}}}\prod\limits_{i=a_1}^{d_1(H)-1}(i+\delta)\prod\limits_{i=a_2}^{d_2(H)-1}(i+\delta).
	\end{split}}
	The factor $(\alpha_{s}+p_{s}-1)_{p_{s}}$ in the numerator of \eqref{eq:lem:PU:probability:0} can be simplified as
	\eqan{\label{for:prop:equiv:PU:PAr:2}
			&(\alpha_{s}+p_{s}-1)_{p_{s}}={\prod\limits_{i=1}^{p_{s}}} \left( \alpha_s-1+i \right)=\prod\limits_{i=a_{s}}^{d_{s}(H)-1} \left( i+ \alpha_s-1 \right).}
	The last equality in \eqref{for:prop:equiv:PU:PAr:2} is obtained by substituting the expression for $p_{m_{[u-1]}+j}$ and a simple change of variables. For any $s\in[3,m_{[n]}]$, we can find a $u\in[3,n]$ and $j\in [m_u]$ such that $s=m_{[u-1]}+j$. Therefore, 
	\eqn{\label{for:prop:equiv:PU:PAr:3}
		\prod\limits_{s\in\left[ 3,m_{[n]} \right]}(\alpha_s+p_s-1)_{p_s} =  \prod\limits_{u\in[3,n]}\prod\limits_{j\in \left[ m_u \right]}\prod\limits_{i=a_{m_{[u-1]}+j}}^{d_{m_{[u-1]}+j}(H)-1} \left( i+ \frac{\delta}{m_u} \right).}
	So far, from \eqref{for:prop:equiv:PU:PAr:1} and \eqref{for:prop:equiv:PU:PAr:3}, we have already obtained the first factor in \eqref{for:equiv:CPU:PArs:6} and an additional term in the denominator. 
	Note that for $s\geq 3$,
	\eqn{\label{for:prop:equiv:PU:PAr:4}
		\begin{split}
			\alpha_s+\beta_s = \beta_{s+1}-1,\qquad\mbox{and}\qquad
			p_s+q_s = q_{s+1}+1.
	\end{split}}
	Therefore the remaining term in \eqref{eq:lem:PU:probability:0} can be rewritten in a telescoping product form {as}
{\eqan{\label{for:prop:equiv:PU:PAr:5}
			&\prod\limits_{s\in\left[3,m_{[n]}\right]}\frac{(\beta_s+q_s-1)_{q_s}}{(\alpha_s+\beta_s+p_s+q_s-1)_{p_s+q_s}}\nn\\
			&\qquad= \prod\limits_{s\in\left[3,m_{[n]}\right]}\frac{(\beta_s+q_s-1)_{q_s}}{(\beta_{s+1}+q_{s+1}-1)_{q_{s+1}+1}}\nn\\
			&\qquad= \prod\limits_{s\in\left[3,m_{[n]}\right]} \frac{1}{\left( \beta_{s+1}-1 \right)}\prod\limits_{s\in\left[3,m_{[n]}\right]}\frac{(\beta_s+q_s-1)_{q_s}}{(\beta_{s+1}+q_{s+1}-1)_{q_{s+1}}}.
	}		}
	Since $(\beta_{3}+q_3-1)_{q_3}=\left(d_{[2]}(H)+2\delta-1\right)_{d_{[2]}(H)-a_{[2]}}$ and {$q_{m_{[n]}+1}=0$}, {we can simplify \eqref{for:prop:equiv:PU:PAr:5} as}
	\eqan{\label{for:prop:equiv:PU:PAr:6}
		&\prod\limits_{s\in{\left[3,m_{[n]}\right]}}\frac{(\beta_s+q_s-1)_{q_s}}{(\alpha_s+\beta_s+p_s+q_s-1)_{p_s+q_s}}\nn\\
		&\qquad=\left(d_{[2]}(H)+2\delta-1\right)_{d_{[2]}(H)-a_{[2]}}\\
		&\qquad\qquad\times\prod\limits_{u\in[3,n]}\prod\limits_{j\in [m_u]}\frac{1}{a_{[2]}+2\left( m_{[u-1]}+j-2 \right)+(u-1)\delta+\frac{j}{m_u}\delta-1}.\nonumber
	}
	Therefore, substituting these simplified expressions from \eqref{for:prop:equiv:PU:PAr:6} and \eqref{for:prop:equiv:PU:PAr:1} in \eqref{eq:lem:PU:probability:0}, we obtain:
	\eqan{\label{for:prop:equiv:PU:PAr:7}
		&\prob_m\left( \PU_{m_{[n]}}^{\sss\rm{(SL)}}(\boldsymbol{1},\boldsymbol{\psi})\overset{\star}{\simeq} H \right)\nn \\
		&\qquad= \prod\limits_{u\in[1,n]}\prod\limits_{j\in \left[ m_u \right]}\prod\limits_{i=a_{m_{[u-1]}+j}}^{d_{m_{[u-1]}+j}(H)-1} \left( i+ \frac{\delta}{m_u} \right)\nn\\
		&\hspace{2cm}\times \prod\limits_{u\in[3,n]}\prod\limits_{j\in [m_u]}\frac{1}{a_{[2]}+2\left( m_{[u-1]}+j-2 \right)+(u-1)\delta+\frac{j}{m_u}\delta-1}~,
	}
	which, {as required,} matches the expression of $\prob_m\left( \PArs_{m_{[n]}}\left( \boldsymbol{m},1,\delta \right)\overset{\star}{\simeq} H \right)$ in \eqref{for:equiv:CPU:PArs:6}.
\end{proof}

\subsection{Equivalence of models (B) and (D) with their respective P\'olya urn descriptions}
{{Following similar calculations, we can show that model (B) and (D) are equal in distribution to $\CPU^{\sss\rm{(NSL)}}$ and $\PU^{\sss\rm{(NSL)}}$, respectively.}} 
{For \(k\geq 3\) and \(l\in[m_{k}{]}\), define
\eqn{\label{def:psi:1-1}
	\psi_{m_{[k-1]}+l} \sim {\Beta}\left( 1+\frac{\delta}{m_k},\, a_{[2]} + 2\left( m_{[k-1]}+l-3 \right)+1+(k-1)\delta+\frac{(l-1)}{m_k}\delta \right),
}
where $a_{[2]}=a_1+a_2$ and, for $k\leq 2$,
\eqn{\label{def:psi:1-2}
	\begin{split}
		\psi_1 \equiv~1,\qquad
		\mbox{and}\qquad \psi_{2} \sim {\Beta}\left( a_2+\delta,a_1+\delta \right).
\end{split}}
Then, $\PArt_n (\boldsymbol{m},\delta)$ is equal in distribution {to} $\CPU_n^{\sss\rm{(NSL)}}(\boldsymbol{m},{\boldsymbol{\psi}})$ where $\boldsymbol{\psi}$ is the sequence of $\Beta$ random variables \( (\psi_{i})_{i\geq1} \) as defined in \eqref{def:psi:1-1} and \eqref{def:psi:1-2}.}
\begin{Theorem}[Equivalence of model (B) and $\CPU^{\sss\rm{(NSL)}}$]\label{thm:equiv:CPU:PArt}
	Conditionally on $\boldsymbol{m}$, for any graph $G\in\mathcal{H}_{n}$,
	\eqn{\label{eq:thm:equiv:PArt:CPU}
		\prob_m\left( \PArt_n(\boldsymbol{m},\delta) \overset{\star}{\simeq} G \right) = \prob_m\left( \CPU_{n}^{\sss\mathrm{(NSL)}}(\boldsymbol{m},\boldsymbol{\psi}) \overset{\star}{\simeq} G \right)~,}
		{where \(\bf{\psi}\) is defined in \eqref{def:psi:1-1} and \eqref{def:psi:1-2}.}
\end{Theorem}

Next, we describe the equivalence of $\PAri_n(\boldsymbol{m},\delta)$ and $\PU_n^{\sss\rm{(SL)}}(\boldsymbol{m},\boldsymbol{\psi})$ where $\boldsymbol{\psi}$ is a sequence of independent $\Beta$ random variables defined as
\eqn{\label{eq:def:psi:1}
	\begin{split}
		&\psi_v\sim\Beta\left( m_v+\delta, a_{[2]}+2\left( m_{[v-1]}-2 \right)+m_v+(v-1)\delta \right)\quad \mbox{for }v\geq 3,\qquad\text{and}\\
		&\psi_2\sim\Beta(a_2+\delta,a_1+\delta),\qquad\text{and}~\psi_1=1.
\end{split}}
\begin{Theorem}[Equivalence of model (D) and $\PU^{\sss\rm{(NSL)}}$]\label{thm:equiv:PU:PAri}
	Conditionally on $\boldsymbol{m}$, for any graph $G\in\mathcal{H}_{n}$,
	\eqn{\label{eq:thm:equiv:PAri:CPU}
		\prob_m\left( \PAri_n(\boldsymbol{m},\delta) \overset{\star}{\simeq} G \right) = \prob_m\left( \PU_{n}^{\sss\mathrm{(NSL)}}(\boldsymbol{m},\boldsymbol{\psi}) \overset{\star}{\simeq} G \right),}
	where $\boldsymbol{\psi}$ is the sequence of $\Beta$ random variables defined in \eqref{eq:def:psi:1}.
\end{Theorem}

\arxiversion{The proofs of Theorem~\ref{thm:equiv:CPU:PArt} and \ref{thm:equiv:PU:PAri} follow in exactly the same way as that of Theorem~\ref{thm:equiv:CPU:PArs}, and we have included them in Appendix~\ref{sec:appendix}.}
\journalversion{The proofs of Theorem~\ref{thm:equiv:CPU:PArt} and \ref{thm:equiv:PU:PAri} follow in exactly the same way as those of Theorem~\ref{thm:equiv:CPU:PArs}. {These proofs can be obtained in \cite[Appendix A]{GHvdHR22}.}}

\begin{Remark}[Models (E) and (F)]
	{\rm We have now presented P\'olya urn representations for models (A), (B) and (D), but not for models (E) and (F). This is due to the fact that models (E) and (F) do not have such representations (except for $m=1$). As a result, we handle model (E) using a coupling to model (D) in Section \ref{sec:model(F):coupling}.} \hfill$\blacksquare$
\end{Remark}

\section{Preliminary results}
\label{sec:prelim:result}
For analysing the convergence to the local limit of $\CPU$, we need some analytical results of both $\CPU$ and $\RPPT$. Some of the proofs here follow {those} in \cite{BergerBorgs}, {while  some include significant novel ideas.} We explain how these results {can} be reproduced for $\CPU^{\sss\rm{(NSL)}}$ or $\PU$ at the end of this section. {This section is organised as follows. In Section \ref{sec-prel-exp}, we provide auxiliary results on expected values of random variables. In Section \ref{sec-pos-conc-Beta-gamma-coupling}, we use these to analyze the asymptotics of the positions $\mathcal{S}_k^{\sss(n)}$ and for an effective coupling of $\Beta$ and Gamma variables. Finally, in Section \ref{sec-attachment-RPPT}, we use these results to study the asymptotics of the attachment probabilities, and we prove some regularity properties of the RPPT.
}

\subsection{Preliminaries on expectations of random variables}
\label{sec-prel-exp}


Conditionally on $\boldsymbol{m},$ let $(i_k)_{k\in\N}$ be a sequence of independent Gamma random variables with parameters $m_k+\delta$ and $1$.
Since $M$ has finite $p$-th moment, the $p$-th moment of all $\left( i_k \right)_{k\geq 1}$ are {finite as well:}

\begin{Lemma}\label{lem:finite_moment}
	{The random variables} $\left( i_k \right)_{k\geq 1}$ have {uniformly bounded} $p$-th moment.
\end{Lemma}
\begin{proof}
	Fix $u\in\N$. Let conditionally on $\boldsymbol{m},\ X_1,\ldots,X_{m_u}$ be independent Gamma random variables with parameters $1+{\delta/m_u}$ and $1$. Hence
	\[
	\E_m\left[ X_1^{p} \right] = \frac{\Gamma\left( 1+\frac{\delta}{m_u}+p \right)}{\Gamma\left( 1+\frac{\delta}{m_u} \right)}\leq \frac{\Gamma(1+|\delta|+p)}{A}<K,
	\]
	where
	\begin{align*}
		A={ \min\limits_{z\in[1,2]}\Gamma(z),}
	\end{align*}

	{Note that \(\Gamma(1)=\Gamma(2)=1\), and \(\Gamma(z)\) is a strictly positive convex function for \(z\in\R^{+}\). Therefore, \(A>0\) and, since the gamma function is decreasing in \((0,1]\) and increasing in $[1,\infty)$, the value $A$ is the global minimum of the  gamma function on \(\R^{+}\).}
	
	By the triangle inequality, $\E_m\left[|X_1-\E_m[X_1]|^{p}\right]$ is bounded and the upper bound is independent of $u$ and $m_u$. {By} \cite[Corollary 8.2]{allangut}, for $X_1,\ldots, X_n$ i.i.d.\ mean $0$ random variables with finite $\ell$th moment, there exists a constant $B_{\ell}<\infty$ depending only on $\ell$ such that
	\eqn{\label{allan:gut:result}
		\E[(X_1+\cdots+X_n)^\ell]\leq \begin{cases}
			B_{\ell} n \E[|X_1|^\ell] &\mbox{for }1\leq \ell \leq 2,\\
			B_{\ell} n^{\ell/2}\E\left[|X_1|^{\ell/2}\right] &\mbox{for }\ell>2.
	\end{cases}}
	Since $i_k=X_1+\cdots+X_{m_k}$ and $X_1,\ldots,X_{m_k}$ are i.i.d.\ random variables, there exists {a finite constant} $B_{p}$ depending only on $p$ such that
	\eqn{\E_m\left[ |i_k-\E_m[i_k]|^{p} \right] \leq B_{p} m_k^{{p}} \E_m\left[|X_1-\E_m[X_1]|^{p}\right]\leq  C_1m_k^{{p}}~,}
	{for some $C_1>0$}. Using {the} triangle inequality {for the  $L_{p}$-norm},
	\eqn{\label{for:lem:finite_moment:1}
		\E_m\left[ i_k^{p} \right] \leq \E_m\left[ i_k \right]^{p} + C_1m_k^{{p}} \leq C m_k^{p},}
	{for some $C>0$.} Since we have assumed the existence of the $p$-th moment of $M$, \eqref{for:lem:finite_moment:1} implies that  $\E\left[ i_k^{p} \right]$ is uniformly bounded from above.
\end{proof}

\subsection{Position concentration and {Beta-Gamma} couplings}
\label{sec-pos-conc-Beta-gamma-coupling}
From \eqref{eq:edge-connecting_probability:CPU:SL} we obtain the conditional edge-probabilities for $\CPU^{\sss\rm{(SL)}}$ as
\eqn{\label{edge:connecting:probability:CPU:SL}
	\begin{split}
		\prob_{m,\psi}\left( u\overset{j}{\rightsquigarrow} v~\text{in }\CPU_n^{\sss\rm{(SL)}} \right) =~ \begin{cases}
			\frac{\mathcal{S}_v^{\sss(n)}}{\mathcal{S}_{u,j}^{\sss(n)}}\left( 1-\prod\limits_{l\in [m_v]}\left(1-\psi_{m_{[v-1]+l}}\right) \right)~&\text{for }u>v,\\ 
			1-\prod\limits_{l\in[j]}\left(1-\psi_{m_{[v-1]+l}}\right)~&\text{for }u=v.
		\end{cases}
\end{split}}
Similarly, from \eqref{eq:edge-connecting_probability:CPU:NSL}, the conditional edge-probabilities {for} $\CPU^{\sss\rm{(NSL)}}$ are given by
\eqn{\label{edge:connecting:probability:CPU:NSL}
	\begin{split}
		\prob_{m,\psi}\left( u\overset{j}{\rightsquigarrow} v~\text{in }\CPU_n^{\sss\rm{(NSL)}} \right) =~ \begin{cases}
			\frac{\mathcal{S}_v^{\sss(n)}}{\mathcal{S}_{u,j-1}^{\sss(n)}}\left( 1-\prod\limits_{l\in [m_v]}\left(1-\psi_{m_{[v-1]+l}}\right) \right)~&\text{for }u>v,\\ 
			1-\prod\limits_{l\in[j-1]}\left(1-\psi_{m_{[v-1]+l}}\right)~&\text{for }u=v,
		\end{cases}
\end{split}}
and for $\PU^{\sss\rm{(NSL)}}$ it is
\eqn{\label{edge:connecting:probability:PU:NSL}
	\begin{split}
		\prob_{m,\psi}\left( u\overset{j}{\rightsquigarrow} v~\text{in }\PU_n^{\sss\rm{(NSL)}} \right) =~ 
		\psi_v\frac{\mathcal{S}_v^{\sss(n)}}{\mathcal{S}_{u-1}^{\sss(n)}}~.
\end{split}}

{We thus need to analyze} the asymptotics of the $\mathcal{S}_k^{\sss(n)}$ of $\CPU_n^{\sss\rm{(SL)}}(\boldsymbol{m},\boldsymbol{\psi})$ with $\boldsymbol{\psi}$ defined in \eqref{def:psi:1} and \eqref{def:psi:2}, which we do in the following proposition:
\begin{Proposition}[Position concentration for PU and CPU]\label{lem:position:concentration}
	Recall that $i=\frac{\E[M]+\delta}{2\E[M]+\delta}$. Then, for every $\varepsilon,\omega>0,$ there exists ${K(\varepsilon,\omega)}<\infty$ such that, for all $n>{K(\varepsilon,\omega)}$ and with probability at least $1-\varepsilon$, for both $\PU$ and $\CPU$,
	\eqn{
		\label{eq:lem:position:concentration:1}
		\max\limits_{k\in[n]} \left| \mathcal{S}_k^{\sss(n)} - \left(\frac{k}{n}\right)^i \right| \leq \omega,
	}
	and
	\eqn{
		\label{eq:lem:position:concentration:2}
		\max\limits_{k\in[n]\setminus [{K(\varepsilon,\omega)}]} \left(\frac{k}{n}\right)^{-i}\left| \mathcal{S}_k^{\sss(n)} - \left(\frac{k}{n}\right)^i \right| \leq \omega.
	}
\end{Proposition}
{Note that although the $\Beta$ random variables are different, the position concentration lemma holds true for both $\PU$ and $\CPU$.}
\arxiversion{The proof to Proposition~\ref{lem:position:concentration} is an adaptation of the proof of \cite[Lemma~3.1]{BergerBorgs} by Berger et al.\ to our setting; see Appendix \ref{app:preliminary:result}.}
\journalversion{{The proof to Proposition~\ref{lem:position:concentration} is an adaptation of the proof by Berger et al.\ \cite[Lemma~3.1]{BergerBorgs} to our setting, see the \cite[Appendix B]{GHvdHR22}. }}

In collapsed P\'olya urn graphs, the sequence $\{ \mathcal{S}_{k,j}^{\sss(n)}\colon j\in[m_k] \}$ defined in \eqref{eq:def:position:CPU} is an increasing sequence and from {Proposition}~\ref{lem:position:concentration}, with probability at least $1-\varepsilon,$ for all $n>k>{K(\varepsilon,\omega)},$
\eqn{\label{position_concentration:rest}
	\begin{split}
		\mathcal{S}_{k,j}^{\sss(n)} &\leq \mathcal{S}_{k}^{\sss(n)} \leq \left( \frac{k}{n} \right)^i (1+2\omega),\\
		\mbox{and }\qquad\mathcal{S}_{k,j}^{\sss(n)} &\geq \mathcal{S}_{k-1}^{\sss(n)} \geq \left( \frac{k-1}{n} \right)^i (1-\omega) \geq \left( \frac{k}{n} \right)^i (1-2\omega).
\end{split} }
Therefore it is evident that $k$ for sufficiently large, the variation in {$j\mapsto \mathcal{S}_{k,j}^{\sss(n)}$ is minor.} Hence the {differences between the} SL and NSL {versions} of the collapsed P\'olya urn graphs {will also be minor.} 
The following proposition provides us with a nice coupling between the $\Beta$ random variables $\boldsymbol{\psi}$ and a sequence of Gamma variables. This coupling {is very} useful in analysing the remaining terms in \eqref{edge:connecting:probability:CPU:SL}, \eqref{edge:connecting:probability:CPU:NSL} and \eqref{edge:connecting:probability:PU:NSL}. Before diving into the proposition and its proof, let us denote $p=1+\varrho$ for some $\varrho>0$.
\begin{Proposition}[Beta-Gamma coupling for $\PU$ and $\CPU$]\label{prop:Betagamma:coupling:CPU}
	Consider the sequence $(\psi_k)_{k\in\N}$ for $\CPU$ and define the sequence $(\phi_k)_{k\in\N}$ as
	\eqn{
		\label{def:phi}
		\phi_v:=\sum\limits_{j\in[m_v]} \psi_{m_{[v-1]}+j}.
	}
	{Recall that, {conditionally on $\boldsymbol{m}$}, $i_k$ has a Gamma distribution with parameters $m_k+\delta$ and $1$.
	Then, {there exist $K_\varepsilon$ and $K_\eta$ such that}}
	\begin{itemize}
		\item[(i)] $i_k\leq k^{1-\varrho/2}$ for all $k\geq K_\varepsilon$ with probability at least $1-\varepsilon$;
		\smallskip
		\item[(ii)] consider the function $h_k^{\phi}(x)$ such that $\prob(\phi_k\leq h_k^{\phi}(x)) = \prob(i_k\leq x)$. 
		Then, for every $\eta\in\left(0,\tfrac{1}{2}\right),$ there exists sufficiently large $K_\eta\geq 1$ depending on $\eta$, such that, for all $k\geq K_{\eta}$ and $x\leq k^{1-\varrho/2}$,
		\eqn{\label{for:prop:Betagamma:coupling}
			\frac{1-\eta}{k(2\E[M]+\delta)}x\leq h_k^{\phi}(x)\leq \frac{1+\eta}{k(2\E[M]+\delta)}x.
		}
	\end{itemize}
	Since $(\psi_k)_{k\geq 1}$ is the $\PU$ analogue of the sequence $(\phi_k)_{k\geq 1}$, replacing $(\phi_k)_{k\geq 1}$ by {the} corresponding $(\psi_k)_{k\geq 1}$ of $\PU$ defined in \eqref{eq:def:psi:1}, {a} similar Beta-Gamma coupling holds for $\PU$.
\end{Proposition}
Berger et al.\ \cite[Lemma~3.2]{BergerBorgs} {provide} a {related} result for degenerate $M$, but our {proof technique is a bit different}. For obtaining the upper and lower bounds in \eqref{for:prop:Betagamma:coupling}, we {use a} correlation inequality \cite[Lemma~1.24]{vdH2} and \emph{Chernoff's inequality}.
\begin{proof}
	{First, we start by proving $(i)$. Using \emph{Markov's inequality}, 
		\eqn{\label{for:lem:Beta-gamma:1:2}
			\prob\left( i_k\geq k^{1-\varrho/2} \right)=\prob\left( i_k^{1+\varrho}\geq k^{(1-\varrho/2)(1+\varrho)} \right)\leq \E\left(i_k^{1+\varrho}\right) k^{-\left( 1+\varrho(1-\varrho)/2 \right)}.
		}
		By Lemma~\ref{lem:finite_moment}, $\E\left( i_k^{p} \right)$ is uniformly bounded. Therefore the RHS of \eqref{for:lem:Beta-gamma:1:2} is summable and we obtain $(i)$ using the Borel-Cantelli lemma.}\\
	{We} proceed to prove $(ii).$
	Let 
	\[
	b_{u,j} = a_{[2]}+2(m_{[u-1]}+j-3)+(u-1)\delta +\frac{(j-1)}{m_u}\delta -1,
	\]
	and, conditionally on $\boldsymbol{m}$, let $\left(i_{k}^\prime\right)_{k\in\N}$ be a sequence of independent Gamma variables such that
	\eqn{
		\label{for:lem:Beta-gamma:2}
		i_{m_{[u-1]}+j}^\prime \sim {\rm{Gamma}}\left( 1+{\delta/m_{u}}, 1 \right).}
	The function $x\mapsto h_{u,j}^{\phi}(x)$ is defined such that for all $x\leq b_{u,j}^{1-\varrho/2},$
	\[
	\prob\left( \psi_{m_{[u-1]}+j}\leq h_{u,j}^{\phi}(x) \right) = \prob\left( i_{m_{[u-1]+j}}^\prime \leq x \right).
	\]
	To prove {part $(ii)$}, we start with the sequence $(\psi_k)_{k\in\N}$ with the aim to couple it with appropriately scaled $\left( i_k^\prime \right)_{k\in\N}$. We will prove \eqref{for:prop:Betagamma:coupling} using the following claim:
	\begin{Claim}\label{claim:beta-gamma-1}
		There exists $K_\eta\geq 1$ such that for all $x\leq \left( b_{u,j} \right)^{1-\varrho/2}$ and $u>K_\eta$,
		\eqn{\label{for:lem:Beta-gamma:3}
			\left(1-\frac{\eta}{2}\right)\frac{x}{b_{u,j}}\leq h_{u,j}^{\phi}(x)\leq \frac{x}{b_{u,j}}.}
	\end{Claim}
	\paragraph{{Proof of part $(ii)$ subject to Claim~\ref{claim:beta-gamma-1}}}
	From Claim~\ref{claim:beta-gamma-1} and part $(i)$, there exists a $K_\eta$ such that for any $u> K_\eta$, we can couple $\left(\psi_{m_{[u-1]}+j}\right)_{\substack{u>K_\eta\\j\in [m_u]}}$ and $\left(i_{m_{[u-1]}+j}^\prime\right)_{\substack{u>K_\eta\\j\in [m_u]}}$ with probability at least $1-\varepsilon$, such that
	{\eqn{\label{for:lem:Beta-gamma:4}
		\left( 1-\frac{\eta}{2} \right) \frac{i_{m_{[u-1]}+j}^\prime}{b_{u,j}} \leq \psi_{m_{[u-1]}+j} \leq \frac{i_{m_{[u-1]}+j}^\prime}{b_{u,j}}.
	}
	Rearranging the inequalities in \eqref{for:lem:Beta-gamma:4}, we obtain}
	\eqn{\label{eq:major:3:4}
		{\psi_{m_{[u-1]}+j}\leq \frac{i_{m_{[u-1]}+j}^\prime}{b_{u,j}}\leq \Big(1+\frac{3}{4}\eta\Big)\psi_{m_{[u-1]}+j}.}}
	Since $b_{u,j}$ is random, we replace it with its expectation and encounter some error. By the law of large numbers,
	\eqn{\label{for:lem:Beta-gamma:5}
		\frac{m_{[u-1]}}{u}\overset{a.s.}{\to}\E[M],}
	{and hence, depending on $\eta$ and \(\varepsilon\), there exists $K_\eta(\varepsilon)\geq 1$ large enough such that, with probability at least $1-\varepsilon$, the error {can} be bounded as}
	\eqn{
		\label{for:lem:Beta-gamma:6}
		{\left| \frac{b_{u,j}-u(2\E[M]+\delta)}{u(2\E[M]+\delta)} \right|\leq \frac{\eta}{4},\quad \mbox{for all } u>K_\eta.}
	}
	Therefore, from \eqref{for:lem:Beta-gamma:4} and \eqref{for:lem:Beta-gamma:6}, with probability at least $1-2\varepsilon,$
	{the sequences} $\left(\psi_{m_{[u-1]}+j}\right)_{\substack{u>K_\eta\\j\in m_{[u]}}}$ and $\left(i_{m_{[u-1]}+j}^\prime\right)_{\substack{u>K_\eta\\j\in m_{[u]}}}$ can be coupled such that
	\eqn{\left( 1-\eta \right)\psi_{m_{[u-1]}+j}\leq \frac{i_{m_{[u-1]}+j}^\prime}{u(2\E[M]+\delta)}\leq (1+\eta) \psi_{m_{[u-1]}+j}.}
	{Thus}, summing over all $j\in[m_u],$ we obtain \eqref{for:prop:Betagamma:coupling} subject to \eqref{for:lem:Beta-gamma:3}. {We are left to {prove} Claim \ref{claim:beta-gamma-1}:}
	\paragraph{{Proof of Claim \ref{claim:beta-gamma-1}}}
	To prove the claim it is enough to show that for {all} $x\leq \left( b_{u,j} \right)^{-\varrho/2}$ and $u>K_\eta\geq {1/\sqrt{\eta}},$
	\eqn{\label{for:lem:Beta-gamma:7}
		\begin{split}
			\prob_m\left( \psi_{m_{[u-1]}+j}\leq (1-\eta)x \right) \leq \prob_m\left( i_{m_{[u-1]}+j}^\prime\leq b_{u,j} x \right) \leq \prob_m\left( \psi_{m_{[u-1]}+j}\leq x \right).
		\end{split}
	}
	\subparagraph{{Upper bound in \eqref{for:lem:Beta-gamma:7}}}
	With $\alpha_u = 1+{\delta/m_u},$ we bound
	\eqn{
		\prob_m\left( i_{m_{[u-1]}+j}^\prime\leq b_{u,j} x \right) = \frac{\int\limits_0^{b_{u,j}x} y^{\alpha_u-1}\e^{-y}\,dy}{\int\limits_0^{\infty} y^{\alpha_u-1}\e^{-y}\,dy}\leq  \frac{\int\limits_0^{b_{u,j}x} y^{\alpha_u-1}\e^{-y}\,dy}{\int\limits_0^{b_{u,j}} y^{\alpha_u-1}\e^{-y}\,dy}.\nonumber
	}
	Then, using a change of variables and adjusting the missing {factors from the $\Beta$ density,} we get
	\eqn{\label{for:lem:Beta-gamma:8}
		\begin{split}
			\prob_m\left( i_{m_{[u-1]}+j}^\prime\leq b_{u,j} x \right)
			\leq \frac{\int\limits_{0}^1\one_{\{y\leq x\}}\e^{-b_{u,j}y}\left( 1-y \right)^{-b_{u,j}} {f_{m_{[u-1]}+j}(y)\,dy} }{\int\limits_{0}^1 \e^{-b_{u,j}y}\left( 1-y \right)^{-b_{u,j}} {f_{m_{[u-1]}+j}(y)\,dy }},
	\end{split}}
	where $f_k$ is the {density} of $\psi_k.$
 
	Note that the numerator is $\E_m\left[ \one_{[\psi_{m_{[u-1}+j}\leq x]}e^{-b_{u,j}\psi_{m_{[u-1}+j}}\left( 1-\psi_{m_{[u-1}+j} \right)^{-b_{u,j}} \right]$ and the denominator is $\E_m\left[ \e^{-b_{u,j}\psi_{m_{[u-1}+j}}\left( 1-\psi_{m_{[u-1}+j} \right)^{-b_{u,j}} \right]$. Further, ${z\mapsto\one_{\{z\leq x\}}}$ is non-increasing, {while} $z\mapsto \e^{-b_{u,j}z}\left( 1-z \right)^{-b_{u,j}}$ is increasing. {Recall that, for increasing $f$ and decreasing $g$ on the support of $X$,}
	\eqn{
		\label{for:lem:Beta-gamma:9}
		\E[f(X)g(X)]\leq \E[f(X)]\E[g(X)].
	}
	Therefore, using this correlation inequality in the RHS of \eqref{for:lem:Beta-gamma:8}, we have
	\eqn{\label{for:lem:Beta-gamma:10}
		\prob_m\left( i_{m_{[u-1]}+j}^\prime\leq b_{u,j} x \right) \leq \prob_m\left( \psi_{m_{[u-1]}+j}\leq x \right),
	}
	{which proves the upper bound in \eqref{for:lem:Beta-gamma:7}.}
	\subparagraph{{The lower bound}}
	The {lower bound} in \eqref{for:lem:Beta-gamma:7} can be handled in a similar, albeit slightly more involved, way. Indeed,
	\eqan{\label{for:lem:Beta-gamma:11}
			\prob_m &\left( \psi_{m_{[u-1]}+j}  \leq x(1-\eta) \right)= \frac{\int\limits_{0}^{x(1-\eta)b_{u,j}}y^{a_u-1}\left( 1-\frac{y}{b_{u,j}} \right)^{b_{u,j}}\,dy}{\int\limits_{0}^{b_{u,j}}y^{a_u-1}\left( 1-\frac{y}{b_{u,j}} \right)^{b_{u,j}}\,dy}\nn\\
			&\hspace{1cm}=\frac{\E_m\left[ \left.\one_{\big\{ i_{m_{[u-1]}+j}^\prime\leq b_{u,j}x(1-\eta) \big\}}\e^{i_{m_{[u-1]}+j}^\prime}\left( 1-\frac{i_{m_{[u-1]}+j}^\prime}{b_{u,j}} \right)^{b_{u,j}}\right| i_{m_{[u-1]}+j}^\prime\leq b_{u,j}\right]}{\E_m\left[ \left. \e^{i_{m_{[u-1]}+j}^\prime}\left( 1-\frac{i_{m_{[u-1]}+j}^\prime}{b_{u,j}} \right)^{b_{u,j}}\right| i_{m_{[u-1]}+j}^\prime\leq b_{u,j}\right]}.}
	Denote $f(z) = \one_{\{z\leq b_{u,j}x(1-\eta)\}}$ and $g(z)= \e^z\left( 1-\frac{z}{b_{u,z}} \right)^{b_{u,j}}$. {Note} that $f$ and $g$ are non-increasing and increasing functions, respectively. Therefore, {by} the correlation inequality in \eqref{for:lem:Beta-gamma:9} {once more},
	\eqn{\prob_m \left( \psi_{m_{[u-1]}+j}  \leq x(1-\eta) \right)\leq \prob_m\left( \left. i_{m_{[u-1]}+j}^\prime\leq b_{u,j}x(1-\eta) \right|i_{m_{[u-1]}+j}^\prime\leq b_{u,j} \right)~.}
	{Thus,} to prove the required inequality, it {suffices} to show that, for all $u\geq K_\eta$ and $j\in m_{[u]}$, and $x\leq (b_{u,j})^{-\varrho/2}$,
	\eqn{\label{for:lem:Betagamma:12}
		\prob_m\left(i_{m_{[u-1]}+j}^\prime\leq b_{u,j}x(1-\eta) \Big|i_{m_{[u-1]}+j}^\prime\leq b_{u,j} \right)\leq \prob_m\left( i_{m_{[u-1]}+j}^\prime\leq b_{u,j}x \right).}
	To simplify notation, instead of showing \eqref{for:lem:Betagamma:12}, we prove the same for any Gamma random variable $Z$ having parameters $\alpha$ and $1$, with $\alpha\in(0,1+|\delta|)$. Observe that showing \eqref{for:lem:Betagamma:12} is equivalent to showing
	\eqn{\label{for:lem:Betagamma:13}
		E(x)= \prob\left( Z\leq bx \right)\prob\left( Z\leq b \right)- \prob\left( Z\leq bx(1-\eta) \right) \geq 0,}
	with $x\leq b^{-\varrho/2}$ for large enough $b>K_\eta$.
		
	By definition, $E(0)=0.$ {Therefore, it} remains to show that \eqref{for:lem:Betagamma:13} holds for $x\in\left( 0,b^{-\varrho/2}\right)$. We simplify
	\eqn{
		\label{for:lem:Betagamma:14}
		E(x)= \prob\left( (1-\eta)bx\leq Z \leq bx \right) - \prob(Z>b)\prob(Z\leq bx).
	}
	\subparagraph{{Lower bounding the first expression of $E(x)$}}Using the properties of the density of the Gamma distribution, we can lower bound the first term as
	\eqn{\label{for:lem:Beta-gamma:15}
		\begin{split}
			\prob\left( (1-\eta)bx\leq Z \leq bx \right)\geq& \eta bx \min\left\{ \frac{(bx)^{\alpha-1}}{\Gamma(\alpha)}\e^{-bx}, \frac{(bx)^{\alpha-1}}{\Gamma(\alpha)}\e^{-bx}(1-\eta)^{\alpha-1}\e^{bx\eta} \right\}\\
			=& \eta \frac{(bx)^{\alpha}}{\Gamma(\alpha)}\e^{-bx}\min\left\{ 1, (1-\eta)^{\alpha-1}\e^{bx\eta}\right\}.
	\end{split}}
	With the fact that $\eta<\tfrac{1}{2},$ and $\e^{\eta bx}>1$, we can lower bound the first term in \eqref{for:lem:Betagamma:14} further by 
	\eqn{\label{for:lem:Beta-gamma:16}
		\prob\left( (1-\eta)bx\leq Z \leq bx \right)\geq \eta \frac{(bx)^{\alpha}}{\Gamma(\alpha)} \e^{-bx}\min \{ 2^{1-\alpha},1 \} \geq \eta \frac{(bx)^{\alpha}}{\Gamma(\alpha)} \e^{-bx} 2^{1-\lceil \alpha \rceil}.}
	\subparagraph{{Upper bounding the second term of $E(x)$}} By \emph{Chernoff's inequality},
	\eqn{\label{for:prop:Beta-gamma:Chernoff}
		\prob\left( Z>b \right)=\prob\left( \e^{\frac{1}{2}Z}>\e^{\frac{b}{2}} \right)\leq \E\left[ \e^{\frac{1}{2}Z} \right]\e^{-\frac{b}{2}}= 2^{\alpha}\e^{-\frac{b}{2}}.}
	We use the fact that $e^{-y}\leq 1$ for $y\geq 0$, for upper bounding the distribution function of $Z$ as
	\eqn{\label{for:lem:Beta-gamma:17}
		\prob\left( Z\leq bx \right) = \frac{1}{\Gamma(\alpha)}\int\limits_{0}^{bx} y^{\alpha-1} \e^{-y}\,dy\leq \frac{1}{\Gamma(\alpha)}\int\limits_{0}^{bx} y^{\alpha-1} \,dy= \frac{(bx)^\alpha}{\alpha\Gamma(\alpha)}.}
	Therefore, the second term in \eqref{for:lem:Betagamma:14} can be upper bounded using \eqref{for:lem:Beta-gamma:17} and \eqref{for:prop:Beta-gamma:Chernoff}, as
	\eqn{
		\label{for:lem:Betagamma:15}
		\prob(Z>b)\prob(Z\leq bx)\leq 2^{\alpha} \e^{-\frac{b}{2}}\frac{(bx)^\alpha}{\alpha\Gamma(\alpha)}.
	}
	Hence from \eqref{for:lem:Beta-gamma:16} and \eqref{for:lem:Betagamma:15}, we obtain
	\eqn{\label{for:lem:Betagamma:16}
		E(x)\geq \eta\frac{(bx)^{\alpha}}{\Gamma(\alpha)} \e^{-bx}2^{1-\lceil\alpha\rceil} - 2^{\alpha} \e^{-\frac{b}{2}}\frac{(bx)^\alpha}{\alpha\Gamma(\alpha)}= \frac{(bx)^\alpha}{\Gamma(\alpha)}\e^{-\frac{b}{2}}\left[ \e^{\left(\frac{1}{2}-x\right)b+\log \eta} 2^{1-\lceil \alpha \rceil} - \frac{2^\alpha}{\alpha} \right].
	}
	Remember that $\eta>b^{-2}$ and, for sufficiently large $b,\ x\leq b^{-\varrho/2}\leq \tfrac{1}{4}.$ Therefore, {by} \eqref{for:lem:Betagamma:16}, for $b\geq K_\eta,$ 
	\eqn{\label{for:lem:Beta-gamma:18}
		E(x)\geq \frac{(bx)^\alpha}{\Gamma(\alpha)}\e^{-\frac{b}{2}}\left[  2^{1-\lceil \alpha \rceil} \e^{\frac{b}{4}-2\log b} - \frac{2^\alpha}{\alpha} \right].
	}
	Using $\alpha\in (0,1+|\delta|)$ and taking $b$ large enough, the RHS of \eqref{for:lem:Beta-gamma:18} is non-negative and this proves \eqref{for:lem:Betagamma:13}. This completes the proof of coupling for $\CPU$.
		
	The proof for $\PU$ {is similar}. Indeed, $b_{u,j}$ is replaced by $b_u$ {given by}
	\eqn{\label{for:prop:bgc:PU:1}
		b_u=a_{[2]}+2\left( m_{[u-1]}-2 \right)+m_u+(u-1)\delta.}
	Now we show that there exists $K_\eta\geq 1$ such that for all $u>K_\eta$
	\eqn{\left( 1-\frac{\eta}{2} \right)\frac{x}{b_u}\leq h_u^{\psi}(x)\leq \frac{x}{b_u}.}
	To prove this we {need to} show that for $u>K_\eta\geq \frac{1}{\sqrt{\eta}},$ and $x\leq b_u^{-\varrho/2},$
	\[
	\prob_m\left( \psi_u\leq (1-\eta)x \right)\leq \prob_m\left( i_u\leq b_u x \right)\leq \prob_m\left( \psi_u\leq x \right),
	\]
	which we have already proved.
\end{proof}

\subsection{Asymptotics of attachment probabilities and regularity of the RPPT}
\label{sec-attachment-RPPT}
{The coupling arguments in Proposition~\ref{prop:Betagamma:coupling:CPU} give us a sequence of conditional} Gamma random variables $\left( \hat{i}_v \right)_{v\geq 2}$ corresponding to $\left( \phi_v \right)_{v\geq 3}$ such that $\hat{i}_v=(2\E[M]+\delta)v\phi_v$ for $v$ large enough ({where we assume that $v>K_\eta$}). Similarly, for P\'olya urn graphs, we can couple $\left( \hat{i}_v \right)_{v\geq 2}$ and $\left( \psi_v \right)_{v\geq 2}$.  These relations are crucial {to bound} the edge-connection probabilities in \eqref{edge:connecting:probability:CPU:SL} and \eqref{edge:connecting:probability:CPU:NSL}, {where we recall that $i = \frac{\E[M]+\delta}{2\E[M]+\delta}$:}

\begin{Lemma}[Bound on attachment probabilities]
	\label{lem:bound:attachment_probability}
	For $n\in\N$, consider $v\in[n]$. Then, for every $\varepsilon>0$ there exists $\omega>0$ such that, for both $\CPU$ and $\PU$, with probability larger than $1-\varepsilon$, for every $k\geq v\geq K_\omega$ and $j\in[m_k]$,
	\eqn{\nonumber
		(1-\omega)\frac{\hat{i}_{v}}{(2\E[M]+\delta)v }\Big(\frac{v}{k}\Big)^i
		\leq \prob_{m,\psi}\left(k\overset{j}{\rightsquigarrow} v \right)
		\leq (1+\omega) \frac{\hat{i}_{v}}{(2\E[M]+\delta)v }\Big(\frac{v}{k}\Big)^i,
	}
	where $\hat{i}_{v}$ is a Gamma random variable with parameters $m_v+\delta$ and 1.
\end{Lemma}

\begin{proof}
	Fix $n\in\N$ and $\varepsilon,\omega>0$. The edge-connection probabilities for $\CPU^{\sss\rm{(SL)}},\CPU^{\sss\rm{(NSL)}}$ and $\PU^{\sss\rm{(NSL)}}$ are calculated in \eqref{edge:connecting:probability:CPU:SL},\eqref{edge:connecting:probability:CPU:NSL} and \eqref{edge:connecting:probability:PU:NSL}, respectively. {Choose \(\omega^{\prime}\) such that \(\omega=6\omega^{\prime}/(1-2\omega^{\prime})\) and \(K_{\omega}=K(\varepsilon/3,\omega^{\prime})\).}
	By Proposition~\ref{lem:position:concentration}, with probability {at least} $1-\varepsilon/3$, {for all $k\geq v\geq K_\omega$ and $j\in[m_k]$,}
	\eqn{\label{for:lem:attprob:urn:modela:1:bis}
		\max\limits_{j\in[m_k]}\left|\frac{{\mathcal{S}^{\sss(n)}_v}}{{\mathcal{S}^{\sss(n)}_{(k,j)}}} - \left(\frac{v}{k}\right)^{i} \right| \leq    \left(\frac{\omega}{2}\right) \left(\frac{v}{k}\right)^i	,
	}
	where the value $i$ {comes} from the collapsed P\'olya {u}rn graph model.
	
	{We next} control the remaining terms in the edge-connection probabilities of \eqref{edge:connecting:probability:CPU:SL}, \eqref{edge:connecting:probability:CPU:NSL} and \eqref{edge:connecting:probability:PU:NSL}. For $\CPU,$ we rewrite the remaining expression as
	\eqn{
		\label{for:lem:attprob:urn:modela:2}
		1-\prod_{i=1}^{m_v}\left(1-\psi_{m_{[v-1]}+i}\right)  = \sum_{i=1}^{m_v} \psi_{m_{[v-1]}+i} E_m(v)=\phi_v+E_m(v),
	}
	{which implicitly defines the error term $E_m(v)$.} On the other hand, for $\PU,$ the remaining expression turns out to be $\psi_v.$ By the assumptions, we can apply Proposition~\ref{prop:Betagamma:coupling:CPU} to couple the Gamma random variables to $\phi_v$ and $\psi_v$ of $\CPU$ and $\PU$, respectively. As a consequence, with probability {at least} $1-\varepsilon/3$, simultaneously for all $v$ sufficiently large,
	\eqn{\nonumber
		\begin{split}
			&\left( 1-\frac{\omega}{2} \right) \frac{\hat{i}_{v}}{2\E[M]+\delta}\leq v\phi_{v} \leq \left( 1+\frac{\omega}{2} \right) \frac{\hat{i}_{v}}{2\E[M]+\delta},\\
			\text{and}\quad&\left( 1-\frac{\omega}{2} \right) \frac{\hat{i}_{v}}{2\E[M]+\delta}\leq v\psi_{v} \leq \left( 1+\frac{\omega}{2} \right) \frac{\hat{i}_{v}}{2\E[M]+\delta},
		\end{split}
	}
	for $\CPU$ and $\PU$, respectively, {where} $\hat{i}_{v}$ is a sequence of independent Gamma distribution with parameters $m_v+\delta$ and $1$.
	
	The term $ {|E_m(v)|}$ in \eqref{for:lem:attprob:urn:modela:2} is bounded by 
	\eqn{
		\label{for:lem:bound:attachment_probability:1}
		\begin{split}
			\sum\limits_{\substack{i\neq j\\i,j\in[m_v]}}\psi_{m_{[v-1]}+i}\psi_{m_{[v-1]}+j} \leq \Big( \sum\limits_{i\in[m_v]}\psi_{m_{[v-1]}+i} \Big)^2\leq (1+2\omega)\left( \frac{\hat{i}_v}{v} \right)^2.
		\end{split}
	}
	{By} Proposition~\ref{lem:position:concentration}, {$\hat{i}_v\leq v^{1-\varrho/2}$} for all $v>K_\varepsilon,$ with probability at least $1-\varepsilon$. Therefore, {by} \eqref{for:lem:bound:attachment_probability:1}, 
		\eqn{{|E_m(v)|}\leq {\frac{\hat{i}_v}{v}}O\left(v^{-\varrho/2}\right).}
	This completes the proof.
\end{proof}

\begin{remark}\label{remark:equal-edge-probability}
	{\rm
		Observe that the edge-connection probabilities are essentially the same for the three models: $\CPU_n^{\sss\rm{(SL)}},\ \CPU_n^{\sss\rm{(NSL)}}$ and $\PU_n^{\sss\rm{(NSL)}}$, except for the self-loop creation probability. {The latter} is insignificant for large graphs. Using Lemma~\ref{lem:bound:attachment_probability}, we approximate the attachment probabilities of all these models by the same expression, with an error {that can be effectively taken care of.}}\hfill $\blacksquare$
\end{remark}

We {close this} section by recalling a regularity property of the $\RPPT$ {that is similar to \cite[Lemma~3.3]{BergerBorgs} and} which {is} useful in Section~\ref{sec:local_convergence}. \arxiversion{We provide a proof to this in the appendix:}\journalversion{{Proof to the lemma is provided in \cite[Appendix B]{GHvdHR22}.}}
\begin{Lemma}[Regularity of {the} RPPT]\label{lem:RPPT:property}
	Fix $r\geq 0$ and $\varepsilon>0$. Then there exist $K,C<\infty$ and ${\eta(\varepsilon,r)} >0$ such that, with probability at least $1-\varepsilon$,
	\begin{enumerate}
		\item\label{RPPT_prop:1} $A_\omega\geq {\eta(\varepsilon,r)}$, for all $\omega\in B_r^{\sss (G)}(\emp)$;
		\item\label{RPPT_prop:2} $\left| B_r^{\sss (G)}(\emp) \right|\leq C$;
		\item\label{RPPT_prop:3} $\Gamma_\omega\leq K$, for all $\omega\in B_r^{\sss(G)}(\emp).$
	\end{enumerate}
\end{Lemma}

\section{Local convergence}
\label{sec:local_convergence}
We prove local convergence for the vertex-marked preferential attachment model (A). For any finite vertex-marked tree $\tree$, with vertex marks in $[0,1]$, let $V(\tree)$ denote the set of vertices of $\tree$ and $\{ a_\omega \in [0,1]:\omega\in V(\tree) \}$ denote the age-set of the vertices. {Fix $r\in\N$ and $\varepsilon>0$.} Let $G_n=\PArs_n(\boldsymbol{m},\delta)$. If {$B_r^{\sss(G_n)}(v)~\simeq~ \tree$ and} $v_\omega$ is the vertex in $G_n$ corresponding to vertex $\omega$ of $\tree$, then we define
\[
N_{r,n}\left(\tree,(a_\omega)_{\omega\in V(\tree)}\right)=\sum\limits_{v\in[n]}\one_{\left\{ B_r^{\sss(G_n)}(v)~\simeq~ \tree, ~|v_\omega/n-a_\omega|\leq 1/r~\forall\omega~\in~V(\tree) \right\}}~,
\]
where $B_r^{\sss(G_n)}(v)$ is the $r$-neighbourhood of $v$ in $G_n$. With $B_r^{\sss(G)}(\emp)$ denoting the $r$-neighbourhood of the $\RPPT(M,\delta)$ and $A_\omega$ the age in $\RPPT(M,\delta)$ of the node $\omega$ of $\tree$, we aim to show that
\eqn{\label{target:1}
	\frac{N_{r,n}\left(\tree,(a_\omega)_{\omega\in V(\tree)}\right)}{n}\overset{\prob}{\to}\mu\left( B_r^{\sss(G)}(\emp)\simeq~\tree, ~|A_\omega-a_\omega|\leq 1/r~\forall\omega~\in~V(\tree) \right).}
To prove \eqref{target:1}, we use the second moment method, i.e., we prove ${\E\left[ N_{r,n}\left(\tree,(a_\omega)_{\omega\in V(\tree)}\right)\right]/n}$ converges to the limit and {that} the variance of ${N_{r,n}/n}$ {vanishes for $n\rightarrow \infty$.} {Throughout this section, we consider $\eta=\eta(\varepsilon,r)$ as introduced in Lemma~\ref{lem:RPPT:property}.}
\subsection{First moment convergence}

Here we prove the first moment convergence. Let $\left( \tree, (a_\omega)_{\omega\in V(\tree)} \right)$ be a vertex-marked tree with marks $(a_\omega)_{\omega\in V(\tree)}$ taking values in $[0,1]^{|V(\tree)|}$. {We compute}
\begin{align}\label{firstmoment:aim}
	{\frac{1}{n}}\E\left[ N_{r,n}\left(\tree,(a_\omega)_{\omega\in V(\tree)}\right)\right]= &\prob\left( B_r^{\sss(G_n)}({o})~\simeq~ \tree, ~|v_\omega/n-a_\omega|\leq 1/r~\forall\omega~\in~V(t) \right), 
\end{align}
{where $o\in [n]$ is chosen uniformly at random. We aim to show that this converges to the RHS of \eqref{target:1}.}

Instead of proving the first moment convergence, we prove {the stronger statement that the {\em age densities}} of the vertices in the $r$-neighbourhood of a uniformly chosen vertex in $\PArs_n(\boldsymbol{m},\delta)$ converges pointwise to the age density of the nodes in the $r${-neighbourhood} of $\RPPT(M,\delta)$.

Define $f_{r,\tree}\left((A_\omega)_{\omega\in V(t)}\right)$ as the density of the ages in the $\RPPT(M,\delta)$, when the ordered $r$-neighbourhood $B_r^{\sss(G)}(\emp)$ is in the same equivalence class as $\tree$. Then,
\eqn{\label{eq:RPPT:density:1}
	\mu\left( B_r^{\sss(G)}(\emp)\simeq \tree, A_\omega\in \,da_\omega,~\forall \omega\in V(\tree) \right) = f_{r,\tree}\left((a_\omega)_{\omega\in V(t)}\right)\prod\limits_{\omega\in V(\tree)}\,da_\omega.}
\begin{Theorem}[First moment density convergence theorem]
	\label{thm:prelimit:density}
	Fix $\delta>-\infsupp(M),$ and consider $G_n=\PArs_n(\boldsymbol{m},\delta)$. Uniformly for all $a_\omega>\eta,$ distinct $v_\omega$, and $\hat{i}_{v_\omega}\leq (v_\omega)^{1-\frac{\varrho}{2}}$, for all $\omega\in V(\tree)$,
	\eqn{\label{eq:prelimit:density:1}
		\begin{split}
			&\prob_{m,\psi}\left( B_r^{\sss(G_n)}(v)~\simeq~ \tree,~v_\omega=\lceil na_\omega \rceil~ \forall \omega\in V(\tree) \right)\\
			&\qquad= (1+o_\prob(1))\frac{1}{n^{|V(\tree)|}}g_{r,\tree}\left( \left( {a_\omega} \right)_{\omega\in V(\tree)}; \left( m_{v_\omega},\hat{i}_{v_\omega} \right)_{\omega\in V(\tree)} \right),
	\end{split}}
	for some measurable function $g_{r,\tree}\left( \left( a_\omega \right)_{\omega\in V(\tree)}; \left( m_{v_\omega},\hat{i}_{v_\omega} \right)_{\omega\in V(\tree)} \right)$, where $(\hat{i}_v)_{v\in\N}$ are the {Gamma} random variables coupled with the corresponding $\Beta$ random variables in Proposition~\ref{prop:Betagamma:coupling:CPU}.
	Consequently, with $\left( \hat{i}_{v_\omega} \right)_{\omega\in V(\tree)}$ a conditionally independent sequence of ${\rm{Gamma}}(m_{v_{\omega}}+\delta,1)$ random variables,
	\eqn{\label{eq:prelimit:density:2}
		\E\left[ g_{r,\tree}\left( \left( a_\omega \right)_{\omega\in V(\tree)}; \left( m_{v_\omega},\hat{i}_{v_\omega} \right)_{\omega\in V(\tree)} \right) \right] = f_{r,t}\left((a_\omega)_{\omega\in V(t)}\right).}
\end{Theorem}
By \eqref{eq:prelimit:density:2}, Theorem~\ref{thm:prelimit:density} can be seen as a \textit{local density limit theorem} for the ages of the vertices {in the} $r-$neighbourhoods. This is {significantly} stronger than local convergence of preferential attachment models. Further, when $(\lceil na_\omega \rceil)_{\omega\in V(\tree)}$ are distinct (which occurs whp), $\left( \hat{i}_{\lceil na_\omega\rceil} \right)_{\omega\in V(\tree)}$ are {\em conditionally independent} Gamma variables with parameters $m_{v_\omega}+\delta$ and $1$.

We prove Theorem~\ref{thm:prelimit:density} below in several steps. First we calculate the conditional density $f_{r,\tree}$. Using the equivalence of $\CPU$ and model (A) in Proposition~\ref{prop:equiv:CPU:PArs}, we compute the explicit expression in \eqref{eq:prelimit:density:1}. Let $\partial V(\tree)$ denotes the leaf nodes of the tree $\tree$ and $V^\circ(\tree)$ the set of vertices in the interior of the tree $\tree$, i.e., $V^\circ(\tree)= V(\tree)\setminus \partial V(\tree)$. {Then the age densities in $\RPPT(M,\delta)$ are identified as follows:}

\begin{Proposition}[Law of vertex-marked neighbourhood of RPPT]
	\label{prop:RPPT:density} 
	Let $M$ be the law of the out-degrees, and fix $\delta>-\infsupp(M)$. Let $n_{\old}$ and $n_{\young}$ denote the number of $\Old$ and $\Young$ labelled nodes in $\tree$ and $E(\tree)$ the edge-set of the tree $\tree$. Then,
	\eqan{\label{eq:prop:RPPT:density}
			f_{r,t}\left((a_\omega)_{\omega\in V(t)}\right)= &~i^{n_{\old}}(1-i)^{n_{\young}}\E\left[\prod\limits_{\omega\in V^\circ(t)}  {m_{-}(\omega)} !\, \Gamma_\omega^{d_{\omega}^{({\rm{in}})}(t)}\exp{\left(-\Gamma_\omega \lambda(a_\omega)\right)}\right.\nn\\
			&\hspace{4cm}\times\left.\prod\limits_{(\omega,\omega l)\in E(\tree)}{(a_\omega \vee a_{\omega l})^{-i}(a_\omega \wedge a_{\omega l})^{-(1-i)}}\right],}
	where $d_{\omega}^{({\rm{in}})}(t) = \#\{ \omega l:~ a_{\omega l}>a_\omega \}$ denotes the number of $\Young$ labelled children of $\omega,$ while $(\Gamma_\omega,~m_{-}(\omega))$ are distributed as in Section~\ref{sec:RPPT} and we recall from \eqref{lambda-def} that  $\lambda$ is a real-valued function on $(0,1)$ defined as 
	\eqn{\lambda(x)=\frac{1-x^{1-i}}{x^{1-i}}.}
\end{Proposition}

To prove {Proposition \ref{prop:RPPT:density},} we need to identify the densities of $(A_\omega)_{\omega\in V(\tree)}$. We start by analysing the density of the age of $\Old$ labelled nodes:

\begin{Lemma}[Conditional age density of $\Old$ labelled children]
	\label{lem:old:density} Conditionally on $a_\omega,$ the age of its parent $\omega\in V^\circ(\tree)$, the density of {an} $\Old$ labelled {child} on $[0,a_\omega]$ is given by
	\eqn{\label{eq:lem:old:density}
		f_{\old}(x)=i x^{-(1-i)}a_\omega^{-i}.}
\end{Lemma}

\begin{proof}
	From the construction of $\RPPT$, for any $\Old$ labelled node $\omega l$, its age is given by $U^{1/i}A_\omega$, where $U$ is {uniform} in $[0,1]$ independently of $A_\omega.$ Therefore, conditionally on $A_\omega=a_\omega,$
	\begin{align*}
		\prob\left( A_{\omega l}\leq x\mid A_\omega=a_\omega \right) =&~ \prob\left( A_\omega U^{1/i}\leq x\mid A_\omega=a_\omega \right) = \prob\left(U^{1/i}\leq x a_\omega^{-1}\right)= a_\omega^{-i} x^{i}.
	\end{align*}
	Now, differentiating with respect to $x,$ we obtain the conditional density of $A_{\omega l}$ in \eqref{eq:lem:old:density}.
\end{proof}

The ages of the $\Young$ labelled {children of $\omega$} follow an inhomogeneous Poisson process with intensity
\begin{equation}
	\label{intensity-younger-children}
	\rho_{\omega}(x) = (1-i)\Gamma_\omega\frac{x^{-i}}{A_\omega^{1-i}}.
\end{equation}
{This leads to the following density result on the ages of all children of $\omega$:}
\begin{Lemma}[Conditional age density of children]
	\label{lem:children:age:density}
	For any $\omega\in V^\circ(\tree),$ conditionally on $(m_{-}(\omega),A_\omega,\Gamma_\omega)$, the density of the ages of the children of $\omega$ is given by
	\eqan{\label{eq:lem:children:age:density}
			&f_\omega\left(\left.(a_{\omega l})_{l\in d_{\omega}(\tree)}\right|m_{-}(\omega),a_\omega,\Gamma_\omega \right) =m_{-}(\omega)!\prod\limits_{l=1}^{m_{-}(\omega)}\left[ i (a_{\omega l})^{i-1}(a_\omega)^{-i} \right]\exp{\left(-\Gamma_\omega\lambda(a_\omega)\right)}\nn\\
			&\hspace{6.5cm}\times\prod\limits_{k=1}^{d_{\omega}^{{\rm{(in)}}}(\tree)}\left[ (1-i)\left( a_{\omega (m_{-}(\omega)+k)} \right)^{-i}(a_\omega)^{i-1}\Gamma_\omega \right] ,}
	where {$m_{-}(\omega)$ is as defined in Section~\ref{sec:RPPT},} and $d_{\omega}(\tree)= m_{-}(\omega)+d_{\omega}^{\sss\rm{(in)}}(\tree)$ denotes the number of children of $\omega$.
\end{Lemma}
\begin{proof}
	{From the construction of the $\RPPT$, we see that any node $\omega$ has $m_{-}(\omega)$ many $\Old$ labelled children. Recall that $m_{-}(\omega)\sim M^{(\delta)}$ if $\omega$ has label $\Old$, whereas $m_{-}(\omega) \sim M^{(0)}$ if $\omega$ has label $\Young$, and $m_{-}(\emp)\sim M$}. Since the uniform random variables are chosen independently for obtaining the age of the $\Old$ labelled children, using Lemma~\ref{lem:old:density}, 
	\eqn{\label{for:lem:cad:1}
		f\left( \left.(a_{\omega l})_{l\in [m_{-}(\omega)]}\right|m_{-}(\omega),a_\omega,\Gamma_\omega \right) = m_{-}(\omega)!\prod\limits_{l=1}^{m_{-}(\omega)}\left[ i (a_{\omega l})^{i-1}(a_\omega)^{-i} \right].}
	{Indeed, since} there is no particular order for connecting to the older nodes, $\omega$ can connect to its older children with its edges in $m_{-}(\omega)!$ different ways.
	
	The label $\Young$ children have ages coming from a Poisson process with (random) intensity $x\mapsto\rho_\omega(x)$ on $\left[ a_\omega,1 \right]$ {defined in \eqref{intensity-younger-children}}. 
	Therefore for $k\geq2$, {conditionally on $a_{\omega (m_{-}(\omega)+k-1)},~\omega (m_{-}(\omega)+k)$} has an age following a non-homogeneous exponential distribution with intensity $x\mapsto \rho_\omega(x)$. 
		
	Additionally, there is one more factor arising in this part of the density, which is the no-further $\Young$ labelled child after $\omega d_\omega(t)$. Conditionally on $a_{\omega d_{\omega}(\tree)}$, this no-further child factor is given by
	\[
	\exp{\left(-\Gamma_\omega\frac{1-a_{\omega d_{\omega}(\tree)}^{1-i}}{a_\omega^{1-i}}\right)}.
	\]
	This no-further child part is also independent of the ages of the $\Young$ labelled nodes except for $a_{\omega d_\omega(\tree)}$, from the property of the Poisson process. Therefore,
	\eqn{\label{for:lem:cad:2}
		\begin{split}
			&f\left( \left. \left(a_{\omega (m_{-}(\omega)+l)}\right)_{l\in [d_\omega^{\sss\rm{(in)}}(\tree)]}\right|m_{-}(\omega),a_\omega,\Gamma_\omega \right)\\
			&= \exp{\left(-\Gamma_\omega\lambda(a_\omega)\right)}\prod\limits_{k=m_{-}(\omega)+1}^{d_{\omega}(\tree)}\left[ (1-i)\left( a_{\omega k} \right)^{-i}(a_\omega)^{i-1} \Gamma_\omega \right] .
	\end{split}}
	Since the $\Young$ labelled nodes connect one by one sequentially in a particular order, $\omega$ connects to its $\Young$ labelled children in only $1$ way. Now, from the construction, the edges to the $\Old$ labelled children and the edges to the $\Young$ labelled children are created independently (conditionally on $(m_{-}(\omega),A_\omega,\Gamma_\omega)$). Therefore from \eqref{for:lem:cad:1} and \eqref{for:lem:cad:2} and this independence,
	\eqan{\label{for:lem:cad:3}
		&f_\omega\left(\left.(a_{\omega l})_{l\in [d_{\omega}(\tree)]}\right|m_{-}(\omega),a_\omega,\Gamma_\omega \right) \nn\\
		= &f\left( \left.(a_{\omega l})_{l\in [m_{-}(\omega)]}\right|m_{-}(\omega),a_\omega,\Gamma_\omega \right) f\left( \left. \left(a_{\omega (m_{-}(\omega)+l)}\right)_{l\in [d_\omega^{\sss\rm{(in)}}(\tree)]}\right|m_{-}(\omega),a_\omega,\Gamma_\omega \right),}
	which leads to the required expression in \eqref{eq:lem:children:age:density}.
\end{proof}

Now we have the required tools to prove Proposition~\ref{prop:RPPT:density}, for which we use induction on  $r$, the depth of the tree:
\begin{proof}[ Proof of Proposition~\ref{prop:RPPT:density}]
	{For proving this proposition,} we define some notation \textbf{\textit{that is used in this proof only}}. Define $\tree_r$ to be the $r$-neighbourhood of the root in $\tree$. Therefore,
	\eqn{V^\circ(\tree_{r+1})=V^\circ(\tree_r)\cup \partial V(\tree_r). \nonumber}
	First, we prove the proposition for $r=1$. Here $V^\circ(\tree_1)=\{\emp\}$. By Lemma~\ref{lem:children:age:density},
	\eqan{\label{for:prop:RPPTd:1}
			&f_{\emp}\left(\left.(a_{\emp l})_{l\in [d_{\emp}(\tree)]}\right|m_{-}(\emp),a_\emp,\Gamma_\emp \right)= {m_{-}(\emp)}!\prod\limits_{l=1}^{m_{-}({\emp})}\left[ i (a_{\emp l})^{i-1}(a_\emp)^{-i} \right]\nn\\
			&\hspace{3.5cm}\times\prod\limits_{k=1}^{d_{\emp}^{{\rm{(in)}}}(\tree)}\left[ (1-i)\left( a_{\emp (m_{-}({\emp})+k)} \right)^{-i}(a_\emp)^{i-1}\Gamma_\emp \right] \exp{\left(-\Gamma_\emp\lambda(a_\emp)\right)}.}
	Now, observe that for every edge to an $\Old$ labelled child, we obtain a factor $i$ and for every edge to $\Young$ labelled children, we obtain a factor $(1-i)$ and a factor $\Gamma_\emp$ in \eqref{for:prop:RPPTd:1}. Further note that $\Old$ labelled children of $\emp$ have smaller age than $a_\emp$ and $\Young$ labelled children have higher age than $a_\emp$. $\tree_1$ has ${m_{-}(\emp)}$ many $\Old$ labelled nodes and $d_{\emp}^{\sss\rm{(in)}}(\tree)$ many $\Young$ labelled nodes. Therefore \eqref{for:prop:RPPTd:1} can be rewritten as
	\eqan{\label{for:prop:RPPTd:2}
			f_{\emp}\left(\left.(a_{\emp l})_{l\in {[d_{\emp}(\tree)]}}\right|{m_{-}(\emp)},a_\emp,\Gamma_\emp \right) =& i^{n_{1,\old}}(1-i)^{n_{1,\young}}{m_{-}(\emp)}!~ \Gamma_\emp^{d_{\emp}^{\sss\rm{(in)}}(\tree)}\exp{\left(-\Gamma_\emp\lambda(a_\emp)\right)}\nn\\
			&\hspace{0.7cm}\times\prod\limits_{(\emp,\emp l)\in E(\tree_1)}\left[ \left( a_\emp\vee a_{\emp l} \right)^{-i}\left( a_\emp\wedge a_{\emp l} \right)^{-(1-i)} \right],}
	 where $n_{1,\old}$ and $n_{1,\young}$ denote the number of $\Old$ and $\Young$ labelled nodes in $\tree_1$. Since $A_\emp$ has a uniform distribution on $[0,1]$, and is independent of $m_\emp$ and $\Gamma_\emp$,
	\eqan{\label{for:prop:RPPTd:3}
			f_{1,\tree}\left( (a_\omega)_{\omega\in V(t)} \right) =&~ \E\left[ f_{1,\tree}\left(\left.a_\emp,(a_{\emp l})_{l\in {[d_{\emp}(\tree)]}}\right|{m_{-}(\emp)},\Gamma_\emp \right) \right]\nn\\
			=&~\E\left[f_{\emp}\left(\left.(a_{\emp l})_{l\in {[d_{\emp}(\tree)]}}\right|{m_{-}(\emp)},a_\emp,\Gamma_\emp \right)\right]\nn\\
			=& i^{n_{1,\old}}(1-i)^{n_{1,\young}} \E\left[ {m_{-}(\emp)!}\Gamma_\emp^{d_{\emp}^{\sss\rm{(in)}}(\tree)}\exp{\left(-\Gamma_\emp\lambda(a_\emp)\right)} \right.\\
			&\hspace{3cm} \times \prod\limits_{(\emp,\emp l)\in E(\tree_1)}\left[ \left( a_\emp\vee a_{\emp l} \right)^{-i}\left( a_\emp\wedge a_{\emp l} \right)^{-(1-i)} \right] \Bigg].\nn
		}
	The first equality comes from the fact that $V^\circ(\tree_1)=\emp$ and hence the proposition is proved for $r=1.$ {We} proceed toward the induction step. Let \eqref{prop:RPPT:density} be true for $r=k\in\N$. We wish to show that the result holds true for $r=k+1$. 
	
	We have the distribution for the ages of the nodes in $V(\tree_k).$ What remains is to compute the density of the boundary conditionally on the ages of $V(\tree_k)$. Now, the nodes in $\partial V(\tree_{k+1})$ are the children of the nodes in $\partial V(\tree_k)$ and their age-distribution is independent of the age of nodes in $V^\circ (\tree_k)$. {By} Lemma~\ref{lem:children:age:density},
	\eqan{\label{for:prop:RPPTd:4}
			&f\left( \left.(a_\omega)_{\omega\in\partial V(\tree_{k+1})}\right|\left( {m_{-}(u)},a_u,\Gamma_u \right)_{u\in V(\tree_k)} \right)\nn\\
			=&f\left( \left.(a_{\omega l})_{\omega l\in\partial V(\tree_{k+1})}\right|\left( {m_{-}(\omega)},a_\omega,\Gamma_\omega \right)_{\omega\in \partial V(\tree_k)} \right)\\
			=& i^{n_{k+1,\old}} (1-i)^{n_{k+1,\young}}\prod\limits_{\omega\in \partial V(\tree_k)} \left[ {m_{-}(\omega)}!\, \Gamma_{\omega}^{d_{\omega}^{\sss\rm{(in)}}(\tree_{k+1})}\exp{\left( -\Gamma_\omega\lambda(a_\omega) \right)}\right.\nn\\
			&\hspace{3 cm}\times\prod\limits_{(\omega, \omega l)\in E(\tree_{k+1})} \left( a_\omega \vee a_{\omega l} \right)^{-i} \left( a_\omega \wedge a_{\omega l} \right)^{-(1-i)} \Big],\nn}
		where $n_{k+1,\old}$ and $n_{k+1,\young}$ denote the number of $\Old$ labelled and $\Young$ labelled nodes in $\partial V(\tree_{k+1})$. On the other hand, $\tree_k$ is a rooted tree of depth $k$. Therefore, using induction on $r$, the depth of the tree from the root
		\eqan{\label{for:prop:RPPTd:5}
				f\left( (a_\omega)_{\omega\in V(\tree_k)} \right) = &~i^{n_{[k],\old}}(1-i)^{n_{[k],\young}}\E\left[\prod\limits_{\omega\in V^\circ(\tree_k)} { m_{-}(\omega)} !\, \Gamma_\omega^{d_{\omega}^{({\rm{in}})}(t)}\exp{\left(-\Gamma_\omega\lambda(a_\omega)\right)}\right.\nn\\
				&\qquad\qquad\qquad\qquad\qquad\times\left.\prod\limits_{(\omega,\omega l)\in E(\tree_k)}{(a_\omega \vee a_{\omega l})^{-i}(a_\omega \wedge a_{\omega l})^{-(1-i)}}\right], }
		where $n_{[k],\old}$ and $n_{[k],\young}$ denote the number of $\Old$ labelled and $\Young$ labelled nodes in $V(\tree_k)$. Since $\tree$ is tree, $d_{\omega}^{\sss\rm{(in)}}(\tree_{k+1})=d_{\omega}^{\sss\rm{(in)}}(\tree)$. Moreover $n_{[k],\old}+n_{k+1,\old} =n_{[k+1],\old}$ and  $n_{[k],\young}+n_{k+1,\young} =n_{[k+1],\young}$ and $\left( {m_{-}(\omega)},\Gamma_\omega \right)_{\omega\in V(\tree)}$ are independent random variables. Therefore from \eqref{for:prop:RPPTd:4} and \eqref{for:prop:RPPTd:5}, we obtain the required result for $r=k+1$ as
		\eqan{
				f_{k+1,\tree}\left( (a_\omega)_{\omega\in V(\tree)} \right) = 			&~i^{n_{[k+1],\old}}(1-i)^{n_{[k+1],\young}}\E\left[\prod\limits_{\omega\in V^\circ(\tree)} {m_{-}(\omega)}!\, \Gamma_\omega^{d_{\omega}^{({\rm{in}})}(t)}\exp{\left(-\Gamma_\omega\lambda(a_\omega)\right)}\right.\nn\\
				&\qquad\qquad\qquad\qquad\times\left.\prod\limits_{(\omega,\omega l)\in E(\tree_{k+1})}{(a_\omega \vee a_{\omega l})^{-i}(a_\omega \wedge a_{\omega l})^{-(1-i)}}\right].}
		The general claim follows by induction in $r.$
	\end{proof}
	
	By Proposition~\ref{prop:RPPT:density}, we have the exact expression for the density of the ages of the nodes {in} the $\RPPT$. To prove Theorem~\ref{thm:prelimit:density}, we compute the expression for $g_{r,\tree}$. Instead of finding $g_{r,\tree}$ directly, we use {Theorem~\ref{thm:equiv:CPU:PArs}}. {To make} our computation simpler, we first introduce edge-marks in the tree and then lift the edge-marks carefully.
	
	Let $\overline{\rm{t}}=\left( \tree, (a_\omega)_{\omega\in V(\tree)}, (e_{\omega,\omega j})_{(\omega,\omega j)\in E(\tree)} \right)$ be the edge-marked version of $\tree$, where $E(\tree)$ is the edge-set of $\tree$. 
	We write $\overline{B}_r^{\sss(G_n)}(v)\doteq\overline{\tree}$, to denote that the vertex and edge-marks of the $r$-neighbourhood of the vertex $v$ in $\CPU_n(\boldsymbol{m},\boldsymbol{\psi})$ are given by those in $\tree$.
	
	\begin{Proposition}[Density of vertex and edge-marked $\CPU$]\label{prop:CPU:density}
		Let $o_n$ be a uniformly chosen vertex from $\CPU_n^{\sss\rm{(SL)}}(\boldsymbol{m},\boldsymbol{\psi})$. Then,
		\eqan{\label{eq:prop:CPU:density:1}
				&\prob_{m,\psi}\left( \overline{B}_r^{\sss(G_n)}({o_n})\doteq\overline{\tree},~v_\omega=\lceil na_\omega \rceil,~\forall \omega\in V(\tree) \right)\nn\\
				=~&(1+o_\prob(1))n^{-|V(\tree)|}\prod\limits_{\omega\in V^\circ(\tree)}\left[\left( \frac{\hat{i}_{v_\omega}}{2\E[M]+\delta} \right)^{d_\omega^{\sss\rm{(in)}}(\tree)} \exp{\left( -\hat{i}_{v_\omega} \lambda(a_\omega) \right)}\right]\\
				&\hspace{2cm}\times\prod\limits_{(\omega,\omega l)\in E(\tree)}\left( a_{\omega}\vee a_{\omega l} \right)^{-i}\left( a_{\omega}\wedge a_{\omega l} \right)^{-(1-i)}\prod\limits_{\substack{\omega\in V(\tree)\\\omega \text{ is $\Old$ }}}\frac{\hat{i}_{v_\omega}}{2\E[M]+\delta}~.\nn}
	\end{Proposition}
	
	Before we prove Proposition \ref{prop:CPU:density} we begin with a lemma that gives an estimate of the edge-connection probabilities and this estimate will be useful in the proof of the above proposition. The estimate is given in terms of the coupled Gamma random variables $(\hat{i}_v)_{v\geq 2}$ from Proposition \ref{prop:Betagamma:coupling:CPU}.
	\begin{Lemma}
		\label{lem:NFE:main}
		Conditionally on $(\boldsymbol{m,\psi}),$ for all $\omega\in V^\circ(\tree)$,
		\eqn{\label{eq:lem:NFE:main}
			\sum\limits_{u,j\colon u\geq v_\omega} p^{\sss(j)}(u,v_\omega) = (1+o_\prob(1))\hat{i}_{v_\omega}\lambda(a_\omega)~,}
    {where $p^{\sss(j)}(u,v_\omega)$ is the conditional probability of connecting the $j${\rm th} out-edge of $u$ to $v_\omega$.}
	\end{Lemma}
	\begin{proof}
		For every $u$, there are $m_u$ many out-edges from $u$ and using the expression for the edge-connection probabilities in Lemma~\ref{lem:bound:attachment_probability},
		\eqan{\label{for:lem:NFE:main:1}
				\sum\limits_{u,j:u\geq v_\omega} p^{\sss(j)}(u,v_\omega)=&~ (1+o_\prob(1))\sum\limits_{u\geq v_\omega} m_u \frac{\hat{i}_{v_\omega}}{(2\E[M]+\delta)v_\omega}\left( \frac{v_\omega}{u} \right)^{i}\nn\\
				=&~(1+o_\prob(1))\frac{\hat{i}_{v_\omega}}{(2\E[M]+\delta)v_\omega^{1-i}}\sum\limits_{u\geq v_\omega} m_u u^{-i}\\
				=&~ (1+o_\prob(1))\frac{\hat{i}_{v_\omega}}{(2\E[M]+\delta)\left(\frac{v_\omega}{n}\right)^{1-i}}\Big[\frac{1}{n}\sum\limits_{u\geq v_\omega} m_u \left( \frac{u}{n} \right)^{-i}  \Big].\nn}
		Define $T_n:=~ \frac{1}{n}\sum\limits_{u\geq v_\omega} m_u \left( \frac{u}{n} \right)^{-i}.$ With $v_\omega=\lceil na_\omega \rceil$, we aim to show that 
		\eqn{
			\label{WSLLN}
			T_n\overset{\prob}{\to} \E[M]\int\limits_{a_{\omega}}^1 z^{-i}\,dz
		}
		{using a {\em weighted}} strong law of large numbers. Let $(X_i)_{i\geq 1}$ be a sequence of i.i.d.\ random variables with finite mean. Then from \cite[Theorem 5]{choisung87}, a sufficient condition for $\sum\limits_{i=1}^n a_{i,n} X_i~\left(\text{with }\sum\limits_{i=1}^n a_{i,n}=1\right)$ to converge to $\E[X]$ is {that $\max\limits_{i\in [n]} ~a_{i,n} = O(1/n).$}
		
		{In} our case, consider $a_{u,n} = \frac{\frac{1}{n}\left( \frac{u}{n} \right)^{-i}}{\frac{1}{n}{\sum\limits_{z\geq v_\omega}\left( \frac{z}{n} \right)^{-i}}},$ {so that}
		\eqn{\label{for:lem:NFE:main:2}
			{\max\limits_{u\in [v_\omega,n]} \frac{\frac{1}{n}\left( \frac{u}{n} \right)^{-i}}{\frac{1}{n}\sum\limits_{z\geq v_\omega}\left( \frac{z}{n} \right)^{-i}}\leq \frac{a_{\omega}^{-i}}{n(1-a_\omega^{1-i})}=~O\left( \frac{1}{n} \right)}  ~.}
		Therefore, the weighted strong law {implies that \eqref{WSLLN} holds.} Note that the denominator here is the Riemann sum approximation of $\int\limits_{a_\omega}^1 t^{-i}\,dt$, {so that}
		\eqn{\label{for:lem:NFE:main:3}
			{T_n=(1+o_\prob(1))\E[M]\int\limits_{a_\omega}^1 z^{-i}\,dz=~(1+o_\prob(1))\frac{\E[M]}{1-i}\left[ 1-a_\omega^{1-i} \right].}}
		Hence, from \eqref{for:lem:NFE:main:1} and \eqref{for:lem:NFE:main:3},
		\eqn{\begin{split}
				\sum\limits_{u,j:u\geq v_\omega} p^{\sss(j)}(u,v_\omega)=&~ (1+o_\prob(1))\frac{\hat{i}_{v_\omega}}{(2\E[M]+\delta)a_{\omega}^{1-i}} \frac{\E[M]}{1-i}\left[ 1-a_\omega^{1-i} \right]\\
				=&~ (1+o_\prob(1))\hat{i}_{v_\omega}\lambda(a_\omega),
		\end{split}}
		as required.
	\end{proof}
	
	\begin{proof}[ Proof of Proposition~\ref{prop:CPU:density}]
		The proof of Proposition~\ref{prop:CPU:density} is divided into several steps. Recall that { Lemma~\ref{lem:bound:attachment_probability} yields} the edge-connection probabilities for $\CPU$. 
		\paragraph{Computing the conditional law of $\overline{B}_r^{\sss(G_n)}(o_n)$}
		{The construction of the $\CPU$ implies the conditional independence of the} edge-connection events. Conditionally on $(\boldsymbol{m,\psi})$, let $p^{\sss(j)}(v,u)$ denote the probability of connecting the \(j\)-th edge from $v$ to $u$. {Then,}
		\eqan{\label{eq:conditional_law:CPU:1}
				&\prob_{m,\psi}\left( \overline{B}_r^{\sss(G_n)}(o_n)\doteq\overline{\tree},~v_\omega=\lceil na_\omega \rceil,~\forall \omega\in V(\tree) \right)\nn\\
				=~&\frac{1}{n}\prod\limits_{u\in V(\tree)}\prod\limits_{\substack{\omega\in V(\tree)\\{a_{\omega}>a_{u}}\\\omega\overset{j}{\sim} u}} p^{\sss(j)}(v_\omega,v_u) 
				{\prod\limits_{u\in[n]}\prod\limits_{j\in[m_{u}]}\Bigg[ 1-\sum\limits_{\substack{\omega\in V^{\circ}(\tree)\\ (\frac{u}{n},j,\omega)\notin E(\overline{\tree})\\ u>v_{\omega}}}p^{\sss(j)}(u,v_\omega) \Bigg]}.
		}
		The factor of ${1/n}$ arises due to the uniform choice of the root and the first product comprises all edge-connection probabilities to make sure that the edges in $\overline{\tree}$ are there in $\overline{B}_r^{\sss(G_n)}$ keeping the edge-marks the same. The last product ensures that there is no further edge in the $(r-1)$-neighbourhood of the randomly chosen vertex $o_n$, so that {the vertex and edge-marks of $\overline{B}_r^{\sss(G_n)}(o_n)$ are same as those in $\overline{\tree}$.}
		\smallskip
		\paragraph{The no-further edge probability}
		We continue by analyzing the second product on the RHS of \eqref{eq:conditional_law:CPU:1} which, for simplicity, we call the \textit{no-further edge probability}. Observe that in this part of the expression, we have not included the edges that are connected in $\tree,$ i.e., we exclude those factors $\left[ 1-p^{\sss(j)}(v_u,v_\omega) \right]$ for which $\omega\in V^\circ(\tree)$ and $u\overset{j}{\rightsquigarrow}\omega$ in $\tree$. Recall from Lemma~\ref{lem:RPPT:property} that the minimum age of the vertices in $\tree$ is $\eta$ whp. Therefore by Lemma~\ref{lem:bound:attachment_probability}, for all $u>\eta n,~p^{\sss(j)}(u,v)$ is $o_\prob(1).$ Since $\tree$ is finite, we {include} a finite number of $\left[ 1-p^{\sss(j)}(v_u,v_\omega) \right]$ factors and hence we can approximate the \textit{no-further edge probability} as
		\eqan{\label{eq:NFE:1}
			&{\prod\limits_{u\in[n]}\prod\limits_{j\in[m_{u}]}\Bigg[ 1-\sum\limits_{\substack{\omega\in V^{\circ}(\tree)\\ (\frac{u}{n},j,\omega)\notin E(\overline{\tree})\\ u>v_{\omega}}}p^{\sss(j)}(u,v_\omega) \Bigg]}\nn\\
			=&~{(1+o_\prob(1))\prod\limits_{u\in[n]}\prod\limits_{j\in[m_{u}]}\prod\limits_{\substack{\omega\in V^{\circ}(\tree)\\ (\frac{u}{n},j,\omega)\notin E(\overline{\tree})\\ u>v_{\omega}}}\left[ 1-p^{\sss(j)}(u,v_\omega) \right]}\nn\\
			=&~(1+o_\prob(1))\prod\limits_{{\omega\in V^{\circ}(\tree)}}\prod\limits_{u>v_{\omega}}\prod\limits_{j\in[m_{u}]}\left[ 1-p^{\sss(j)}(u,v_\omega) \right].}
We can approximate
		\eqan{\label{eq:NFE:2}
			&\prod\limits_{\substack{u,j:\\u\geq v_\omega}}\Big[ 1-p^{\sss(j)}(u,v_\omega) \Big] \nn\\
			=& \exp{\Big(\Theta(1)\sum\limits_{u,j\colon u\geq v_\omega}p^{\sss(j)}(u,v_\omega)^2\Big)}\exp{\Big( -\sum\limits_{u,j:u\geq v_\omega} p^{\sss(j)}(u,v_\omega) \Big)}.}
		We next investigate the first term in the RHS of \eqref{eq:NFE:2}, {while} the second, and the main, term is estimated using Lemma~\ref{lem:NFE:main}. 
		We prove that the first exponential term in \eqref{eq:NFE:2} is $(1+o_\prob(1))$. Using Lemma~\ref{lem:bound:attachment_probability}, similarly as in \eqref{for:lem:NFE:main:1},
		\eqn{\label{eq:NFE:3}
			\begin{split}
				\sum\limits_{u,j:u\geq v_\omega} p^{\sss(j)}(u,v_\omega)^2 =& (1+o_\prob(1))\sum\limits_{u\geq v_\omega} m_u \frac{\hat{i}_{v_\omega}^2}{(2\E[M]+\delta)^2v_\omega^2}\left( \frac{v_\omega}{u} \right)^{2i} \\
				\leq& (1+o_\prob(1))\frac{\hat{i}_{v_\omega}^2}{(2\E[M]+\delta)^2\left({v_\omega}\right)^{2-i}}\Big[ \sum\limits_{u\geq v_\omega} m_u\left( \frac{1}{u} \right)^{i} \Big]  \\
				=& (1+o_\prob(1))\frac{(\hat{i}_{v_\omega}^2/n)}{(2\E[M]+\delta)^2\left(\frac{v_\omega}{n}\right)^{2-i}}\Big[ \frac{1}{n}\sum\limits_{u\geq v_\omega} m_u\left( \frac{u}{n} \right)^{-i} \Big].
		\end{split}}
		By \eqref{for:lem:NFE:main:3} and recalling that $v_\omega=\lceil na_\omega \rceil$,
		\[
		\frac{1}{n} \sum\limits_{u\geq v_\omega} m_u\left( \frac{u}{n} \right)^{-i} \overset{a.s.}{\to} \frac{\E[M]}{1-i}\left[ 1-a_\omega^{1-i} \right]<C,
		\]
		for some constant $C>0$. It is enough to show that ${\hat{i}_{v_\omega}^2/n}=o_\prob(1)$. {For this, we fix $\varepsilon,\zeta>0$, and {note} that, for sufficiently large $n$,}
		\eqn{\label{eq:NFE:4}
			\begin{split}
				\prob\left( \frac{\hat{i}_{v_\omega}^2}{n}\geq \zeta \right)=\prob\left( \hat{i}_{v_\omega}^{p}\geq (n\zeta)^{p/2} \right)\leq {\E\left[ \hat{i}_{v_\omega}^{p} \right]}(n\zeta)^{-p/2}\leq \varepsilon.
		\end{split}}
		Therefore,  $\sum\limits_{u,j:u\geq v_\omega} p^{\sss(j)}(u,v_\omega)^2 = o_\prob(1)$ and, by \eqref{eq:NFE:2} and Lemma~\ref{lem:NFE:main},
		\eqn{\label{eq:NFE:5}
			\prod\limits_{\substack{u,j:\\u\geq v_\omega}}\left[ 1-p^{\sss(j)}(u,v_\omega) \right] = (1+o_\prob(1))\exp{\left( -\hat{i}_{v_\omega}\lambda(a_\omega) \right)}.}
		\smallskip
		\paragraph{Conclusion of the proof}
		{Substituting the} \textit{no-further edge probability} obtained in \eqref{eq:NFE:5} and the conditional probability estimates obtained in {Lemmas}~\ref{lem:bound:attachment_probability} and \ref{lem:NFE:main}, we obtain
		\eqan{\label{eq:conditional_law:CPU:2}
				&\prob_{m,\psi}\left( \overline{B}_r^{\sss(G_n)}\doteq\overline{\tree},~v_\omega=\lceil na_\omega \rceil,~\forall \omega\in V(\tree) \right)\nn\\
				=~& (1+o_\prob(1))\frac{1}{n}\prod\limits_{\omega\in V(\tree)}\prod\limits_{\substack{u\in V(\tree)\\{a_{u}>a_{\omega}}\\u\overset{j}{\sim}\omega}}\frac{\hat{i}_{v_\omega}}{(2\E[M]+\delta)v_\omega}\left( \frac{v_\omega}{v_u} \right)^{i}\prod\limits_{\omega\in V^\circ(\tree)}\exp{\left( -\hat{i}_{v_\omega}\lambda(a_\omega) \right)}\\
				=~& (1+o_\prob(1))\frac{1}{n}\prod\limits_{\omega\in V(\tree)}\left(\frac{\hat{i}_{v_\omega}}{(2\E[M]+\delta)}\right)^{d_{v_\omega}^{\sss\rm{(in)}}(G_n)}\prod\limits_{\omega\in V^\circ(\tree)}\exp{\left( -\hat{i}_{v_\omega}\lambda(a_\omega) \right)}\nn\\
				&\hspace{5cm}\times\prod\limits_{(\omega,\omega l)\in E(\tree)} {\left( v_\omega\vee v_{\omega l} \right)^{-i}\left( v_\omega\wedge v_{\omega l} \right)^{-(1-i)}},\nn}
		where $d_{v_{\omega}}^{\sss\rm{(in)}}(G_n) = d(v_\omega)-m_{v_\omega}$ is the number of vertices in $\CPU$ connected to $v_\omega$. Observe that $d_{v_\omega}^{\sss\rm{(in)}}(G_n)=d_{\omega}^{\sss\rm{(in)}}(\tree)$ when $\omega$ has label $\Young$, and $d_{v_\omega}^{\sss\rm{(in)}}(G_n)=d_{\omega}^{\sss\rm{(in)}}(\tree)+1$ when $\omega$ has label $\Old$. Further, $d_{v_{\omega}}^{\sss\rm{(in)}}(G_n) = 1$ for $\omega\in \partial V(\tree)$ {and label $\Old$.} Therefore, \eqref{eq:conditional_law:CPU:2} can be re-written as
		\eqan{\label{eq:conditional_law:CPU:3}
				&\prob_{m,\psi}\left( \overline{B}_r^{\sss(G_n)}\doteq\overline{\tree},~v_\omega=\lceil na_\omega \rceil,~\forall \omega\in V(\tree) \right)\nn\\
				=~& (1+o_\prob(1))n^{-(1+|E(\tree)|)}\prod\limits_{\omega\in V^\circ(\tree)}\left(\frac{\hat{i}_{v_\omega}}{(2\E[M]+\delta)}\right)^{d_{\omega}^{\sss\rm{(in)}}(\tree)}\prod\limits_{\omega\in V^\circ(\tree)}\exp{\left( -\hat{i}_{v_\omega}\lambda(a_\omega) \right)}\nn\\
				&\hspace{3.5cm}\times\prod\limits_{\substack{\omega\in V(\tree)\\\omega\text{ is $\Old$ }}}\frac{\hat{i}_{v_\omega}}{(2\E[M]+\delta)}\prod\limits_{(\omega,\omega l)\in E(\tree)} {\left( a_\omega\vee a_{\omega l} \right)^{-i}\left( a_\omega\wedge a_{\omega l} \right)^{-(1-i)}}.}
		Since $\tree$ is a tree, $|V(\tree)|=1+|E(\tree)|$ and hence \eqref{eq:conditional_law:CPU:3} leads to \eqref{eq:prop:CPU:density:1}.
	\end{proof}
		
	\begin{Remark}[Density of vertex-marked $\CPU$]\label{rem:density:CPU}
		{\rm Every vertex $\omega\in V^\circ(\tree)$ has $m_{v_\omega}$ many out-edges that can be marked in $m_{v_\omega}!$ different ways.} If we {sum \eqref{eq:prop:CPU:density:1} out over} the edge-marks, {then} all such {edge-marked} graphs produce the same vertex-marked graph. 
		Observe that we have not considered any of the out-edges from the $\Old$ labelled vertices in $\partial V(\tree)$. Exactly one out-edge from every $\Young$ labelled {vertex} in $\partial V(\tree)$ is considered and its edge-mark can be labelled in $m_{v_\omega}$ many possible ways. Therefore by summing out the edge-marks,
		\eqan{\label{eq:prop:CPU:density:2}
				&\prob_{m,\psi}\left( {B}_r^{\sss(G_n)}(o)\simeq{\tree},~v_\omega=\lceil na_\omega \rceil,~\forall \omega\in V(\tree) \right)\nn\\
				&\qquad=~(1+o_\prob(1))n^{-|V(\tree)|}\prod\limits_{\omega\in V^\circ(\tree)}\left[m_{v_\omega}!\left( \frac{\hat{i}_{v_\omega}}{2\E[M]+\delta} \right)^{d_\omega^{\sss\rm{(in)}}(\tree)} \exp{\left( -\hat{i}_{v_\omega} \lambda(a_\omega) \right)}\right]\nn\\
				&\quad\qquad\qquad\times\prod\limits_{(\omega,\omega l)\in E(\tree)}\left( a_{\omega}\vee a_{\omega l} \right)^{-i}\left( a_{\omega}\wedge a_{\omega l} \right)^{-(1-i)}\prod\limits_{\substack{\omega\in V(\tree)\\\omega \text{ is $\Old$ }}}\frac{\hat{i}_{v_\omega}}{2\E[M]+\delta}\prod\limits_{\substack{\omega\in\partial V(\tree)\\\omega \text{ is ${\Young}$ }}}m_{v_{\omega}}~.}
		Conditionally on $( \boldsymbol{m},\boldsymbol{\psi} ),$ we have obtained the density of the $r$-neighbourhood of a randomly chosen vertex of $\PArs_n(\boldsymbol{m},\delta)$. This will be useful in the proof of Theorem~\ref{thm:prelimit:density} below.
		\hfill$\blacksquare$
	\end{Remark}
	
	From Remark~\ref{rem:density:CPU}, we have explicitly obtained the age density of the $\CPU$ graphs. Now we aim to show that this age density of $\CPU$ converges to that of the $\RPPT$.
	\begin{proof}[ Proof of Theorem~\ref{thm:prelimit:density}]
		We have obtained the explicit form of $f_{r,\tree}\left( (a_\omega)_{\omega\in V(\tree)}; (\hat{i}_{v_\omega})_{\omega\in V(\tree)} \right)$ in Proposition~\ref{prop:RPPT:density}. Remark~\ref{rem:density:CPU} gives us an expression for the LHS of \eqref{eq:prelimit:density:1}. {Thus,} we are left to show that \eqan{\label{for:thm:prelim:density:1}
				&\E\left[ \prod\limits_{\omega\in V^\circ(\tree)}\left[m_{v_\omega}!\left( \frac{\hat{i}_{v_\omega}}{2\E[M]+\delta} \right)^{d_\omega^{\sss\rm{(in)}}(\tree)} \exp{\left( -\hat{i}_{v_\omega} \lambda(a_\omega) \right)}\right]\prod\limits_{\substack{\omega\in V(\tree)\\\omega \text{ is $\Old$ }}}\frac{\hat{i}_{v_\omega}}{2\E[M]+\delta} \prod\limits_{\substack{\omega\in\partial V(\tree)\\\omega \text{ is $\Young$}}}m_{v_{\omega}}\right]\nn\\
				&\hspace{2cm}= i^{n_{\old}}(1-i)^{n_{\young}}\E\left[\prod\limits_{\omega\in V^\circ(t)} {m_{-}(\omega)^\prime}! ~\left(\hat{i}_\omega^\prime\right)^{d_{\omega}^{({\rm{in}})}(t)}\exp{\left(-\hat{i}_\omega^\prime\lambda(a_{\omega})\right)}\right],
			}
		where $\hat{i}_\omega^\prime\sim{\rm{Gamma}}\left({m_{-}(\omega)}^\prime+\delta+\one_{\{\omega\text{ is $\Old$}\}},1\right)$ and $\hat{i}_{v_{\omega}}\sim {\rm{Gamma}}\left( m_{v_\omega}+\delta,1 \right)$. The distribution of ${m_{-}(\omega)}^\prime$ is same as $M^{(\delta)}$ when $\omega$ has label $\Young$, $M^{(0)}$ when $\omega$ has label $\Old$ and, $M$ when $\omega$ is the root. $m_{v_\omega}$ are i.i.d.\ {copies of $M$.} To prove \eqref{for:thm:prelim:density:1}, we start by simplifying the LHS.
		
		First we rewrite
		\eqn{\label{for:thm:prelim:density:2}
			\prod\limits_{\substack{\omega\in V(\tree)\\\omega\text{ is $\Old$ }}}\frac{\hat{i}_{v_\omega}}{2\E[M]+\delta}=\prod\limits_{\substack{\omega\in V(\tree)\\\omega\text{ is $\Old$ }}}\left( \frac{\hat{i}_{v_\omega}}{m_{v_{\omega}}+\delta} \right)\left( \frac{m_{v_\omega}+\delta}{\E[M]+\delta} \right)\left( \frac{\E[M]+\delta}{2\E[M]+\delta} \right).}
		Remember that $i=\frac{\E[M]+\delta}{2\E[M]+\delta}$. Therefore from \eqref{for:thm:prelim:density:2}, we see that we obtain a factor $i$ for each of the $\Old$ labelled vertices in $\tree$. The first two terms in RHS of \eqref{for:thm:prelim:density:2} give rise to the size-biasing of the Gamma random variables and the out-edge distribution of the $\Old$ labelled vertices, as we observed in the definition of $\RPPT$. Now, for $\omega\in\partial V(\tree)$, there is no other term in the LHS of \eqref{for:thm:prelim:density:1} containing $\hat{i}_{v_\omega}$ and $m_{v_{\omega}}$ and these are independent of the rest. Therefore taking expectation with respect to these $\hat{i}_{v_\omega}$ and $m_{v_{\omega}}$, the first two terms {in \eqref{for:thm:prelim:density:2}} turn out to be $1$. 
		Next, we rewrite
		\eqan{\label{for:thm:prelim:density:3}
				\prod\limits_{\substack{\omega\in V^\circ(\tree)\\\omega\text{ is $\Young$ }}}m_{v_\omega}!\prod\limits_{\substack{\omega\in\partial V(\tree)\\\omega \text{ is $\Young$ }}}& m_{v_{\omega}}
				= \prod\limits_{\substack{\omega\in V^\circ(\tree)\\\omega\text{ is $\Young$ }}}\left(m_{v_\omega}-1\right)!\prod\limits_{\omega \text{ is $\Young$ }}m_{v_{\omega}}\nn\\
				=& \prod\limits_{\substack{\omega\in V^\circ(\tree)\\\omega\text{ is $\Young$ }}}\left(m_{v_\omega}-1\right)!\prod\limits_{\omega \text{ is $\Young$ }}\left( \frac{m_{v_{\omega}}}{\E[M]} \right)\left( \frac{\E[M]}{2\E[M]+\delta} \right)(2\E[M]+\delta)^{n_{\young}}.}
		Again {by} definition, $1-i=\frac{\E[M]}{2\E[M]+\delta}$. Therefore, similarly as in \eqref{for:thm:prelim:density:2}, we see that the size-biased out-degree distributions of the $\Young$ labelled vertices arise from the second term in \eqref{for:thm:prelim:density:3} and for every $\Young$ labelled vertex we obtain a factor $1-i$ as we observe in the $\RPPT$. There is no size-biasing in $i_{v_\emp}$ and $m_{v_\emp}$. The vertices in $V^\circ(\tree)$ can be partitioned in $3$ sets: the root, $\Old$ labelled vertices and $\Young$ labelled vertices. Therefore, using the simplification of expressions in \eqref{for:thm:prelim:density:2} and \eqref{for:thm:prelim:density:3},
		\eqan{\label{for:thm:prelim:density:4}
			\mbox{LHS of \eqref{for:thm:prelim:density:1}}
				=& ~i^{n_{\old}}(1-i)^{n_{\young}}(2\E[M]+\delta)^{n_{\young}-\sum\limits_{\omega\in V^\circ(\tree)} d_\omega^{\sss\rm{(in)}}(\tree)}\nn\\
				&\qquad\times\E\left[ \prod\limits_{\omega\in V^{\circ}(\tree)}{m_{-}(\omega)}^\prime!\left( {\hat{i}_{\omega}^\prime} \right)^{d_\omega^{\sss\rm{(in)}}(\tree)} \exp{\left( -\hat{i}_{\omega}^\prime \lambda(a_\omega) \right)} \right],}
		where $\hat{i}_{\omega}^\prime$ is defined as before.
		Note that $d_\omega^{\sss\rm{(in)}}(\tree)$ denotes the number of $\Young$ labelled children of $\omega$ and therefore summing {$d_\omega^{\sss\rm{(in)}}(\tree)$ over} all $\omega\in V^\circ (\tree)$ gives the total number of $\Young$ labelled vertices in $\tree$ and hence the terms in the exponent of $2\E[M]+\delta$ cancel. Therefore, \eqref{for:thm:prelim:density:1} follows immediately from \eqref{for:thm:prelim:density:4}, proving the required result.
	\end{proof}
	For obtaining the first moment convergence {in} \eqref{firstmoment:aim} using Theorem~\ref{thm:prelimit:density}, we have to sum over the range ${\left[ \lceil n(a_\omega-1/r)\rceil,\lceil n(a_\omega+1/r)\rceil \right]}^{|V(\tree)|}$. Summing the equality over this range, we obtain
	\[
	{\frac{1}{n}}\E\left[ N_{r,n}\left(\tree,(a_\omega)_{\omega\in V(\tree)}\right)\right] \to \mu\left( B_r^{\sss(G)}(\emp)\simeq \tree,~\left| A_\omega-a_\omega \right|\leq\frac{1}{r},\forall \omega\in V(\tree) \right).
	\]
	\subsection{Second moment convergence}
	Here we show that
	\[
	{\frac{1}{n^2}}\E\left[N_{r,n}\left(\tree,(a_\omega)_{\omega\in V(\tree)}\right)^2\right]\to \mu\left( B_r^{\sss(G)}(\emp)\simeq \tree,~\left| A_\omega-a_\omega \right|\leq\frac{1}{r},\forall \omega\in V(\tree) \right)^2,
	\]
	or, alternatively, the variance {of $N_{r,n}/n$ vanishes.} {Expanding the double sum yields}
	\eqan{\label{eq:second-moment:expansion}
			N_{r,n}\left(\tree,(a_\omega)_{\omega\in V(\tree)}\right)^2 =& \sum\limits_{v\in [n]} \one_{\left\{ B_r^{\sss(G_n)}(v)~\simeq~ \tree, ~|v_\omega/n-a_\omega|\leq 1/r~\forall\omega~\in~V(t) \right\}}\nn\\
			&+\sum\limits_{\substack{u\neq v\\ u,v\in[n]}} \one_{\left\{ B_r^{\sss(G_n)}(u)~\simeq~ \tree,~B_r^{\sss(G_n)}(v)\simeq~ \tree, ~\left|\frac{u_\omega}{n}-a_\omega\right|\leq \frac{1}{r},~\left|\frac{v_\omega}{n}-a_\omega\right|\leq \frac{1}{r}~\forall\omega~\in~V(t) \right\}},}
	where $v_\omega$ and $u_\omega$ are the vertices in $B_r^{\sss(G_n)}(v)$ and $B_r^{\sss(G_n)}(u)$ corresponding to the vertex $\omega\in V(\tree)$. Note that the first term in the RHS of \eqref{eq:second-moment:expansion} equals $N_{r,n}\left(\tree,(a_\omega)_{\omega\in V(\tree)}\right)$. {By} the results proved earlier in this section, it follows that this term, upon dividing by $n^2$, vanishes.
	Therefore, we are {left} to analyse the second term {on} the RHS of \eqref{eq:second-moment:expansion}.
	First, we show that $r$-neighbourhoods of two uniformly chosen vertices are disjoint whp. Note that
	\eqn{\label{eq:disjoint-neighbourhood}
		\prob\left( B_r^{\sss(G_n)}\left( o_n^{\sss(1)} \right)\cap B_r^{\sss(G_n)}\left( o_n^{\sss(2)} \right) = \emp \right)=\prob\left( o_n^{\sss(2)}\notin B_{2r}^{\sss(G_n)}\left( o_n^{\sss(1)} \right) \right)
			= 1-\frac{1}{n}\E\left[ \left| B_{2r}^{\sss(G_n)}\left( o_n^{\sss(1)} \right) \right| \right].}
	Previously we have proved that $B_{2r}^{\sss(G_n)}\left( o_n^{\sss(1)} \right)$ converges in distribution to $B_r^{\sss(G)}\left( \emp \right)$, where $(G,\varnothing)$ is the random P\'olya point tree with parameters $M$ and $\delta$, so that $\left\{\left| B_{2r}^{\sss(G_n)}\left( o_n^{\sss(1)} \right) \right|\right\}_{n\in\N}$ is a tight sequence of random variables. Therefore the $r$-neighbourhood of two uniformly chosen vertices are whp disjoint. 
	Now as we did {for} the first moment, we use a density argument {for the second term:} 
	\begin{Theorem}{\label{thm:second-moment:prelim}}
		Fix $\delta>-\infsupp(M)$, and let $G_n=\PArs_n(\boldsymbol{m},\delta)$. Let $\left(\tree_1,(a_\omega)_{\omega\in V(\tree_1)}\right)$ and $\left(\tree_2,(a_\omega)_{\omega\in V(\tree_2)}\right)$ be two vertex-marked trees of depth $r$ with disjoint vertex marks. If $v_\omega$ denotes the vertex in $G_n$ corresponding to $\omega\in V(\tree_1)\cup V(\tree_2)$, then, uniformly for all distinct $v_\omega\geq \eta n$ and $\hat{i}_{v_\omega}\leq \left( v_\omega \right)^{1-{\varrho/2}}$,
		\eqan{\label{eq:thm:second-moment:prelim}
			&\prob_{m,\psi}\left( B_r^{\sss(G_n)}\left(o_n^{\sss(1)}\right)\simeq \tree_1,~B_r^{\sss(G_n)}\left(o_n^{\sss(2)}\right)\simeq \tree_2,~v_{\omega}=\lceil na_\omega \rceil,\forall \omega\in V(\tree_1)\cup V(\tree_2) \right)\nn\\
			&\hspace{2cm}=(1+o_\prob(1))\frac{1}{n^{|V(\tree_1)\cup V(\tree_2)|}}~g_{r,\tree_1}\left( (a_\omega)_{\omega\in V(\tree_1)}; \left( m_{v_\omega},\hat{i}_{v_\omega} \right)_{\omega\in V(\tree_1)} \right)\\
			&\hspace{6cm} \times g_{r,\tree_2}\left( (a_\omega)_{\omega\in V(\tree_2)}; \left( m_{v_\omega},\hat{i}_{v_\omega} \right)_{\omega\in V(\tree_2)} \right),\nonumber
		}
		where $g_{r,\tree}(\cdot)$ is as in Theorem~\ref{thm:prelimit:density} and {$o_n^{\sss(1)},o_n^{\sss(2)}\in[n]$ are chosen independently and uniformly at random.}
	\end{Theorem}
	\begin{proof}
		Since most of the steps in this proof are similar to the proof of Theorem~\ref{thm:prelimit:density}, we will be more concise.
		We restrict our analysis to the case where $B_r^{\sss(G_n)}\left( o_n^{\sss(1)} \right)$ and $B_r^{\sss(G_n)}\left( o_n^{\sss(2)} \right)$ are disjoint graphs. Using the conditional independence of the edge-connection events for $\CPU$, we can {write} the LHS of \eqref{eq:thm:second-moment:prelim} {as}
		\eqn{\label{for:thm:SMP:1}
			\begin{split}
				&\prob_{m,\psi}\left( B_r^{\sss(G_n)}\left(o_n^{\sss(1)}\right)\simeq \tree_1,~B_r^{\sss(G_n)}\left(o_n^{\sss(2)}\right)\simeq \tree_2,~v_{\omega}=\lceil na_\omega \rceil,\forall \omega\in V(\tree_1)\cup V(\tree_2) \right)\\
				=&~ \prob_{m,\psi}\left( B_r^{\sss(G_n)}\left(o_n^{\sss(1)}\right)\simeq \tree_1,~v_{\omega}=\lceil na_\omega \rceil,\forall \omega\in V(\tree_1) \right)\\
				&\hspace{4cm}\times\prob_{m,\psi}\left( B_r^{\sss(G_n)}\left(o_n^{\sss(2)}\right)\simeq \tree_2,~v_{\omega}=\lceil na_\omega \rceil,\forall \omega\in V(\tree_2) \right).
		\end{split}}
		Theorem~\ref{thm:prelimit:density} {then immediately implies} Theorem~\ref{thm:second-moment:prelim}.
	\end{proof}

	Considering $\tree_1\simeq\tree_2$ but with different marks, {by} Theorem~\ref{thm:second-moment:prelim}, it {follows} that 
	\eqn{\label{eq:second-moment:expectation}
		\begin{split}
			{\frac{1}{n^2}}\E\left[N_{r,n}\left( \tree,(a_\omega)_{\omega\in V(\tree)} \right)\right] =~& (1+o(1))
			\mu\left( B_r^{\sss(G)}(\emp)\simeq \tree,~\left| A_\omega-a_\omega \right|\leq\frac{1}{r},\forall \omega\in V(\tree) \right)^2.
	\end{split}}
	{Therefore the {variance of $N_{r,n}\left( \tree,(a_\omega)_{\omega\in V(\tree)} \right)/n$ vanishes} and the local convergence of model (A) to the random P\'olya point tree follows immediately.
		
		\begin{remark}[Convergence of models (B) and (D)]
			{\rm We have {now} proved the local convergence of the model (A) which is equal in distribution {to} $\CPU^{\sss\rm{(SL)}}$. {The proof of} local convergence of models (B) and (D) follows along the same lines. Theorem~\ref{thm:equiv:CPU:PArt} and \ref{thm:equiv:PU:PAri} show that models (B) and (D) have the same law as $\CPU^{\sss\rm{(NSL)}}$ and $\PU^{\sss\rm{(NSL)}}$, respectively. Recall that by Remark~\ref{remark:equal-edge-probability}, the edge-connection probabilities for all the models $\CPU^{\sss\rm{(SL)}},\CPU^{\sss\rm{(NSL)}}$ and $\PU^{\sss\rm{(NSL)}}$ behave similarly. Therefore, {upon substitution of these in \eqref{eq:conditional_law:CPU:1},} the local convergence {of} models (B) and (D) {follows} from {the same} calculation.}\hfill$\blacksquare$
	\end{remark}}
	\section{Coupling between models { (D) and (E)}}
	\label{sec:model(F):coupling}
	In this section, we prove that models {(D) and (E)} have the same local limit.
	We do this by showing that for any fixed $r\geq 1$, we can couple models {(D) and (E)} in such a way that whp, the {$r$-neighbourhoods} of a uniformly chosen vertex in {both the} models are identical. 
	
	We follow the {coupling argument} used in \cite{BergerBorgs} with proper modifications as required for our models. 
	{For} all $j\in[m_n],$ the conditional edge-connection probability for model (E) is given by
	\eqn{\label{edge:probability:F}
		\begin{split}
			\prob_m\left( n\overset{j}{\rightsquigarrow} v\left| \PArf_{n-1}(\boldsymbol{m},\delta)  \right.\right)=~\frac{{d^{\prime}_v(n-1)}+\delta}{a_{[2]}+2\left( m_{[n-1]}-2 \right)+(n-1)\delta}\,,
	\end{split}}
	{where \(d^{\prime}_{v}(n-1)\) is the degree of vertex \(v\) in \(\PArf_{n-1}(\boldsymbol{m},\delta)\).}
	We aim to couple $\PArf_n(\boldsymbol{m},\delta)$ with $\PAri_n(\boldsymbol{m},\delta)$.
	Let $V=\{ 1,2,\ldots \}$ be the vertices of the preferential attachment model. For $1\neq n\in V$ and $i\in [m_n]$,  {let} $e_n^i$ and $f_n^i$ denote the vertices in $[n-1]$ to which vertex $n$ connects with its $i$-th edge in models (D) and (E), {respectively}. Denote
	\eqn{
		\begin{split}
			{\mathbf{e}_n}=\left\{ e_n^i \right\}_{i\in[m_n]}\qquad \text{and}\qquad \mathbf{f}_n=\left\{ f_n^i \right\}_{i\in[m_n]}\,.
	\end{split}}
	  {Since we start with the same initial graph $G_0,$  {we have that} $\mathbf{e}_2=\mathbf{f}_2$. 
    {Conditional} on $\boldsymbol{m}$, \(\{{\bf{e}}_{\ell}\}_{\ell<n}\) and \( \{{\bf{e}}_n^{i}\}_{i<j} \), let \({\bf{e}}_n^{j}\) have distribution \(D_1^{\sss(n,j)}\), and, conditionally on $\{\mathbf{f}_l\}_{l<n},$ {let} $\mathbf{f}_n^j$ have distribution $D_2^{\sss(n,j)}$. Let \(D^{(n,j)}\) be a coupling of $D_1^{\sss(n,j)}$ and $D_2^{\sss(n,j)}$ that minimizes the total variation distance. Then, we choose $\mathbf{e}_n^{j}$ and $\mathbf{f}_n^{j}$ according to ${D^{\sss(n,j)}}$. This provides us with a coupling between the edge-attachments in models (D) and (E).} {Let $(G_n)_{n\geq 2}$ and $(G_n^\prime)_{n\geq 2}$ be the sequences of preferential attachment graphs of models (D) and (E), respectively, under this optimal coupling, which, for the convenience of the reader, we now explain in more detail.}

    {\paragraph{The optimal coupling}
		Since the out-degree distributions of both model (D) and (E) are i.i.d.\ copies of \(M\), to construct the optimal coupling between \(G_{n}\) and \(G_{n}^{\prime}\), we use the same initial graph and the same sequence of out-degrees \((m_{v})_{{v\geq 3}}\). Next, conditionally on \((m_{v})_{{v\geq 3}}\), we couple \(G_{n}\) and \(G_{n}^{\prime}\). Note that by \eqref{sec:model:eq:1} and \eqref{eq:model:f:1}, for any \(n\geq3\), the denominator of the probabilities of the first edge-connection probabilities from \(n\) to any vertex in \([n-1]\) are equal. We couple \(G_{n}\) and \(G_{n}^{\prime}\) by coupling each edge optimally one by one. We provide an example of one of these edge-couplings. Let \(\hat{G}_{n,j}\) and \(\hat{G}_{n,j}^{\prime}\) denote the coupled graphs when \(j\) edges from vertex \(n\) have been constructed optimally. Further, conditionally on \(\hat{G}_{n,j}\) and \(\hat{G}_{n,j}^{\prime}\), let \(p_{j}(v)\) and \(p_{j}^{\prime}(v)\) denote the probability of connecting \(n\) to \(v\) with \((j+1)\)-th edge, for all \(v\in[n-1]\). Therefore, \(\mathbf{e}_{n}^{j+1}\) has probability mass function given by \(\big\{p_{j}(v):v\in[n-1]\big\}\) and \(\mathbf{f}_{n}^{j+1}\) has probability mass function given by \(\big\{p_{j}^{\prime}(v):v\in[n-1]\big\}\). We couple the random variables \(\mathbf{e}_{n}^{j+1}\) and \(\mathbf{f}_{n}^{j+1}\) optimally. Conditionally on \((\hat{G}_{n,j},\hat{G}_{n,j}^{\prime})\), for \(u\in[n-1]\),
		\eqn{\label{opt:coup:agree}
			\prob \left( \mathbf{e}_{n}^{j+1} = \mathbf{f}_{n}^{j+1} = u \mid \hat{G}_{n,j},\hat{G}_{n,j}^{\prime} \right) = p_{j}(u)\wedge p_{j}^{\prime}(u), }
		whereas, for \(u\neq \omega \in [n-1]\),
		\eqan{\label{opt:coup:disagree}
		&\prob \left( \mathbf{e}_{n}^{j+1} =u,~ \mathbf{f}_{n}^{j+1} = \omega \mid \hat{G}_{n,j},\hat{G}_{n,j}^{\prime} \right)\nn\\
		&\hspace{2cm} = \frac{\left( p_{j}(u) - p_{j}(u)\wedge p_{j}^{\prime}(u) \right) \left( p_{j}(\omega) - p_{j}(\omega)\wedge p_{j}^{\prime}(\omega) \right)}{\tfrac{1}{2}\sum\limits_{z\in[n-1]} |p_{j}(u) - p_{j}^{\prime}(u)|}~.}
		This \emph{optimal coupling} technique is described in detail in, e.g., \cite[Theorem~2.9]{vdH1}.
	}
	\begin{Proposition}[Coupling between models (D) and (E)]{\label{prop:coupling:D:E}}
		Fix $\varepsilon>0$ and $r\in\N$. Then,  {for the above coupling $(\hat{G}_n, \hat{G}_n^\prime)_{n\geq 1}$, there exists $n_0\in\N$} (possibly depending on $\varepsilon$ and $r$) such that for all $n>n_0,$ with probability at least $1-\varepsilon,$ the $r$-neighbourhoods of a randomly chosen vertex $o_n\in[n]$ are the same in both  {$\hat{G}_n$ and $\hat{G}_n^\prime$.}
	\end{Proposition}
	
	To simplify notation, below we write  {$((G_n,G_n^\prime))_{n\geq 2}$ for the coupled graph process.} Fix $\varepsilon>0$ and $r\in\N$ and let $B_r(v)$ and $B_r^\prime(v)$ denote the $r$-neighbourhood of $v$ in $G_n$ and $G_n^\prime$, respectively. We {denote} the set of bad vertices $v$ for which $B_r(v)\neq B_r^\prime(v)$ {by}
	\[
	\mathcal{B}_{n,r}=\left\{ v\in[n]\colon B_r(v)\neq B_r^\prime(v) \right\}.
	\]
	Since $o_n$ is chosen uniformly from the set of vertices, the probability that $o_n$ is a bad vertex {equals} $\E\left({|\mathcal{B}_{n,r}|/n}\right)$. {Thus,} to prove Proposition~{\ref{prop:coupling:D:E}}, it is enough to show that $\E(|\mathcal{B}_{n,r}|)\leq \varepsilon n.$

	{For any \(n>v\), define
    \eqn{
	   \label{new:proof:2}
	   {\sf C}_1(v,n)=\sum\limits_{i\in[m_n]}\one_{\{{\bf{e}}_n^{i}=v\}},\quad\text{and}\quad {\sf{C}}_2(v,n)=\sum\limits_{i\in[m_n]}\one_{\{{\bf{f}}_n^{i}=v\}}~.
    }
	{We} claim that for {$B_r(v)\neq B_r^\prime(v)$ to hold, there must be a vertex $u \in B_r(v)$ such that \( {\sf{C}}_1(u,n^\prime)\neq {\sf{C}}_2(u,n^\prime) \) for some $n^\prime>u$.} To investigate these bad vertices, {we thus concentrate on the bad events}
	\eqn{\label{eq:bad:events}
		A_{n^\prime}^{\sss(u)}=\left\{ {\sf{C}}_1(u,n^\prime)\neq {\sf{C}}_2(u,n^\prime) \right\}\,.
	}}
	
	{Let us consider 
	\eqn{\label{eq:def:bad-ev}
	A^{\sss(u)}=\bigcup\limits_{n\geq n^\prime>u} A_{n^\prime}^{\sss(u)},}
	i.e., $A^{\sss(u)}$ is the event that the number of edges received by $u$ is different in $G_n$ and $G_n^\prime$. The $1$-neighbourhoods of vertex $v$ {are different} in both $G_n$ and $G_n^\prime$ precisely when  $A^{\sss(u)}$ happens for at least one of the $u\in B_1(v)$. Now inductively on $r$, it {is easy to show} that $B_r(v)=B_r^\prime(v)$ unless there exists a $u\in B_r(v)$ for which $A^{\sss(u)}$ holds true.} {Let us denote $\eta^\prime=\eta(\varepsilon/8,2r)$ from Lemma~\ref{lem:RPPT:property}.}
	The following lemma gives us a tool to bound the probabilities of the bad events:
	\begin{Lemma}[Conditional control on bad events]\label{lem:conditional:control:bad_events}
		With the coupling of models {(D) and (E)} defined earlier in this section and fixing $v\geq \eta^\prime n,$ we define $A_n=A_n^{\sss(v)}$ to keep the notations simple. Then,
		\eqn{\label{lem:eq:conditional:control:bad_events}
			\prob\left( {A_n\cap\Big(\bigcap\limits_{h=v+1}^{n-1} A_h^c\Big)}\right)=~o\left( {n^{-1}} \right).}
	\end{Lemma}
	Lemma~\ref{lem:conditional:control:bad_events} implies {that}, conditionally on the event that there have not been any bad {events} till the $(n-1)$-th vertex joins the graph, the probability of the bad event taking place when the $n$th vertex is joining the graph is $o({1/n}).$

	We {first} provide the proof of Proposition~\ref{prop:coupling:D:E} subject to Lemma~\ref{lem:conditional:control:bad_events}, and then we prove Lemma~\ref{lem:conditional:control:bad_events}.
	
	\begin{proof}[Proof of Proposition~\ref{prop:coupling:D:E}] First we construct a set {$W_{n,r}$} of well-behaved vertices with the following properties:
	\begin{itemize}
		\item[$\rhd$] for all $v\in W_{n,r}$, {$B_{2r}(v)$} has no more than $N$ vertices;
		\item[$\rhd$] the oldest vertex in $B_{2r}(v)$ is not older than $\eta^\prime n$,
	\end{itemize}
	{where $\eta^\prime$ is as defined earlier, immediately before Lemma~\ref{lem:conditional:control:bad_events}}. By the proof of local convergence of model (D) and the regularity property of the RPPT (Lemma~\ref{lem:RPPT:property}),
	\eqn{\label{eq:nice:vertex:1}
		\E|W_{n,r}|\geq \left(1-\frac{\varepsilon}{2}\right)n.}
		
		{To} prove the proposition, it is {then} enough to show that
		\eqn{
			\label{eq:to:prove:prop:coupling:1}
			\E\left| W_{n,r}\cap \mathcal{B}_{n,r} \right|\leq {\varepsilon n/2}\,,}
		i.e., there are few well-behaved vertices that are bad. {To prove \eqref{eq:to:prove:prop:coupling:1}, we note that if} $v\in W_{n,r}\cap \mathcal{B}_{n,r},$ there must be a $u\in B_r(v)$ such that $A^{\sss(u)}$ is true. 
		Therefore, from {the definition of the bad events in \eqref{eq:def:bad-ev},}
		\eqn{\label{eq:to:prove:prop:coupling:2}
			\begin{split}
				\left| W_{n,r}\cap \mathcal{B}_{n,r} \right|=&\sum\limits_{v\in W_{n,r}}\one_{\{v\in \mathcal{B}_{n,r}\}}\leq \sum\limits_{v\in W_{n,r}}\sum\limits_{u\in B_{r}(v)}\one_{A^{\sss(u)}}\\
				=& \sum\limits_{u\in (\eta^\prime n,n]}\one_{A^{\sss(u)}}\sum\limits_{v\in B_r(u)}\one_{\{v\in W_{n,r}\}}.
		\end{split}}
		Note that the second sum in the RHS of \eqref{eq:to:prove:prop:coupling:2} produces $0$ if there is no well-behaved vertex in $B_r(u)$. Otherwise, the sum can be bounded by $\left|B_{r}(u)\right|$. Again $v \in B_r(u)$ implies that $u\in B_r(v)$ and $B_r(u)$ is a subset of $B_{2r}(v)$. If there is a well-behaved vertex in $B_r(u)$, we can bound $\left| B_r(u) \right|{\leq}N$ from the properties of the well-behaved vertices. Therefore the second sum on the RHS of \eqref{eq:to:prove:prop:coupling:2} can be uniformly bounded by $N$ for all vertices $u>\eta^\prime n$. Hence the cardinality of the {set of} well-behaved bad vertices can be further upper-bounded as
		\eqn{\label{eq:to:prove:prop:coupling:3}
			\left| W_{n,r}\cap \mathcal{B}_{n,r} \right|\leq N \sum\limits_{u\in (\eta^\prime n,n]}\one_{A^{\sss(u)}},}
		so that
		\eqn{\label{eq:to:prove:prop:coupling:3-1}
			\E \left| W_{n,r}\cap \mathcal{B}_{n,r} \right| \leq N\sum\limits_{u\in (\eta^\prime n,n]} \prob\left( A^{\sss(u)} \right).
		}
		{By} the control on the bad events obtained in Lemma~\ref{lem:conditional:control:bad_events}, {for \(u\geq\eta^{\prime} n\),}
		\eqn{\label{eq:to:prove:prop:coupling:4}
			\begin{split}
				\prob\left( A^{\sss(u)} \right) = &\prob\Big( \bigcup\limits_{l=u+1}^{n}A^{\sss(u)}_{l} \Big) =\sum\limits_{l=u+1}^{n}\prob\Big( A^{\sss(u)}_l\cap \Big( \bigcap\limits_{h=u+1}^{{l-1}} \left(A^{\sss(u)}_h\right)^c \Big) \Big) = o(1) .
			\end{split}
		}
		{Substituting} the bound {on} $\prob\left( A^{\sss(u)} \right)$ in \eqref{eq:to:prove:prop:coupling:3-1},
		\eqn{\label{eq:to:prove:prop:coupling:5}
			\E\left( \left| W_{n,r}\cap \mathcal{B}_{n,r} \right| \right) 
			= o(n),
		}
		which completes the proof of Proposition~\ref{prop:coupling:D:E} subject to Lemma~\ref{lem:conditional:control:bad_events}.
	\end{proof}
{It remains to prove Lemma~\ref{lem:conditional:control:bad_events} about bad events, {for which} we need the following lemma:
	\begin{Lemma}[Bounds on \(p\)-th moments of degrees in model (D)]\label{Lem:auxiliary:control:bad_event}
			There exists $c_1>0$ such that for all {$v\geq \eta^\prime n,$
			\(\E\left[ d_{v}(n)^p\right] \) is finite and independent of \(n\),
			where \(d_{v}(n)\) is the degree of vertex \(v\) in \(\PAri_{n}({\boldsymbol{m}},\delta)\)~.}
		\end{Lemma} 
		\begin{proof} {Let
			\eqn{
			\binom{a+p}{a}=\frac{\Gamma(a+p+1)}{\Gamma(a+1)\Gamma(p+1)}
			}
		denote the generalised binomial coefficient. Using \emph{Gautschi's inequality} multiple times, it can be shown that there exists \(\mathsf{C}_{1}\) and \(\mathsf{C}_{2}\), depending only on \(p\), such that for all \(a\) positive,
			\eqan{\label{for:aux:9}
				\mathsf{C}_{1} a^{p} \leq \binom{a+p}{a}\leq \mathsf{C}_{2}a^{p}~.
			}}
			{To prove the moment bound on \(d_{v}(n)\), it thus suffices to uniformly bound the expectation of
			\eqan{\label{for:aux:1}
				\mathcal{M}_{n} = \binom{d_{v}(n)+\delta+p}{d_{v}(n)+\delta} := \frac{\Gamma(d_{v}(n)+\delta+p+1)}{\Gamma(p+1)\Gamma(d_{v}(n)+\delta+1)}~,
			}
			for \(p\in(1,2]\). For this, conditionally on \(\boldsymbol{m}\), we compute
			\eqan{\label{for:aux:2}
				&\E_{m}\Big[ \mathcal{M}_{n+1} \mid d_{v}(n)\Big] \nn\\
				&\hspace{0.5cm}= \mathcal{M}_{n} \E_{m} \left[ \frac{\Gamma(d_{v}(n+1)+\delta+p+1)}{\Gamma(d_{v}(n)+\delta+p+1)} \frac{\Gamma(d_{v}(n)+\delta+1)}{\Gamma(d_{v}(n+1)+\delta+1)} \mid d_{v}(n)\right] \nn\\
				&\hspace{0.5cm} = \mathcal{M}_{n} \E_{m} \left[ \prod\limits_{i=1}^{m_{n}} \frac{\Gamma(d_{v}(n,i-1)+X_{n+1}^{\sss (i)}+\delta+p+1)}{\Gamma(d_{v}(n,i-1)+\delta+p+1)} \frac{\Gamma(d_{v}(n,i-1)+\delta+1)}{\Gamma(d_{v}(n,i-1)+X_{n+1}^{\sss (i)}+\delta+1)} \mid d_{v}(n) \right],
			}
where \(d_{v}(n,0)=d_{v}(n)\), and we let \((X_{n+1}^{\sss(i)})_{\in[m_{n}]}\) be a sequence of indicator random variables defined as
			\[
				X_{n+1}^{\sss (i)} = \one_{\{n+1~\mbox{\small connects to}~v~\mbox{\small with its}~i\mbox{\small-th edge}\}}~.
			\]
Next, we simplify	
			\eqan{\label{for:aux:3}
				&\frac{\Gamma(d_{v}(n,i-1)+X_{n}^{\sss (i)}+\delta+p+1)}{\Gamma(d_{v}(n,i-1)+\delta+p+1)} \frac{\Gamma(d_{v}(n,i-1)+\delta+1)}{\Gamma(d_{v}(n,i-1)+X_{n}^{\sss (i)}+\delta+1)}\nn\\
				 &\hspace{1cm}= \one_{\{X_{n}^{\sss (i)}=0\}}+\one_{\{X_{n}^{\sss (i)}=1\}}\frac{d_{v}(n,i-1)+X_{n}^{\sss (i)}+\delta+p}{d_{v}(n,i-1)+X_{n}^{\sss (i)}+\delta}\nn\\
				 &\hspace{1cm}= 1 + \frac{X_{n}^{\sss(i)} p}{d_{v}(n,i-1)+\delta+1}. 
			}
Taking expectation with respect to \(X_{n}^{\sss(i)}\) conditionally on \(\boldsymbol{m}\) and \(d_{v}(n,i-1)\) gives
			\eqan{\label{for:aux:4}
			\E_{m}&\left[\frac{\Gamma(d_{v}(n,i-1)+X_{n}^{\sss (i)}+\delta+p+1)}{\Gamma(d_{v}(n,i-1)+\delta+p+1)} \frac{\Gamma(d_{v}(n,i-1)+\delta+1)}{\Gamma(d_{v}(n,i-1)+X_{n}^{\sss (i)}+\delta+1)} \mid d_{v}(n,i-1)\right]\nn\\
			&= 1 + \frac{p}{d_{v}(n,i-1)+\delta+1}\E_{m}[X_{n}^{\sss(i)}\mid d_{v}(n,i-1)]\nn\\
			&= 1 + \frac{d_{v}(n,i-1)+\delta}{d_{v}(n,i-1)+\delta+1}\frac{p}{a_{[2]}+2(m_{[n]}-2)+(i-1)+n\delta}.
			}
			Note that the factor \( \tfrac{d_{v}(n,i-1)+\delta}{d_{v}(n,i-1)+\delta+1} \) in the RHS of \eqref{for:aux:4} is at most \(1\). Since \(p>1\), using \(1+px\leq (1+x)^{p}\), we bound the conditional expectation in \eqref{for:aux:4} by
			\eqan{\label{for:aux:5}
				&\E_{m}\left[ \frac{\Gamma(d_{v}(n,i-1)+X_{n}^{\sss (i)}+\delta+p+1)}{\Gamma(d_{v}(n,i-1)+\delta+p+1)} \frac{\Gamma(d_{v}(n,i-1)+\delta+1)}{\Gamma(d_{v}(n,i-1)+X_{n}^{\sss (i)}+\delta+1)} \mid d_{v}(n,i-1)\right]\nn\\
				&\hspace{1cm} \leq \left( 1+\frac{1}{a_{[2]}+2(m_{[n]}-2)+(i-1)+n\delta } \right)^{p} = \left(\frac{c_{n,i}}{c_{n,i-1}}\right)^{p}~,
			}
where \(c_{n,i} = a_{[2]}+2(m_{[n]}-2)+i+n\delta\). Next, we use the tower property to bound the conditional expectation in the RHS of \eqref{for:aux:2} as
			 \eqan{\label{for:aux:6}
			 &\mathcal{M}_{n} \E_{m} \left[ \prod\limits_{i=1}^{m_{n}} \frac{\Gamma(d_{v}(n,i-1)+X_{n+1}^{\sss (i)}+\delta+p+1)}{\Gamma(d_{v}(n,i-1)+\delta+p+1)} \frac{\Gamma(d_{v}(n,i-1)+\delta+1)}{\Gamma(d_{v}(n,i-1)+X_{n+1}^{\sss (i)}+\delta+1)} \mid d_{v}(n) \right]\nn\\
			 &\hspace{1cm}\leq \mathcal{M}_{n} \prod\limits_{i=1}^{m_{n}} \left(\frac{c_{n,i}}{c_{n,i-1}}\right)^{p} = \mathcal{M}_{n} \left(\frac{c_{n+1}}{c_{n}}\right)^{p}~,
			 }
where \(c_{k}=c_{k,0}\) for all \(k\in\N\). Therefore, iterating this bound gives us
			\eqan{\label{for:aux:7}
				\E\big[ \mathcal{M}_{n} \big] \leq \E \left[ \mathcal{M}_{v+1} \prod\limits_{u=v+1}^{n-1} \left(\frac{c_{u+1}}{c_{u}}\right)^{p}\right] = \E\left[ \binom{m_{v}+\delta+p}{m_{v}+\delta} \left(\frac{c_{n}}{c_{v+1}}\right)^{p}   \right].
			}
Bounding \(\left(\frac{c_{n}}{c_{v+1}}\right)^{p}\leq 2+2\left(\tfrac{c_{n}-c_{v+1}}{c_{v-1}}\right)^{p} \), we obtain
			\eqan{\label{for:aux:8}
			\E\big[ \mathcal{M}_{n} \big] \leq 2\E\left[ \binom{m_{v}+\delta+p}{m_{v}+\delta} \right] + 2 \E\left[ \binom{m_{v}+\delta+p}{m_{v}+\delta} \right] \E[(c_{n}-c_{v+1})^{p}]\E[1/c_{v-1}^{p}].
			}
Since \(v\geq \eta^{\prime}n\), the multiplicative factor \(\E[(c_{n}-c_{v+1})^{p}]\E[1/c_{v-1}^{p}]\) can be bounded by \(\E[(c_{n}-c_{\eta^{\prime}n+1})^{p}]\E[1/c_{\eta^{\prime}n-1}^{p}]\).
Following a similar computation as in \arxiversion{Lemma~\ref{prop:M-inverse:expectation}}\journalversion{\cite[Lemma~B.1]{GHvdHR22}}, we obtain \(\E[1/c_{\eta^{\prime}n-1}^{p}] \leq \mathsf{C}_{0}(\eta^{\prime}n)^{-p}  \) for some constant \(\mathsf{C}_{0}\) finite, whereas \( c_{n}-c_{\eta^{\prime}n+1} \) can be rewritten as 
\[
	c_{n}-c_{\eta^{\prime}n+1} = \sum\limits_{u=\eta^{\prime}n+2}^{n} (2m_{u}+\delta)~.
\]
Now, by {\em Minkowski inequality} \cite[Theorem~2.6 in Chapter 3]{allangut},
\eqan{\label{for:aux:10}
	\left(\E[(c_{n}-c_{\eta^{\prime}n+1})^{p}]\right)^{\tfrac{1}{p}} \leq \sum\limits_{u=\eta^{\prime}n}^{n} \left(\E \big[ (2m_{u}+\delta)^{p} \big] \right)^{\tfrac{1}{p}} \leq (n-\eta^{\prime}n) \left(\E\big[ (2M+\delta)^{p}\big] \right)^{\tfrac{1}{p}}.
}
Therefore, \(\E[(c_{n}-c_{\eta^{\prime}n+1})^{p}]\) can be bounded by \( n^{p}(1-\eta^{\prime})^{p}\E[(2M+\delta)^{p}] \). 
Since \(m_{v}\sim M\), by \eqref{for:aux:9}, \(\E[\mathcal{M}_{n}]\) in \eqref{for:aux:8} can be further bounded as
			\eqan{\label{for:aux:11}
				\E[\mathcal{M}_{n}] \leq 2\mathsf{C}_{2} \E[(M+\delta)^{p}]\left( 1+ \tfrac{\E[(2M+\delta)^{p}]}{\mathsf{C}_{0}}(\tfrac{1}{\eta^{\prime}}-1)^{p} \right)~.
			}
			Since \(\E[M^{p}]\) is finite, by \eqref{for:aux:9} \(p\)-th moment of \(d_{v}(n)\) is finite and free of \(n\).}
		\end{proof}
		{We next turn to the proof of Lemma~\ref{lem:conditional:control:bad_events}:}

\paragraph{Proof of Lemma~\ref{lem:conditional:control:bad_events}} 
		{To reduce notational complexity, we denote the conditioning event as
		\eqn{\label{eqn:conditional-event-simplified}
			{\mathcal{E}} := \bigcap\limits_{h=v+1}^{n-1} (A_h^{\sss(v)})^c~.
		}}
		To prove Lemma~\ref{lem:conditional:control:bad_events}, we bound the LHS {of \eqref{lem:eq:conditional:control:bad_events}} by
        {\eqan{\label{eq:new:proof:1}
            \prob\Big(\{{\sf{C}}_1&(v,n)>1\}\cap{\mathcal{E}} \Big)+\prob\Big(\{{\sf{C}}_2(v,n)>1\} \cap {\mathcal{E}} \Big)\nn\\
            &+\prob\Big(\{{\sf{C}}_1(v,n)=1,{\sf{C}}_2(v,n)=0\}\cap{\mathcal{E}} \Big)+\prob\Big(\{ {\sf{C}}_1(v,n)=0,{\sf{C}}_2(v,n)=1 \}\cap {\mathcal{E}} \Big)~,}}
with \(1\) being a trivial upper bound.
		{Since we are optimally coupling \(G_{n}\) and \(G_{n}^{\prime}\), the out-degree of vertex \(v\) in both \(G_{n-1}\) and \(G_{n-1}^{\prime}\) are equal to \(m_{v}\). Additionally, the event  {\(\mathcal{E}\)} guarantees that the in-degree of vertex \(v\) is equal in  both \(G_{n-1}\) and \(G_{n-1}^{\prime}\). Therefore, within the scope of the event} \({\mathcal{E}}\), the degree of the vertex \(v\) is equal in both \(G_{n-1}\) and \( G_{n-1}^\prime \). 
   {Note that the probability of the first edge connection to vertex \(v\) from \(n\) in \(G_{n-1}\) is less than or equal to the edge connection probability between vertices \(n\) and \(v\) in \(G^{\prime}_{n-1}\). Since we are performing the optimal coupling edge by edge, by \eqref{opt:coup:agree}, the first edge connection from \(n\) to \(v\) in \(G_{n-1}\) ensures an edge connection from \(n\) to \(v\) in \(G^{\prime}_{n-1}\). Hence, it \textbf{can not hold} that
	\eqn{
		{\sf{C}}_{1}(v,n) = 1,\qquad \text{whereas}\qquad {\sf{C}}_{2}(v,n)=0~.
	}     
Therefore, the third probability in \eqref{eq:new:proof:1} equals \(0\).} Now, we move on to upper bounding the remaining terms in \eqref{eq:new:proof:1}.
        
		In $\PArf_n(\boldsymbol{m},\delta){=G^{\prime}_{n}}$, conditionally on $\boldsymbol{m}$ and ${(G_{n-1},G^{\prime}_{n-1})}$, the probability of having $k$ connections from $n$ to $v$ and $m_n-k$ connections to other vertices in $[n-1]$ is given by
		\eqn{\label{for:lem:bad:events:2}
			Q_k = \binom{m_n}{k}(p^\prime)^k(1-p^\prime)^{m_n-k},}
		where ${p^\prime= {(d^{\prime}_{v}(n-1)+\delta)/(a_{[2]}+2\left( m_{[n-1]}-2 \right)+(n-1)\delta)}}$ {is} as defined in \eqref{edge:probability:F}. On the other hand, in $\PAri_n(\boldsymbol{m},\delta){=G_{n}}$, conditionally on $\boldsymbol{m}$ and ${(G_{n-1},G^{\prime}_{n-1})}$, the probability of the same event is
		\eqn{
			\label{for:lem:bad:events:3}
			T_k=\binom{m_n}{k}\prod\limits_{l=0}^{k-1}p_l(k)\prod\limits_{l=k}^{m_n-1}(1-p_l(k)),
		}
where 
		\eqn{\label{eq:def:plk}
		p_l(k)=\begin{cases}
			&\frac{{d_{v}(n-1)}+l+\delta}{a_{[2]}+2\left( m_{[n-1]}-2 \right)+l+(n-1)\delta},\quad\text{for}~l<k,\\
			&\frac{{d_{v}(n-1)}+k+\delta}{a_{[2]}+2\left( m_{[n-1]}-2 \right)+l+(n-1)\delta},\quad\text{for}~l\geq k.
		\end{cases}
		}
		{Indeed,} the exchangeability {of} model (D) {implies} that the probability of connecting any $k$ edges from $n$ to $v$ is {the} same, {which explains \eqref{for:lem:bad:events:3}}. {Since \({\mathcal{E}}\) is measurable with respect to \((G_{n-1},G_{n-1}^{\prime})\)}, for \(i\in\{1,2\}\),
		\eqan{\label{for:lem:bad:events:9}
		&{\prob\Big(\{{\sf{C}}_i(v,n)>1\}\cap{\mathcal{E}} \Big)
		= \E \big[ \prob_{m}\big({\sf{C}}_i(v,n)>1\mid (G_{n-1},G^{\prime}_{n-1}) \big) \one_{{\mathcal{E}}} \big]~.}
		}{Therefore, conditionally on \(\{\boldsymbol{m}, (G_{n-1},G^{\prime}_{n-1})\}\), }
        {\eqan{\label{eq:new:proof:2}
            &\prob_{m}\Big({\sf{C}}_1(v,n)>1\mid (G_{n-1},G^{\prime}_{n-1}) \Big)+\prob_{m}\Big({\sf{C}}_2(v,n)>1 \mid (G_{n-1},G^{\prime}_{n-1}) \Big)\\
            &\hspace{1cm}=~\sum\limits_{k=2}^{m_n}Q_k+\sum\limits_{k=2}^{m_n}T_k~.\nn
        }}
        {Conditionally on $\boldsymbol{m},$ let $\prob_{m}^{\sss\rm{(D)}}$ denote the law for model (D).
		Next we {condition} on the event that the first $2$ out-edges from $n$ connect to $v$, which is denoted by $n\overset{1,2}{\rightsquigarrow}v$. {We bound}		
		\eqan{\label{for:lem:bad:events:7}
			\sum\limits_{k=2}^{m_n}T_k =&\sum\limits_{k=2}^{m_n}\binom{m_n}{k}\prob_{m}^{\sss\rm{(D)}}\left( n\rightsquigarrow v\ \text{{only} with its first $k$ edges}{\mid (G_{n-1},G^{\prime}_{n-1}) } \right)\nn\\
				\leq&~ \frac{m_n^2}{2}\prob_{m}^{\sss\rm{(D)}}\left( n\overset{1,2}{\rightsquigarrow} v\mid (G_{n-1},G^{\prime}_{n-1})\right)\nn\\
				&\hspace{0.5cm}\times\sum\limits_{k=2}^{m_n}\binom{m_n-2}{k-2}\prob_{m}^{\sss\rm{(D)}}\left( n\rightsquigarrow v\ \text{{only} with its first $k$ edges}\left| n\overset{1,2}{\rightsquigarrow} v, (G_{n-1},G^{\prime}_{n-1})\right. \right)\nn\\
				=&~{\frac{m_n^2}{2}\prob_{m}^{\sss\rm{(D)}}\left( n\overset{1,2}{\rightsquigarrow} v\mid (G_{n-1},G^{\prime}_{n-1})\right)}\nn\\
				\leq & \frac{m_n^2(d_{v}(n-1)+\delta+1)^2}{2n^2}.
		}
		{The last equality in \eqref{for:lem:bad:events:7} follows from the fact that the conditional probabilities {sum} to $1$, whereas
		the last inequality in \eqref{for:lem:bad:events:7} is obtained by substituting the first and second edge-connection probabilities from $n$ to $v$, and lower bounding the denominator of the probabilities by \(n\).}
		
		Similarly, for model (E), if $\prob_{m}^{\sss\rm{(E)}}$ denotes the conditional law for model (E), then the probability of connecting more than $1$ edge from $n$ to $v$ can be bounded as
		\eqan{\label{for:lem:bad:events:8}
			\sum\limits_{k=2}^{m_n}Q_k=&\sum\limits_{k=2}^{m_n}\binom{m_n}{k}\prob_{m}^{\sss\rm{(E)}}\left( n\rightsquigarrow v\ \text{{only} with its first $k$ edges} {\mid (G_{n-1},G^{\prime}_{n-1}) } \right)\nn\\
				\leq&~ \frac{m_n^2}{2}\prob_{m}^{\sss\rm{(E)}}\left( n\overset{1,2}{\rightsquigarrow} v{\mid (G_{n-1},G^{\prime}_{n-1}) } \right)\nn\\
				&\hspace{0.5cm}\times\sum\limits_{k=2}^{m_n}\binom{m_n-2}{k-2}\prob_{m}^{\sss\rm{(E)}}\left( n\rightsquigarrow v\ \text{{only} with its first $k$ edges}\mid n\overset{1,2}{\rightsquigarrow} v,~{(G_{n-1},G^{\prime}_{n-1}) }  \right)\nn\\
				\leq&~ \frac{m_n^2(d^{\prime}_{v}(n-1)+\delta)^2}{2n^2}.
		}
		Since \(d_{v}(n-1)=d^{\prime}_{v}(n-1)\) on \({\mathcal{E}}\), by equation \eqref{for:lem:bad:events:7} and \eqref{for:lem:bad:events:8}, we bound the LHS of \eqref{for:lem:bad:events:9} as
		\eqn{
			\prob\Big(\{{\sf{C}}_i(v,n)>1\}\cap{\mathcal{E}} \Big) \leq \E \Big[ \Big(\tfrac{m_n^2(d_{v}(n-1)+\delta+1)^2}{2n^2} \wedge 1\big)\one_{{\mathcal{E}}} \Big]~.}		
		
		Lastly, we need to bound the remaining term in \eqref{eq:new:proof:1}. For that, we go back to the definition of \({\sf{C}}_1(v,n)\) and \({\sf{C}}_2(v,n)\). 
		Conditionally on \(\{{\boldsymbol{m}},(G_{n-1},G^{\prime}_{n-1})\}\), we split the probability as
        \eqan{\label{eq:new:proof:3}
            &{\prob}_{m}\Big( {\sf{C}}_1(v,n)=0,{\sf{C}}_2(v,n)=1{\mid (G_{n-1},G^{\prime}_{n-1})} \Big)\nn\\
            &\hspace{1cm}\leq \sum\limits_{i=1}^{m_n} {\prob}_{m}\Big( {\mathbf{e}}_n^{i}\neq v,~{\mathbf{f}}_n^{i}=v \mid  \bigcap\limits_{j<i}\big\{ \mathbf{e}_n^{j}\neq v~\text{and}~\mathbf{f}_n^{j}\neq v \big\},~{(G_{n-1},G^{\prime}_{n-1})} \Big)~.
        }
        {Note that the summand in the RHS of \eqref{eq:new:proof:3} represents the probability that the \(i\)-th edge from \(n\) connects to \(v\) in \(G_n^\prime\), whereas it does not connect to \(v\) in \(G_n\), conditionally on the fact that none of the first \((i-1)\) edges from \(n\) connects to \(v\) neither in \(G_n\) nor in \(G_n^\prime\), and the degree of vertex \(v\) being the same in both the graphs until the \((i-1)\)-th edge-connection. Therefore, for any \(i\in[m_{n}]\), {by \eqref{opt:coup:disagree}}
        \eqn{\label{eq:new:proof:4}
            {\prob}_{m}\Big( {\mathbf{e}}_n^{i}\neq v,~{\mathbf{f}}_n^{i}=v \mid  \bigcap\limits_{j<i}\big\{ \mathbf{e}_n^{j}\neq v~\text{and}~\mathbf{f}_n^{j}\neq v \big\},~{(G_{n-1},G^{\prime}_{n-1})} \Big) = p^\prime-p_i(0)~,
        } 
        {where we recall \(p^{\prime} = (d^{\prime}_{v}(n-1)+\delta)/(a_{[2]}+2\left( m_{[n-1]}-2 \right)+(n-1)\delta) \) and \(p_{i}(0)\) from \eqref{eq:def:plk}.}
        Subject to \({\mathcal{E}},\) RHS of \eqref{eq:new:proof:4} can be bounded by \({(d_{v}(n-1)+\delta)i}/{n^2}\),
        and hence, summing over \(i\in[m_n]\), we obtain
        \eqan{\label{eq:new:proof:5}
            {\prob}_{m}\Big( {\sf{C}}_1(v,n)=0,{\sf{C}}_2(v,n)=1 {\mid (G_{n-1},G^{\prime}_{n-1})} \Big)\one_{{\mathcal{E}}} \leq& \sum\limits_{i\in[m_n]}\frac{({d_{v}(n-1)} +\delta)i}{n^2}\one_{{\mathcal{E}}}\nn\\
            &\leq \frac{({d_{v}(n-1)}+\delta)m_n^2}{n^2}\one_{{\mathcal{E}}}~.
        }}
	{Similarly as in \eqref{for:lem:bad:events:9}, we bound the last term in \eqref{eq:new:proof:1} as
	\eqn{\label{eq:new:proof:06}
		{\prob}\Big( \{{\sf{C}}_1(v,n)=0,{\sf{C}}_2(v,n)=1 \}\cap {\mathcal{E}} \Big) \leq \E\Big[ \Big(\frac{({d_{v}(n-1)}+\delta)m_n^2}{n^2}\wedge1 \Big)\one_{{\mathcal{E}}}\Big]~.
	}
	}
        {Therefore, collecting the bounds for all the contributions from \eqref{eq:new:proof:1}, \eqref{eq:new:proof:2}, \eqref{for:lem:bad:events:7}, \eqref{for:lem:bad:events:8} and \eqref{eq:new:proof:06}, conditionally on $\boldsymbol{m},$ we obtain
	\eqan{\label{for:lem:bad:events:11}
			\prob\left( A_n\cap\left(\bigcap\limits_{h=v+1}^{n-1} A_h^c\right) \right) 
			&\leq 3\E\left[ \Big(\frac{{({d_{v}(n-1)}+\delta+1)^2}m_n^2}{n^2}\wedge1 \Big)\one_{{\mathcal{E}}} \right]\nn\\
			&\leq {c_0} \E\left[ \frac{{{d_{v}(n-1)}^2}m_n^2}{n^2}\wedge1 \right]~,
		}
for some constant ${c_0}>0$. {Now, we bound the RHS of \eqref{for:lem:bad:events:11} by
		\eqan{\label{for:lem:bad:events:11-1}
				&\E\left[  \frac{m_n^2 d_v(n-1)^2}{n^2}\wedge 1  \right]\nn\\
				&\hspace{1cm}=\prob\Big( m_nd_v(n-1) >n\Big)+\E\left[ \tfrac{m_n^2d_v(n-1)^2}{n^{2}}\one_{\left\{\tfrac{m_nd_v(n-1)}{n}\leq 1\right\}} \right].
				}
		}
		{By the \emph{Markov's inequality} and the fact that \(x^{2}\geq x^{p}\) for \(x\in[0,1]\) and \(p\in (1,2]\), we bound the RHS of \eqref{for:lem:bad:events:11-1} as
		\eqan{
		\label{for:lem:bad:events:12}
			&\prob\Big( m_nd_v(n-1)>n\Big)+\E\left[\frac{m_n^2d_v(n-1)^2}{n^{2}}\one_{\left\{\tfrac{m_nd_v(n-1)}{n}\leq 1\right\}} \right]\nn\\
			&\hspace{1cm}\leq\tfrac{1}{n^{p}}\E [ m_{n}^{p}d_{v}(n-1)^{p} ] + \E\left[ \tfrac{m_n^pd_v(n-1)^p}{n^{p}}\right]= \tfrac{2}{n^{p}}\E [ m_{n}^{p}d_{v}(n-1)^{p}].
		}
		}
	{Note that \(d_{v}(n-1)\) and \(m_{n}\) are independent random variables by construction. Therefore, by Lemma~\ref{Lem:auxiliary:control:bad_event}, the expectation in the RHS of \eqref{for:lem:bad:events:12} is finite for \( v\geq \eta^{\prime}n \).}

        Substituting the bounds obtained from \eqref{for:lem:bad:events:11-1} and \eqref{for:lem:bad:events:12} in \eqref{for:lem:bad:events:11} yields
		\[
		\prob\left( A_n\cap\left(\bigcap\limits_{h=v+1}^{n-1} A_h^c\right) \right) \leq c_4 n^{-{p}}=o\left( n^{-1} \right).
		\]
		This completes the proof of Lemma~\ref{lem:conditional:control:bad_events}.}
		{\qed}
		\smallskip
\section{Proof of Corollary~\ref{cor-older-younger-neighbours}}\label{sec-proof-cor-older-younger}
		In this section, we prove Corollary~\ref{cor-older-younger-neighbours} {as a consequence of Theorem~\ref{theorem:PArs}.}
		As a consequence of local convergence, if $G_n$ converges locally in probability to $(G,o)$ having law $\mu$, then for any bounded continuous function $h:\mathcal{G}_\star\mapsto\R,$
			\eqn{\label{for:cor:neighbour:1}
				\frac{1}{n}\sum\limits_{u\in[n]}h(G_n,u)\overset{\prob}{\to}\E_{\mu}[h(G,o)]~.}
			We start by proving \eqref{eq:deg:old}. From Theorem~\ref{theorem:PArs}, we have the local convergence of preferential attachment models (A-F) to the $\RPPT(M,\delta)$. Define for fixed $k\geq 1$, 
			\eqn{\label{for:cor:neighbour:2}
				h(H,u) = \sum\limits_{\substack{v,j:\\u\overset{j}{\rightsquigarrow} v}}\frac{\one_{\{ d_v(H)=k \}}}{m_u}~,} where $d_v(H)$ is the degree of the vertex $v$ in the graph $H$. 
			Clearly $h$ is a bounded continuous function on $\mathcal{G}_\star$. Therefore by \eqref{for:cor:neighbour:1},
			\eqn{\label{for:cor:neighbour:3}
				\frac{1}{n}\sum\limits_{\substack{u,v,j:\\u\overset{j}{\rightsquigarrow} v}}\frac{\one_{\{d_v(n)=k\}}}{m_u}\overset{\prob}{\to} \E_{\mu_\star} \big[ h(G,\emp) \big]~,}
			where $(G,\emp)$ is the rooted $\RPPT(M,\delta)$ with law $\mu_\star$. Now $h(H,o)$ is the fraction of older neighbours of $o$ in $H$ having degree $k$. Therefore, $\E_{\mu_\star} \big[ h(G,\emp) \big]$ is the probability that a random $\Old$ labelled neighbour of the root of $\RPPT(M,\delta)$ has degree $k.$ From the definition of $\RPPT(M,\delta)$, the degree distribution of an $\Old$ labelled node with age $a_\omega$ is $1+M^{(\delta)}+Y\big( M^{(\delta)}+1,a_\omega \big)$, where $Y\big( M^{(\delta)}+1,a_\omega \big)$ is a {mixed} Poisson random variable with mixing distribution $\Gamma_{\rm in}\big( M^{(\delta)}+1 \big)\lambda\big( a_\omega \big)$ and $\Gamma_{\rm in}\big( M^{(\delta)}+1 \big)$ is as defined in Section~\ref{sec:RPPT}. On the other hand, a random $\Old$ labelled neighbour of the root $\emp$ has age distributed as $U_\emp U_1^{1/i}$. Therefore a random $\Old$ labelled neighbour of the root has the degree distribution $1+M^{(\delta)}+Y\big( M^{(\delta)}+1,A_{\old}\big),$ where $Y\big( M^{(\delta)}+1,A_{\old}\big)$ is as defined previously in Corollary~\ref{cor-older-younger-neighbours} and
			\eqn{\label{for:cor:neighbour:4}
				\E_{\mu_\star} \big[ h(G,o) \big] = \prob\big( 1+M^{(\delta)}+Y\big( M^{(\delta)}+1,A_{\old}\big)=k \big)= \Tilde{p}_k^{({\old})}~.}
			We prove \eqref{eq:deg:young} in a similar way with a few adaptations. Instead of considering all vertices, we consider the vertices that have at least one younger neighbour. Next we choose the bounded continuous function as follows
			\eqn{\label{for:cor:neighbour:5}
				\begin{split}
					h_1(H,v)&=\sum\limits_{\substack{u,j:\\u\overset{j}{\rightsquigarrow} v}}\frac{\one_{\{d_v(H)>m_v\}}\one_{\{d_u(H)=k\}}}{d_v(H)-m_v}~,\\
					\mbox{and}\qquad\qquad h_2(H,v)&=\one_{\{v~{\sss\mbox{\rm has at least one younger neighbour}}\}}~.
			\end{split}}
			Similarly by \eqref{for:cor:neighbour:1}, 
			\eqn{\label{for:cor:neighbour:6}
				\begin{split}
					\frac{1}{n}\sum\limits_{\substack{u,v,j:\\u\overset{j}{\rightsquigarrow} v}}\frac{\one_{\{d_v(n)>m_v\}}\one_{\{d_u(n)=k\}}}{d_v(n)-m_v} &\overset{\prob}{\to} \E_{\mu_\star} \big[ h_1(G,\emp) \big]~,\\
					\mbox{and}\qquad \frac{1}{n}\sum\limits_{v\in [n]} \one_{\{v~{\sss\mbox{\rm has at least one younger neighbour}}\}}&\overset{\prob}{\to} \E_{\mu_\star} \big[ h_2(G,\emp) \big]~.
			\end{split}} 
			where $(G,\emp)$ is the rooted $\RPPT(M,\delta)$ with law $\mu_\star$. Now $h_1(H,o)$ is the fraction of younger neighbours of $o$ in $H$ having degree $k$. If $o$ has no younger neighbour then define $h_1(H,o)=0$. Therefore
			\eqn{\label{for:cor:neighbour:7}
				\E_{\mu_\star} \big[ h_1(G,\emp) \big] = \E_{\mu_\star} \big[ h_1(G,\emp)\big|\din_\emp>0 \big]\prob(\din_\emp>0)~,}
			and $\E_{\mu_\star} \big[ h_1(G,\emp)\big| \din_\emp>0 \big]$ is the probability that a random $\Young$ labelled neighbour of the root of $\RPPT(M,\delta)$ has degree $k$, conditionally on the event that the root has at least one younger neighbour.
			From the definition of $\RPPT(M,\delta)$, the degree distribution of a $\Young$ labelled node of age $a_\omega$ is $M^{(0)}+Y\big( M^{(0)},a_\omega \big)$ where $Y\big( M^{(0)},a_\omega \big)$ is a {mixed} Poisson random variable with mixing distribution $\Gamma_{\rm in}\big( M^{(0)} \big)\lambda\big( a_\omega \big)$.

			Using the fact that $\din_\emp$ is the total number of points in a Poisson point process, the ages of the $\Young$ labelled neighbours of $\emp$, conditioned on $\din_\emp=n$, are i.i.d.\ random variables with density \eqref{eq:density:A} \cite[Exercise 4.34]{resnick1992adventures}. Hence conditioned on $\din_\emp=n$, a uniformly chosen younger neighbour of the root $\emp$ has age distribution given by $A_1$ with density \eqref{eq:density:A}.
			Therefore conditionally on the root having at least one younger neighbour, a random $\Young$ labelled neighbour of the root has the degree distribution $M^{(0)}+Y\big(M^{(0)},A_{\young}\big)$ where $Y\big(M^{(0)},A_{\young}\big)$ is as defined earlier in Corollary~\ref{cor-older-younger-neighbours}, and
			\eqn{\label{for:cor:neighbour:8}
				\E_{\mu_\star} \big[ h_1(G,\emp) \big] = \prob\big(\din_\emp>0\big)\prob\Big( M^{(0)}+Y\big( M^{(0)},A_{\young} \big) =k \Big) = \prob\big(\din_\emp>0\big) \Tilde{p}_k^{({\young})}~.}
			On the other hand, $\E_{\mu_\star} \big[ h_2(G,\emp) \big]$ is the probability of the root $\emp$ having at least one younger neighbour and it is given by $\prob\big( \din_\emp>0 \big)$. Therefore,
			\eqn{\label{for:cor:neighbour:9}
				\E_{\mu_\star} \big[ h_2(G,\emp) \big] = \prob\big( \din_\emp>0 \big)~.}
			Since $\prob\big(\din_\emp>0\big)$ is non-zero, \eqref{eq:deg:young} follows immediately from \eqref{for:cor:neighbour:8} and \eqref{for:cor:neighbour:9}. This completes the proof of (a).
			
			\smallskip\noindent
			{The proof} of (b) makes use of similar calculations in \cite[Lemma 5.2]{BergerBorgs} and \cite[Proposition 1.4]{Dei} \arxiversion{and we defer this proof to Appendix~\ref{app:degree}.}\journalversion{{and the proof is in \cite[Appendix C]{GHvdHR22}.}}
		\qed

\arxiversion{
\begin{appendix}
%

\section{Additional proofs}
\label{sec:appendix}
\FC{In this section we prove Theorems~\ref{thm:equiv:CPU:PArt} and \ref{thm:equiv:PU:PAri}. {These proofs} follow {that} of} Theorem~\ref{thm:equiv:CPU:PArs}, except for a few {minor} adaptations. \FC{Here we show that models (D) and (B) are equal in distribution with $\PU^{\sss(\rm{NSL})}$ and $\CPU^{\sss(\rm{NSL})}$ respectively. Recall that $\CPU^{\sss(\rm{NSL})}$ \textcolor{red}{can be} obtained by collapsing a special case $\PU^{\sss(\rm{NSL})}$. So, we prove Lemma~\ref{lem:edge Probability:Polya Urn graph:NSL}-\ref{lem:PU:conditional:graph_probability:NSL} and Corollary~\ref{cor:cond:prob:CPU:NSL} for general choice of $\boldsymbol{m}$ and $\boldsymbol{\psi}$. Later, while proving Theorem~\ref{thm:equiv:CPU:PArt}, we shall use these lemmas and corollary with $\boldsymbol{m}=\boldsymbol{1}$, whereas for proving Theorem~\ref{thm:equiv:PU:PAri}, we continue with the $\boldsymbol{m}$ of model (D).} {The following} lemmas and propositions are the $\PU^{\sss\rm{(NSL)}}$ {analogues} of the lemmas and propositions proved in Section~\ref{subsec:equivalence:(A)}. The conditional {edge-connection probabilities} for $\PU^{\sss\rm{(NSL)}}$ {are} given {as follows:}
\begin{Lemma}[Conditional edge-connection {probabilities} of $\PU^{\sss\rm{(NSL)}}$]\label{lem:edge Probability:Polya Urn graph:NSL}
Conditionally on $\boldsymbol{m}$ and $(\psi_k)_{k\geq 1}$, the probability of connecting an edge from {$u$ to $v$} in $\PU_{n}^{\sss\rm{(NSL)}}(\boldsymbol{m},\boldsymbol{\psi})$ is given by $\psi_v (1-\psi)_{(v,u)}$.
\end{Lemma}
The proof to this lemma {is identical to that of} Lemma~\ref{lem:edge Probability:Polya Urn graph}, {using} the definition of  $\PU^{\sss\rm{(NSL)}}$.
{Now, we define the set of all vertex and edge-marked graphs on $n$ vertices by $\bar{\mathcal{H}}_n$. 
Let $\overline{\PU}^{\sss\rm{(NSL)}}$ and $\overline{{\rm PA}}^{\sss(D)}_{n}$ denote the edge-marked version of $\PU^{\sss\rm{(NSL)}}$ and $\PAri_{n}$ respectively. }
{Similarly} as Lemma~\ref{lem:PU:conditional:graph_probability}, 
the above lemma allows us to compute the conditional {law} of ${\overline{\PU}}^{\sss\rm{(NSL)}}$:
\begin{Lemma}[Conditional graph probability of ${\overline{\PU}^{\sss\rm{(NSL)}}}$]\label{lem:PU:conditional:graph_probability:NSL}
    For any graph $H\in{\bar{\mathcal{H}}_n}$,
    \begin{equation}
    \label{eq:PU:conditional:graph_probability}
        \prob_{m,\psi}\left( {\overline{\PU}_{n}^{\rm\sss(NSL)}}(\boldsymbol{m},\boldsymbol{\psi}) \overset{\star}{\simeq} H \right) = \prod\limits_{s\in[2,n]} \psi_s^{p_s}(1-\psi_s)^{q_s},
    \end{equation}
where
    \begin{align*}
        p_s=&~ d_s(H)-f_s,\qquad\text{and}\qquad
        q_s=\sum\limits_{u \in\left( 2,n\right]}\sum\limits_{j=1}^{m_u}\one_{\{s\in(v(u,j),u)\}},
    \end{align*}
    and where $v(u,j)$ is the vertex to which the $j$-th out-edge of $u$ connects and $f_s=m_s$ for all $s\geq 3$ and $f_1=a_1$ and $f_2=a_2$ denote the degrees of the vertex $1$ and $2$ in the initial graph $G_0$.
\end{Lemma}
Since the {proof strategy of} this lemma {is identical to that of} Lemma~\ref{lem:PU:conditional:graph_probability}, we omit the proof to this lemma also. {Using} Lemma~\ref{lemma:Beta Expectation}, conditionally on $\boldsymbol{m},$ the graph probability of ${\overline{\PU}}^{\sss\rm{(NSL)}}$ is computed as follows:

\begin{Corollary}\label{cor:cond:prob:CPU:NSL}
    For
    $H\in\mathcal{H}_n$,
    \eqn{\label{eq:lem:PU:probability}
    \prob_{m}\left( {\overline{\PU}}_{n}^{\rm\sss(NSL)}(\boldsymbol{m},\boldsymbol{\psi}) \overset{\star}{\simeq} H \right) = \prod\limits_{s\in \left[ 2  ,n \right)} \frac{(\alpha_s+p_s-1)_{p_s}(\beta_s+q_s-1)_{q_s}}{(\alpha_s+\beta_s+p_s+q_s-1)_{p_s+q_s}},
    }
    where $\alpha_s$ and $\beta_s$ are the first and second parameters of the $\Beta$ random variables
    and $p_s, q_s$ are defined in Lemma~\ref{lem:PU:conditional:graph_probability:NSL}.
\end{Corollary}
Following the {steps} in \eqref{for:equiv:CPU:PAr:8}-\eqref{for:equiv:CPU:PAr:10},
we can simplify $q_s$ as
    \eqn{
    \label{q_s:for:PU_NSL}
    q_s = d_{[s-1]}(H) - a_{[2]}-2\left( m_{[s-1]}-2 \right)-m_{s}\qquad \text{for}~{s\in [3,n]},
    }
{with $q_2=d_1(H)-a_1$, and $q_s$ satisfies}
    \eqn{
    \label{q_:p_s:reccursion}
    p_s+q_s=q_{s+1}+m_{s+1} \qquad \text{for}~{s\in [3,n)}.
    }
{Note that each vertex in $\PU_{m_{[n]}}^{\sss\mathrm{(NSL)}}(\boldsymbol{1},\boldsymbol{\psi})$ only has one out-edge. Therefore, the edge-marks are redundant in this case.}
Now, {we have} all tools to {adapt} Proposition~\ref{prop:equiv:CPU:PArs} to $\PU^{\sss\rm{(NSL)}}$:
\begin{Proposition}[Equivalence of pre-collapsed model (B) and $\PU^{\sss\rm{(NSL)}}$]\label{prop:equiv:CPU:PArt}
For any graph $H\in\mathcal{H}_{m_{[n]}}$,
\eqn{\label{prop:eq:equiv:CPU:PArt}
\prob_m\left( \PArt_n(\boldsymbol{m},1,\delta) \overset{\star}{\simeq} H \right) = \prob_m\left( \PU_{m_{[n]}}^{\sss\mathrm{(NSL)}}(\boldsymbol{1},\boldsymbol{\psi}) \overset{\star}{\simeq} H \right).}
\end{Proposition}
\begin{proof}
    Following a similar calculation {as the one leading to \eqref{for:equiv:CPU:PArs:5}}, 
    \eqan{\label{for:prop:equiv:PA:PU:NSL}
		&\prob_m\left( \PArt_n(\boldsymbol{m},1,\delta) \overset{\star}{\simeq} H \right)\nn\\
		=~& \prod\limits_{u\in[n]}\prod\limits_{j\in[m_{u}]} \prod\limits_{i=a_{m_{[u-1]}+j}}^{d_{m_{[u-1]}+j}(H)-1}\left( i+\frac{\delta}{m_u} \right)\\
		&\hspace{1.5cm}\times\prod\limits_{u\in[3,n]}\prod\limits_{j\in[m_{u}]} \frac{1}{a_{[2]}+2(m_{[u-1]}+j-3)+\left( (u-1)+{\frac{j-1}{m_u}} \right)\delta}~. \nn
	}
{For model (B),} we consider $\PU_{m_{[n]}}^{\sss\rm{(NSL)}}(\boldsymbol{1},\boldsymbol{\psi}),$ where $\boldsymbol{\psi}$ is the sequence of $\Beta$ variables defined in \eqref{def:psi:1-1} and \eqref{def:psi:1-2}. Therefore, $p_s=d_s(H)-1$ for $s\geq 1$, and 
    \[
        q_s= d_{[s-1]}(H) - a_{[2]}-2(s-3)-1\qquad\text{for}~s\geq 3.
    \]
By Corollary~\ref{cor:cond:prob:CPU:NSL},
    \eqn{\label{for:prop:equiv:CPU:PArt:1}
    \prob_{m}\left( \PU_{m_{[n]}}^{\rm\sss(NSL)}(\boldsymbol{1},\boldsymbol{\psi}) \overset{\star}{\simeq} H \right) = \prod\limits_{s\in \left[ 2  ,m_{[n]} \right)} \frac{(\alpha_s+p_s-1)_{p_s}(\beta_s+q_s-1)_{q_s}}{(\alpha_s+\beta_s+p_s+q_s-1)_{p_s+q_s}},
    }
where $\alpha_s$ and $\beta_s$ are the first and second parameters of the $\Beta$ variable $\psi_s$ defined in \eqref{def:psi:1-1} and \eqref{def:psi:1-2} respectively. Then, by {\eqref{q_:p_s:reccursion}, the recursion \eqref{for:prop:equiv:PU:PAr:4} again holds.}
Now, following the calculations {in \eqref{for:prop:equiv:PU:PAr:5}--\eqref{for:prop:equiv:PU:PAr:7},} and substituting the values of $\alpha_s,\beta_s,p_s,q_s$, it follows immediately that
    \eqan{\label{for:prop:equiv:CPU:PArt:3}
        \prob_{m}\left( \PU_{m_{[n]}}^{\rm\sss(NSL)}(\boldsymbol{1},\boldsymbol{\psi}) \overset{\star}{\simeq} H \right)
        =& \prod\limits_{u\in[n]}\prod\limits_{j\in{[m_{u}]}} \prod\limits_{i=a_{m_{[u-1]}+j}}^{d_{m_{[u-1]}+j}(H)-1}\left( i+\frac{\delta}{m_u} \right)\nn\\
        &\hspace{0.2cm}\times\prod\limits_{u\in[3,n]}\prod\limits_{j\in{[m_{u}]}} \frac{1}{a_{[2]}+2(m_{[u-1]}+j-3)+\left( (u-1)+\frac{{j-1}}{m_u} \right)\delta}~,
    }
{as required.}
\end{proof}
\begin{proof}[ Proof of Theorem~\ref{thm:equiv:CPU:PArt}]
    Theorem~\ref{thm:equiv:CPU:PArt} follows immediately from Proposition~\ref{prop:equiv:CPU:PArt} {in exactly the same} way {as} Theorem~\ref{thm:equiv:CPU:PArs} follows from Proposition~\ref{prop:equiv:CPU:PArs}.
\end{proof}
{Unlike model (B), in Theorem~\ref{thm:equiv:PU:PAri}, we couple model (D) to $\PU_{n}^{\rm\sss(NSL)}(\boldsymbol{m},\boldsymbol{\psi})$, where the edge-marks are relevant. Therefore, for proving Theorem~\ref{thm:equiv:PU:PAri}, we first prove that the coupling holds true for the edge-marked $\PAri$ and ${\PU}^{\sss \rm{(NSL)}}$. Next, we sum over all possible edge-marks to complete the proof of Theorem~\ref{thm:equiv:PU:PAri}. }
\begin{proof}[ Proof of Theorem~\ref{thm:equiv:PU:PAri}]
    Conditionally on $\boldsymbol{m},$ the distribution of marked model (D) is
    \eqan{\label{eq:PAri:density:1}
         &\prob_m\left(  {\overline{\PA}^{\sss \rm{(D)}}_{n}}(\boldsymbol{m},\delta) \overset{\star}{\simeq} H \right) \nn\\
         &\hspace{1cm}= \prod\limits_{u\in[3,n]}\prod\limits_{j\in[m_{u}]} \frac{d_{v(u,j)}(u,j-1)+\delta}{a_{[2]}+2(m_{[u-1]}-2)+(j-1)+(u-1)\delta}~,
    }
where $v(u,j)$ is the vertex in $[u-1]$ to which $u$ connects with its $j$-th edge and $d_v(u,j)$ denotes the degree of the vertex $v$ in ${\overline{\PA}}^{\sss \rm{(D)}}_{u,j}(\boldsymbol{m},\delta).$ Rearranging the numerators of RHS of \eqref{eq:PAri:density:1}, {we obtain, for $\bar{H}\in\bar{\mathcal{H}}_n$,}
    \eqan{\label{eq:PAri:density:2}
         \prob_m\left(  {\overline{\PA}^{\sss \rm{(D)}}_n}(\boldsymbol{m},\delta) \overset{\star}{\simeq} \bar{H} \right)
         &= {\prod\limits_{u\in[n]}\left(\prod\limits_{i=f_u}^{d_u(H)-1}(i+\delta)\right)}\nn\\
         &\hspace{.5cm}{\times\prod\limits_{u\in[3,n]}\left(\prod\limits_{j=1}^{m_u}\frac{1}{a_{[2]}+2(m_{[u-1]}-2)+(j-1)+(u-1)\delta}\right)}~.
        }
\FC{For model {(D)}, we use the $\boldsymbol{\psi}$ defined in \eqref{eq:def:psi:1}}. Using Corollary~\ref{cor:cond:prob:CPU:NSL}, we calculate the conditional {distribution} of $\overline{\PU}_n^{\sss\rm{(NSL)}}(\boldsymbol{m},\psi)$ as
    \eqn{\label{PU:density:1}
        \prob_{m}\left( {\overline{\PU}}_{n}^{\rm\sss(NSL)}(\boldsymbol{m},\boldsymbol{\psi}) \overset{\star}{\simeq} \bar{H} \right) = \prod\limits_{s\in \left[ 2  ,n \right)} \frac{(\alpha_s+p_s-1)_{p_s}(\beta_s+q_s-1)_{q_s}}{(\alpha_s+\beta_s+p_s+q_s-1)_{p_s+q_s}},}
where, for $s\geq 3,$
    \eqn{\label{parameters:PU:D}
        \begin{split}
         \alpha_s = m_s+\delta,\qquad&\text{and}\qquad \beta_s = {a_{[2]}+2\left( m_{[s-1]}-2 \right)+m_s}+(s-1)\delta,\\
         p_s = d_{s}(H)-f_s,\qquad&\text{and}\qquad q_s = d_{[s-1]}(H)-\left( a_{[2]}+2\left( m_{[s]}-2 \right) \right)+m_s.
        \end{split}}
{Therefore, for $s\geq 2$, the recursion relation in \eqref{for:prop:equiv:PU:PAr:4} becomes}
    \eqn{\label{recursion:2}
        \begin{split}
         \alpha_s+\beta_s=&\beta_{s+1}-m_{s+1},\qquad p_s+q_s=q_{s+1}+m_{s+1}.
        \end{split}}
{This gives us all the necessary tools} to prove Theorem~\ref{thm:equiv:PU:PAri}. From the recursion relation in \eqref{recursion:2}, it follows that
    \eqn{\label{for:thm:PU:PAri:1}
    \begin{split}
        \frac{(\beta_s+q_s-1)_{q_s}}{(\alpha_s+\beta_s+p_s+q_s-1)_{q_s+p_s}}=&\frac{(\beta_s+q_s-1)_{q_s}}{(\beta_{s+1}+q_{s+1}-1)_{q_{s+1}+m_{s+1}}}\\
        =& \frac{1}{(\beta_{s+1}-1)_{m_{s+1}}}\frac{(\beta_s+q_s-1)_{q_s}}{(\beta_{s+1}+q_{s+1}-1)_{q_{s+1}}}.
    \end{split}}
On the other hand, the first factor in RHS of \eqref{PU:density:1} can be  rewritten as
    \eqan{\label{for:thm:PU:PAri:2}
        &\frac{(\alpha_2+p_2-1)_{p_2}(\beta_2+q_2-1)_{q_2}}{(\alpha_2+\beta_2+p_2+q_2)_{p_2+q_2}}\nn\\
        &\hspace{1cm}=\frac{1}{(\beta_3+q_3-1)_{q_3}}\prod\limits_{i=0}^{m_{3}-1}\frac{1}{a_{[2]}+i+2\delta}\prod\limits_{i\in[2]}\prod\limits_{j=f_i}^{d_i(H)-1}(j+\delta) .
    }
    Hence substituting the simplifications obtained from \eqref{for:thm:PU:PAri:1} and \eqref{for:thm:PU:PAri:2} in \eqref{PU:density:1},
    \eqn{\label{for:thm:PU:PAri:3}
    \begin{split}
        &\prob_{m}\left( {\overline{\PU}}_{n}^{\rm\sss(NSL)}(\boldsymbol{m},\boldsymbol{\psi}) \overset{\star}{\simeq} \bar{H} \right)\\
        =& \prod\limits_{u\in[n]}\prod\limits_{i=f_u}^{d_u(H)-1}(i+\delta)\prod\limits_{u\in[3,n]}\prod\limits_{j=1}^{m_u}\frac{1}{a_{[2]}+2(m_{[u-1]}-2)+(j-1)+(u-1)\delta},
    \end{split}}
{this proves that $\overline{\PA}^{\sss \rm{(D)}}_n(\boldsymbol{m},\delta)$ and $\overline{\PU}_n^{\sss \rm{(NSL)}}(\boldsymbol{m},\delta)$ have the same law. Summing over all possible permutations of edge-marks of $\bar{H}\in\bar{\mathcal{H}}_n$ completes the proof of Theorem~\ref{thm:equiv:PU:PAri}.}
\end{proof}
\section{Adapted proofs for Preliminary Results}\label{app:preliminary:result}
In this appendix, we prove {Proposition \ref{lem:position:concentration},} and the regularity of the Random P\'olya point tree {in Lemma \ref{lem:RPPT:property}.} {Versions of these results} were proved in \cite{BergerBorgs}, and we have adapted those {proofs} to our settings of PAMs with random out-degrees. \FC{Before proving the  position concentration result in Proposition~\ref{lem:position:concentration}, we prove an auxiliary lemma that we will use several times, and which is a direct application of the dominated convergence theorem and strong law of large numbers:}
\begin{Lemma}
 \label{prop:M-inverse:expectation}
 Let $X_1, X_2,\ldots$ be a sequence of i.i.d.\ random variables with $X_1>c$ for some $c>0$ a.s.\ and finite mean. Then, with $X_{[n]}=X_1+\cdots+X_n$,
 \eqn{\label{eq:prop:M-inverse:expectation}\E\left[ \frac{1}{X_{[n]}} \right] = \left( 1+o(1) \right)\frac{1}{n\E[X_1]}~.}
\end{Lemma}
\begin{proof}
    Note that $X_i\geq c>0$, hence both $1/\E[X]$ and $n/X_{[n]}$ have upper bounds $1/c$. By the strong law of large numbers,
    \[
        \frac{X_{[n]}}{n}\overset{a.s.}{\to} \E[X_1].
    \]
    Since $\E[X_1]>0$ and $n/X_{[n]}\leq 1/c$, by the dominated convergence theorem,
    \begin{align*}
        \E\left[ \frac{n}{X_{[n]}} \right] =& (1+o(1))\frac{1}{\E[X_1]}.
    \end{align*}
    Hence the lemma follows immediately.
\end{proof}
\begin{proof}[ Proof of Proposition \ref{lem:position:concentration}]
\textbf{Proof for }$\CPU$: We follow the line of proof provided in \cite[Lemma~3.1]{BergerBorgs}
with the adaptations as required for our case of i.i.d.\ out-degrees.
  
Fix $\omega,\varepsilon>0$, and let $\bar{\omega}= \log(1+\omega)$. We use the definition of $\mathcal{S}_k^{\sss(n)}$ to bound the error in estimating $\mathcal{S}_k^{\sss(n)}$ by $\left( \frac{k}{n} \right)^{\chi}$. For all $k\in[n-1]$,
    \eqn{\label{for:lem:position:concentration:1}
    {\mathcal{S}}_k^{\sss(n)} = \prod\limits_{l=m_{[k]}+1}^{m_{[n]}} (1-\psi_l) = \exp{\left[ \sum\limits_{l=m_{[k]}+1}^{m_{[n]}} \log (1-\psi_l) \right]},
    }
with $\mathcal{S}_n^{\sss(n)}\equiv 1$. We concentrate on the argument of the exponential in \eqref{for:lem:position:concentration:1}. Note that
    \eqn{
    \label{for:lem:position:concentration:2}
    \var \left( \log (1-\psi_l) \right) \leq \E \left[ \log^2 (1-\psi_l) \right] \leq \E\left[ \frac{\psi_l^2}{(1-\psi_l)^2} \right].}
By \eqref{for:lem:position:concentration:2} and Kolmogorov's maximal inequality,
    \eqn{\label{for:lem:position:concentration:3}
    \prob\Big( \max\limits_{l\in \left[ m_{[n]} -1 \right]} \Big| \sum\limits_{k=l+1}^{m_{[n]}} \log(1-\psi_k) -\E\Big[ \sum\limits_{k=l+1}^{m_{[n]}} \log(1-\psi_k) \Big] \Big| \geq \frac{\bar{\omega}}{2} \Big)\leq \frac{4}{\bar{\omega}^2}\E\left[\sum\limits_{i=2}^{m_{[n]}}\E_m\left[ \frac{\psi_i^2}{(1-\psi_i)^2} \right]\right].
    }
Equation~\eqref{for:lem:position:concentration:3} shows that the maximum of the fluctuations of the argument in \eqref{for:lem:position:concentration:1} can be bounded by the variances of the singles terms. By properties of the $\Beta$ distribution, and recalling that, for $u=3,\ldots,n$ and $j\in[m_u]$,
    \[
    \psi_{m_{[u-1]}+j}\sim \Beta \left( 1+\frac{\delta}{m_u}, a_{[2]} + 2\left( m_{[u-1]}+j-{3} \right) +(u-1)\delta +\frac{j-1}{m_u}\delta \right),
    \]
we can {bound}, for $u>2$ and $j\in[m_u]$,
    \eqn{\label{for:lem:position:concentration:4}
    \E_m\Bigg[ \frac{\psi_{m_{[u-1]}+j}^2}{\left( 1-\psi_{m_{[u-1]}+j} \right)^2} \Bigg] = \mathcal{O}\left( {\big(m_{{[u-1]}}+j\big)}^{-2} \right).}
Notice that {$m_{[n]}\geq n$ for all $n\geq 1,$} and hence $m_{[n]}\to \infty$ as $n\to\infty$. Equation~\eqref{for:lem:position:concentration:4} assures us that the sum on the RHS of \eqref{for:lem:position:concentration:3} is finite as $n\rightarrow\infty$. Therefore, we can fix $N_1(\bar{\omega})\in\N$ such that $\E\left[\sum\limits_{i=N_1}^\infty \frac{\psi_i^2}{(1-\psi_i)^2} \right]\leq { \varepsilon\bar{\omega}^2/4}.$ As a consequence, bounding the sum on the RHS of \eqref{for:lem:position:concentration:3} by the tail of the series, for $n>N_1,$
  \eqn{\label{for:lem:position:concentration:5}
  \begin{split}
      \prob\left( \max\limits_{i\in\left[ m_{[n]}{-1} \right]\setminus \left[ m_{[N_1]}{-2} \right]}\left| \sum\limits_{l=i+1}^{m_{[n]}} \log(1-\psi_l) - \E\left[ \sum\limits_{l=i+1}^{m_{[n]}} \log(1-\psi_l)  \right]  \right|\geq \frac{\bar{\omega}}{2} \right)\\
      \leq\frac{4}{\bar{\omega}^2}\E\left[ \sum\limits_{i=N_1}^\infty \frac{\psi_i^2}{(1-\psi_i)^2} \right]\leq \varepsilon~.
  \end{split}}
{Next, we wish to compare the expectations of $\sum\limits_{k=i}^{m_{[n]}}\log(1-\psi_k)$ and  $\sum\limits_{k=i}^{m_{[n]}}\psi_k$ for $n$ large enough.} Using the fact that, for $x\in(0,1)$, 	
    $$
    \left|\log(1-x){+}x\right|\leq x^2/(1-x),
    $$
and using \eqref{for:lem:position:concentration:4} we {bound}, for $m_{[N_1]}-2\leq i\leq m_{[n]}-1$, 
    \eqn{\label{for:lem:position:concentration:6}
    \Bigg| \E\Big[ \sum\limits_{l=i+1}^{m_{[n]}} \log(1-\psi_l) \Big] {+} \E\Big[ \sum\limits_{l=i+1}^{m_{[n]}} \psi_l \Big] \Bigg| \leq \E\Bigg[ \sum\limits_{l=i+1}^{m_{[n]}}\frac{\psi_l^2}{(1-\psi_l)} \Bigg] < \infty.
    }
{Similarly,} there exists $N_2(\omega)\in\N$ such that $\E\big[ \sum\limits_{l=N_2}^{\infty}\psi_l^2/(1-\psi_l) \big]\leq {\bar{\omega}/3}$ and ${1/\sqrt{N_2}}\leq {\bar{\omega}/6}$. On the other hand, for $u>2$ and $j\in[m_u]$,
    \eqn{\label{for:lem:position:concentration:8}
    \begin{split}
      \E_m\left[ \psi_{m_{[u-1]}+j} \right]  =& \frac{1+\frac{\delta}{m_u}}{a_{[2]}+2\left[ m_{[u-1]}+j-{3} \right]+(u-1)\delta + \frac{j}{m_u}\delta}\\
      =& \frac{1+\frac{\delta}{m_u}}{2 m_{[u-1]} +(u-1)\delta}\left( 1+ o\left( 1 \right)\right).
    \end{split}}
Therefore, for all $u>2,$
    \eqan{
    \label{for:lem:position:concentration:9}
      \E\Big[ \sum\limits_{j=1}^{m_u}\psi_{m_{[u-1]}+j}  \Big]&=\E \Big[ \E_m\Big[ \sum\limits_{j=1}^{m_u}\psi_{m_{[u-1]}+j} \Big] \Big]
      = \E\left[ \frac{m_u+\delta}{2m_{[u-1]}+\delta} \right]\left( 1+ o\left( 1 \right)\right)\nn\\
      &=(\E[M]+\delta)\E\left[\frac{1}{2m_{[u-1]}+{(u-1)}\delta} \right]\left( 1+ o\left( 1 \right)\right)\nn\\
      &=\frac{\E[M]+\delta}{(u-1)(2\E[M]+\delta)}\left( 1+ o\left( 1 \right)\right)\nn\\
      &=\frac{\chi}{u-1}\left( 1+ o\left( 1 \right)\right)~,
      }
by Lemma~\ref{prop:M-inverse:expectation}.
Now, using the bounds in \eqref{for:lem:position:concentration:6} and \eqref{for:lem:position:concentration:9}, for all $N_2\leq k\leq n,$
    \eqn{\label{for:lem:position:concentration:12}
    \begin{split}
      \Bigg| \E\Big[ \sum\limits_{i= m_{[k]}+1}^{m_{[n]}} \log(1-\psi_i) \Big] - \chi\log\left(\frac{k}{n}\right) \Bigg| \leq  o\left( k^{-1/2} \right) + \bar{\omega}/3\leq \frac{\bar{\omega}}{2}.
    \end{split}
    }
As a consequence, for $n>N_2,$
    \eqn{
    \label{for:lem:position:concentration:13}
    \max\limits_{k\in (N_2,n]} \left| \E\left[ \sum\limits_{i= m_{[k]}+1}^{m_{[n]}} \log(1-\psi_i) \right] - \chi\log\left(\frac{k}{n}\right) \right| \leq \frac{\bar{\omega}}{2}.
    } 
Let $N_0=\max\{ N_1,N_2 \}$. By \eqref{for:lem:position:concentration:13} and \eqref{for:lem:position:concentration:5}, for $n>N_0$,
    \eqn{\label{for:lem:position:concentration:14}
    \prob\left( \max\limits_{k\in (N_0,n]} \left| \sum\limits_{i= m_{[k]}+1}^{m_{[n]}} \log(1-\psi_i) - \chi\log\left(\frac{k}{n}\right) \right| \geq \bar{\omega} \right) \leq \varepsilon.
    }
Recalling that $\log \mathcal{S}_k^{\sss(n)} = \sum\limits_{i= m_{[k]}+1}^{m_{[n]}} \log(1-\psi_i)$ and $\bar{\omega}=\log(1+\omega)$,  \eqref{for:lem:position:concentration:14} {implies} that with probability at least $1-\varepsilon,$ for every $i=(N_0,n]$,
  \eqn{\label{for:lem:position:concentration:15}
  \frac{1}{1+\omega}\left( \frac{k}{n} \right)^\chi \leq \mathcal{S}_k^{\sss(n)}\leq (1+\omega)\left( \frac{k}{n} \right)^\chi.}
Since, $\frac{1}{1+\omega}\geq 1-\omega$, we obtain from \eqref{for:lem:position:concentration:15} that
    \eqn{\label{for:lem:position:concentration:16}
    \prob\Big( \bigcap\limits_{u\in(N_0,n]}\left\{ \left| \mathcal{S}_u^{\sss(n)} - \left( \frac{u}{n} \right)^\chi \right|\leq \omega \left( \frac{u}{n} \right)^\chi \right\} \Big)\geq 1-\varepsilon,
    }
which proves \eqref{eq:lem:position:concentration:2}.  

{Finally, to} prove \eqref{eq:lem:position:concentration:1}, we observe that, for fixed $\omega>0$ and $\varepsilon>0,\ \left(\frac{N_0}{n}\right)^\chi \leq {\omega/3}$ for large enough $n$. Now,
    \eqn{\label{for:lem:position:concentration:17}
    \max\limits_{u\in[N_0]}\left| \mathcal{S}_u^{\sss(n)} - \left( \frac{u}{n} \right)^\chi \right| \leq \mathcal{S}_{N_0}^{\sss(n)} + \left(\frac{N_0}{n}\right)^\chi \leq \omega,
    }
    as required.
    
    \medskip\noindent
\textbf{Proof for }$\PU$: 
    In P\'olya urn graphs, conditionally on ${\boldsymbol{m}},$ for $u\geq {3}$, 
    \eqn{\label{for:prop:pos:conc:PU:0-0}\psi_u\sim\Beta\left( m_u+\delta,a_{[2]}+2(m_{[u-1]}-2)+m_u+(u-1)\delta \right).}
    We prove the position concentration lemma for $\PU$ following the proof for $\CPU$. Using Kolmogorov's inequality as in \eqref{for:lem:position:concentration:3}, 
    {\eqan{\label{for:prop:pos:conc:PU:0-1}
    &\prob\Big( \max\limits_{l\in[n-1]\setminus [n_1-2]}\big| \sum\limits_{k=l+1}^n \log (1-\psi_k) - \E\big[ \sum\limits_{k=l+1}^n \log (1-\psi_k) \big] \big|\geq \frac{\bar{\omega}}{2} \Big)\nn\\
    &\hspace{4cm}\leq \frac{4}{\bar{\omega}^2}\E\Big[ \sum\limits_{i=n_1}^n\frac{\psi_i^2}{(1-\psi_i)^2)} \Big]~.}}
        Now similarly as in \eqref{for:lem:position:concentration:4} we first bound $\E\Big[ \frac{\psi_i^2}{(1-\psi_i)^2} \Big]$ by $i^{-p}$ and thus we can choose $n_1$ large enough such that the RHS of \eqref{for:prop:pos:conc:PU:0-1} is smaller than $\varepsilon$. Then,
    \eqn{\label{for:prop:pos:conc:PU:0}\E\left[ \frac{\psi_u^2}{(1-\psi_u)^2} \right]\leq \E\Bigg[ \Big(\frac{m_u+\delta}{a_{[2]}+2(m_{[u-1]}-{2})+m_u+(u-1)\delta}\Big)^2 \Bigg]~.}
   { Since the random variable in the RHS of \eqref{for:prop:pos:conc:PU:0} is bounded above by $1,$ we can bound its second moment by its ${(p\wedge 2)}$-th moment and obtain    \eqn{\begin{split}
        \E\left[ \frac{\psi_u^2}{(1-\psi_u)^2} \right]\leq&\E\Bigg[ \Big(\frac{m_u+\delta}{a_{[2]}+2(m_{[u-1]}-{2})+m_u+(u-1)\delta}\Big)^{{p\wedge 2}} \Bigg]\\
        \leq& \E\Bigg[ \Big(\frac{m_u+\delta}{a_{[2]}+2(m_{[u-1]}-{2})+(u-1)\delta}\Big)^{{p\wedge 2}} \Bigg]~.
    \end{split}}}
 Now, the numerator and the denominator are independent of each other.
{For \(p\geq 2,~M\) has finite second moment and hence the numerator is finite. So, without loss of generality, we may assume that \(p\in(1,2)\).} Using the law of large numbers, {similar to the one used to prove Lemma~\ref{prop:M-inverse:expectation}}, and Lemma~\ref{prop:M-inverse:expectation}, we obtain

\begin{equation}\label{for:prop:pos:conc:PU:1-1}
\E\Bigg[ \Bigg(\frac{1}{a_{[2]} + 2(m_{[u-1]} - 2) + (u-1)\delta}\Bigg)^{{p}} \Bigg] = (1 + o(1))\frac{1}{u^{p}(2\E[M] + \delta)^{{p}}}~.
\end{equation}

Since $m_u$ has finite ${p}$-th moment, there exists a constant $\xi_0>0$ such that
\eqn{\label{for:prop:pos:conc:PU:1-2}
\E\left[ \frac{\psi_u^2}{(1-\psi_u)^2} \right]\leq \xi_0 u^{-{p}}~,}
which replaces \eqref{for:lem:position:concentration:4}. Now it remains to adapt a similar result to \eqref{for:lem:position:concentration:9} to complete the proof for $\PU$. By \eqref{for:prop:pos:conc:PU:0-0},
    \eqn{
    \label{for:prop:pos:conc:PU:1}
    \E[\psi_u]=\E\left[\frac{m_u+\delta}{a_{[2]}+2(m_{[u]}-2)+u\delta}\right]~.}
Note that the numerator and the denominator are not independent here. We can bound the RHS from above as
    \eqn{\label{for:prop:pos:conc:PU:2}
    \begin{split}
        \E\left[\frac{m_u+\delta}{a_{[2]}+2(m_{[u]}-2)+u\delta}\right]\leq \E\left[ \frac{m_u+\delta}{a_{[2]}+2(m_{[u-1]}-2)+(u-1)\delta} \right]=\frac{\chi}{u-1}(1+{o}(1))~.
    \end{split}}
On the other hand for the lower bound, we truncate $m_u$ at $\log (u)$ as
    \eqn{\label{for:prop:pos:conc:PU:3}
    \begin{split}
        \E\left[\frac{m_u+\delta}{a_{[2]}+2(m_{[u]}-2)+u\delta}\right]\geq\E\left[\frac{m_u\one_{\{m_u\leq \log u\}}+\delta}{a_{[2]}+2(m_{[u-1]}-2)+2\log (u)+u\delta}\right].
    \end{split}}
Again we can split the numerator and the denominator since they are now independent. Since $m_u$ has finite mean, $\E[m_u\one_{\{m_u\leq \log u\}}+\delta] = (\E[M]+\delta)(1+o(1)).$ {By Lemma~\ref{prop:M-inverse:expectation},}
    \eqn{\label{for:prop:pos:conc:PU:4}
    \begin{split}
        \E\left[\frac{m_u+\delta}{a_{[2]}+2(m_{[u]}-2)+u\delta}\right]\geq \frac{\E[M]+\delta}{(u-1)(2\E[M]+\delta)}(1+o(1))=\frac{\chi}{u-1}(1+o(1)).
    \end{split}}
Hence from \eqref{for:prop:pos:conc:PU:2} and \eqref{for:prop:pos:conc:PU:4}, we {obtain a} similar result as {in} \eqref{for:lem:position:concentration:8} as
    \eqn{
    \E[\psi_u]=\E\left[\frac{m_u+\delta}{a_{[2]}+2(m_{[u]}-2)+u\delta}\right] = \frac{\chi}{u-1}(1+o(1)).
    }
\end{proof}
\begin{proof}[ Proof of Lemma~\ref{lem:RPPT:property}]
The lemma holds {trivially} for $r=0$ and ${\varepsilon/4}$, i.e., with probability at least $1-{\varepsilon/4},$
    \begin{itemize}
        \item[(1)] $A_\emp \geq \eta\left( \frac{\varepsilon}{4},0\right)$, where $\eta\left( \frac{\varepsilon}{4},0\right)$ is a positive value depending on $\varepsilon$ and $r=0$;
        \item[(2)] $\big| B_0^{\sss(G)}(\emp) \big|=1$;
        \item[(3)] $\Gamma_\emp\leq K\left( \frac{\varepsilon}{4},0\right)$, where $K\left( \frac{\varepsilon}{4},0\right)$ is a natural number depending on $\varepsilon$ and $r=0$.
    \end{itemize}
We {thus} prove the lemma for $r=1$. The proof for $r\geq 2$ {then easily} follows by induction on $r.$
    
Let $U_{\sss(M)}$ be the smallest order statistic of $U_1,\ldots,U_M$, i.e., $U_{\sss(M)}=\min\{ U_1,\ldots,U_M \}$ and $\eta\leq \eta\left( \frac{\varepsilon}{4},0\right)$. To prove (\ref{RPPT_prop:1}), it is enough to show that
    \[
        \prob\Big(\left. A_\emp U_{\sss(M)}^{1/\chi}<\eta~\right|~ A_\emp>\eta( {\varepsilon/4},0)\Big)\leq \frac{\varepsilon}{4}.
    \]
    First, we condition on $M$ and find a lower bound on the LHS. Next, we take an expectation over $M$. Using the density of $U_{\sss(M)}$ and Taylor's inequality,
    \eqn{\label{for:lem:RPPT_prop:1}
    \begin{split}
        &\prob\big(\left. A_\emp U_{\sss(M)}^{1/\chi}<\eta~\right|~ M=m, A_\emp\geq \eta({\varepsilon/4},0) \big)\nn\\
        &\hspace{1cm}={\frac{1}{(1-\eta({\varepsilon/4},0))}}\int\limits_{\eta( {\varepsilon/4},0)}^1 \left( 1-\left( 1-\left( \frac{\eta}{x} \right)^\chi \right)^m \right)\,dx\\
        &\hspace{1cm}\leq {\frac{1}{(1-\eta({\varepsilon/4},0))}}\int\limits_{\eta({\varepsilon/4},0)}^1 m\left(\frac{\eta}{x}\right)^{\chi}\,dx\leq \frac{m\eta^\chi}{1-\chi}~,
    \end{split}}
{where the last inequality follows from the fact that $\frac{1-\eta({\varepsilon/4},0)^{1-\chi}}{1-\eta({\varepsilon/4},0)}\leq 1$.}
Integrating over $m$ and choosing $\eta$ suitably,
    \eqn{\label{for:lem:RPPT_prop:2}
    \prob\big(\left. A_{\emp} U_{\sss(M)}^{1/\chi}<\eta~\right|~ A_{\emp}\geq \eta( {\varepsilon/4},0) \big)\leq \frac{\eta^{\chi}}{2\E[M]+\eta}= \frac{\varepsilon}{4}.}
Denote the $\eta$ in \eqref{for:lem:RPPT_prop:2} by $\eta\left( {\varepsilon},1 \right)$. 

Next, we prove (\ref{RPPT_prop:2}). The number of children of the root is distributed as $(M+\Lambda)$, where $\Lambda$ is the total number of points in an inhomogeneous Poisson process with intensity $\rho_\emp(x).$ Since $M$ is uniformly integrable, there exists $m_0$ such that ${\E\left[ M\one_{\{M>m_0\}} \right]\leq\frac{\varepsilon}{8}}$. Moreover, it can also be shown that 
    \[
        \prob( M^{\sss(0)}>m_0) = \frac{1}{\E[M]}\sum\limits_{k\geq m_0+1} k\prob(M=k)=\frac{\E\left[ M\one_{\{M>m_0\}} \right]}{\E[M]}\leq\frac{\varepsilon}{8}.
    \]
    Similar results {also hold} for $M^{\sss(\delta)}$, {and are} useful for the induction step when $r>1$. Now, we have the parameter of $\Lambda$ bounded above and hence there exists $C(\varepsilon,1)$ such that 
    \[
        M+\Lambda\leq C(\varepsilon,1).
    \]
    We can choose $K(\varepsilon,1)<\infty$ such that $\prob\left(\Gamma_\omega\leq K(\varepsilon,1)\mbox{ for all }\omega\in B_1^{\sss(G)}(\emp)\right)\geq 1-\frac{\varepsilon}{4}.$
\end{proof}
\section{Tail degree distribution of the neighbours}\label{app:degree}
Even though Corollary~\ref{cor-older-younger-neighbours}(b) is about the tails of the $\Old$ and $\Young$ neighbours, since the proof follows the same way, we prove the tail of the root also. In Corollary~\ref{cor-older-younger-neighbours}(a), we have proved that the degree of a randomly chosen $\Old$ child of the root is given by ${1+M^{(\delta)}+Y(M^{(\delta)}+1,A_{\old})},$ where $A_{\old}$ is the age of a randomly chosen $\Old$ neighbour of the root; the degree of a randomly chosen $\Young$ child of the root is given by ${M^{(0)}+Y(M^{(0)},A_{\young})},$ where $A_{\young}$ is the age of a randomly chosen $\Young$ neighbour of the root. Lastly for the root, the degree is given by ${M+Y(M,U_\emp)}$ where $U_\emp$ is a uniform random variable on $[0,1].$ Therefore for all $t\in\N$,
\eqan{\label{for:appd:deg:1}
        &\prob(1+M^{(\delta)}+Y(M^{(\delta)}+1,A_{\old})=t)=\sum\limits_{m=1}^{t-1}\prob(M^{(\delta)}=m)\nn\\
        &\hspace{6cm}\times \int_0^1\prob(Y(m+1,a)=t-m-1)f_{\old}(a)\,da~,\\
        &\prob(M^{(0)}+Y(M^{(0)},A_{\young})=t)=\sum\limits_{m=1}^{t}\prob(M^{(0)}=m)\int_0^1\prob(Y(m,a)=t-m)f_{\young}(a)\,da~,\nn\\
        \mbox{and}\quad&\prob(M+Y(M,U_\emp)=t)=\sum\limits_{m=1}^{t}\prob(M=m)\int_0^1\prob(Y(m,a)=t-m)\,da~,\nn}
where $f_{\old}$ and $f_{\young}$ are the age distribution functions of a random $\Old$ and $\Young$ child of the root. If we manage to explicitly calculate the densities $f_{\old}$ and $f_{\young}$, then we are left to find out $\prob(Y(m,a)=t-m).~Y(m,a)$ is a mixed Poisson random variable with mixing parameter $\Gamma_{in}(m)\lambda(a)$. Therefore
    \eqan{\label{for:appd:deg:2}
        \prob(Y(m,a)=t-m) =& \int\limits_0^{\infty} \frac{(\gamma\lambda(a))^{t-m}}{(t-m)!}\exp{(-\gamma\lambda(a))}\frac{1}{\Gamma(m+\delta)}\gamma^{(m+\delta)-1} \exp{(-\gamma)}\,d\gamma\nn\\
        =& \frac{\lambda(a)^{t-m}}{(t-m)!\Gamma(m+\delta)}\int\limits_0^{\infty} \gamma^{(t+\delta)-1}\exp{(-\gamma(\lambda(a)+1))}\,d\gamma\\
        =& \frac{\Gamma(t+\delta)}{(t-m)!\Gamma(m+\delta)} \lambda(a)^{t-m}(1+\lambda(a))^{t+\delta}~.\nn}
    Substituting the value of $\lambda(a) = a^{-(1-\chi)}(1-a^{1-\chi})$ we obtain a compact form of the probability on the LHS of \eqref{for:appd:deg:2} as
    \eqn{\label{for:appd:deg:3}
    \prob(Y(m,a)=t-m) = \frac{\Gamma(t+\delta)}{\Gamma(m+\delta)(t-m)!}(1-a^{1-\chi})^{t-m}(a^{1-\chi})^{m+\delta}~.}
    This is a similar result as \cite[Lemma~5.2]{BergerBorgs}.  
\subsection{The age densities}
We proceed to explicitly calculate the densities $f_{\old}$ and $f_{\young}$:
\begin{Lemma}[Age of random $\Old$]\label{lem:old:age:density}
    The density of the age of a randomly chosen $\Old$ child of the root is given by
    \eqn{\label{eq:den:rand:old}
    f_{\old}(a) = \frac{\chi}{1-\chi} a^{-(1-\chi)}(1-a^{1-\chi}).}
\end{Lemma}
\begin{proof}
    From the definition of the $\RPPT(M,\delta)$ we have that $A_{\old}$ is distributed as $U_1U_2^{1/\chi}$ where $U_1$ and $U_2$ are independent uniform random variables on $[0,1]$. Therefore,
    \eqn{\label{for:lem:age:rand:old:1}
    \prob(A_{\old} \leq a) = \prob (U_1U_2^{1/\chi}\leq a) = \int_0^1 \prob(U_2\leq a^{\chi} x^{-\chi})\,dx~.
    }
    Note that for $x\in[0,a]$ the probability in the RHS of \eqref{for:lem:age:rand:old:1} is always $1$. For $x\in(a,1]$ the probability is $a^{\chi} x^{-\chi}$. Therefore the RHS of \eqref{for:lem:age:rand:old:1} simplifies as
    \eqn{\label{for:lem:age:rand:old:2}
    \prob(A_{\old}\leq a) = a + a^{\chi}\int_a^1 x^{-\chi}\,dx = a+\frac{a^{\chi}}{1-\chi}(1-a^{1-\chi})~.}
    Now differentiating the RHS of \eqref{for:lem:age:rand:old:2} with respect to $a$,
    \eqn{\label{for:lem:age:rand:old:3}
    f_{\old}(a) = 1 + \frac{\chi}{1-\chi}a^{-(1-\chi)}(1-a^{1-\chi})-\frac{1-\chi}{1-\chi}a^\chi a^{-\chi}=\frac{\chi}{1-\chi}a^{-(1-\chi)}(1-a^{1-\chi})~.}
\end{proof}
Similarly, we have seen that conditionally on the existence of at least one $\Young$ child its age distribution is given in equation \eqref{eq:density:A} of the main paper. We simplify the density here:
\begin{Lemma}[Age of random $\Young$]\label{lem:young:age:density}
    The density of the age of a randomly chosen $\Young$ child of the root is given by
    \eqn{\label{eq:den:rand:young}
    f_{\young}(a) = a^{-\chi}\int_0^{a^{1-\chi}}x^{\tau_e-2}(1-x)^{-1}\,dx~,}
    where $\tau_e=3+\delta/\E(M)$.
\end{Lemma}
\begin{proof}
    From equation \eqref{eq:density:A}, if $U_1$ is a uniform random variable on $[0,1]$, then
    \eqn{\label{for:lem:age:rand:young:1}
    \begin{split}
        f_{\young}(a)= \E\left[(1-\chi)\frac{a^{-\chi}\one_{\{ a\geq U_1 \}}}{1-U_1^{1-\chi}}\right] = (1-\chi)a^{-\chi}\int_0^a (1-y^{1-\chi})^{-1}\,dy~.
    \end{split}}
    Now we perform a change of variable operation by substituting $y^{1-\chi}=x$ in the RHS of \eqref{for:lem:age:rand:young:1}.
    \eqn{\label{for:lem:age:rand:young:2}
    f_{\young}(a) = a^{-\chi}\int_0^{a^{1-\chi}}x^{1/(1-\chi)-1}(1-x)^{-1}\,dx~.}
    Substituting $\tau_e=1+\frac{1}{1-\chi}$ in the RHS of \eqref{for:lem:age:rand:young:2} we obtain the desired form of the density of a randomly chosen $\Young$ neighbour of the root.
\end{proof}
\subsection{Tail calculations}
With Lemma~\ref{lem:old:age:density} and \ref{lem:young:age:density} and equation~\eqref{for:appd:deg:3} in hand, we evaluate the expressions in the RHS of \eqref{for:appd:deg:1} for the root, $\Old$ and $\Young$ neighbours. Before proceeding with the calculation, note that 
\eqn{\label{eq:tail:M}
\begin{split}
    \prob(M=t)=&~L(t) t^{-\tau_M}~,\\
    \prob(M^{(0)}=t)=&~L_1(t) t^{-(\tau_M-1)}~,\\
    \mbox{and}\qquad\prob(M^{(\delta)}=t)=&~L_2(t) t^{-(\tau_M-1)}~. \\
\end{split}}
for some slowly varying functions $L(t),L_1(t) \mbox{ and }L_2(t).$ On the other hand, using Stirling's approximation,
\eqn{\label{eq:sterling:ratio}
\begin{split}
    \frac{\Gamma(m+\delta+k)}{\Gamma(m+\delta)}=&~m^{k}(1+\mathcal{O}(1/m))~,\\
    \mbox{and}\qquad  \frac{\Gamma(t+\delta)}{\Gamma(t+\delta+k)}=&~t^{-k}(1+\mathcal{O}(1/{t}))~.
\end{split}}
\eqref{eq:tail:M} and \eqref{eq:sterling:ratio} will be useful in several steps in the tail calculation of the degree distributions.
\paragraph{{The root}}
First, we carry out the simplest of all these calculations. As we can see from \eqref{for:appd:deg:1}, we need to integrate the RHS of \eqref{for:appd:deg:3}.
\eqn{\label{for:thm:root:1}
    \int_0^1\prob(Y(m,a)=t-m)\,da = \frac{\Gamma(t+\delta)}{\Gamma(m+\delta)(t-m)!}\int_0^1 (1-a^{1-\chi})^{t-m}(a^{1-\chi})^{m+\delta}\,da~.}
Now we do the same change of variable as we did in \eqref{for:lem:age:rand:young:1} and obtain
\eqan{\label{for:thm:root:2}
    \int_0^1\prob(Y(m,a)=t-m)\,da =& \frac{\Gamma(t+\delta)}{(1-\chi)\Gamma(m+\delta)(t-m)!}\int_0^1 u^{(m+\delta+\tau_e-1)-1}(1-u)^{t-m} \,du\nn\\
    =& \frac{(\tau_e-1)\Gamma(t+\delta)\Gamma(t-m+1)\Gamma(m+\delta+\tau_e-1)}{\Gamma(m+\delta)(t-m)!\Gamma(t+\delta+\tau_e)}\\
    =& \frac{(\tau_e-1)\Gamma(t+\delta){\Gamma(m+\delta+\tau_e-1)}}{\Gamma(m+\delta)\Gamma(t+\delta+\tau_e)}~.\nn}
This matches with the expression presented in \cite{Dei}. Now plugging this integral value in \eqref{for:appd:deg:1} and using \eqref{eq:tail:M} and \eqref{eq:sterling:ratio}, \eqan{\label{taiL:root:1}
    \prob(M+Y(M,U_\emp)=t) =&\frac{(\tau_e-1)\Gamma(t+\delta)}{\Gamma(t+\delta+\tau_e)} \sum\limits_{m=1}^{t}\prob(M=m)\frac{{\Gamma(m+\delta+\tau_e-1)}}{\Gamma(m+\delta)}\nn\\
    =&(\tau_e-1)t^{-\tau_e}(1+\mathcal{O}(1/t))\sum\limits_{m=1}^{t} L(m)m^{-(\tau_M-\tau_e+1)}(1+\mathcal{O}(1/m))~.}
For $\tau_M>\tau_e$, the sum on the RHS of \eqref{taiL:root:1} is finite. On the other hand, for $\tau_M\leq\tau_e$ the sum varies regularly as $t^{-(\tau_M-\tau_e)}$ and hence
\eqn{\label{tail:root:2}
\prob(M+Y(M,U_\emp)=t) = L^{(\emp)}(t)t^{-\tau}~,}
where ${\tau=\min\{ \tau_M,\tau_e \}}$ and $L^{(\emp)}(\cdot)$ is a slowly varying function.
\paragraph{{The $\Old$ child}}
From Lemma~\ref{lem:old:age:density}, we have the expression for $f_{\old}(a)$ and therefore the integral in \eqref{for:appd:deg:1} can be simplified as
\eqn{\label{for:thm:old:1}
\begin{split}
    &\int_0^1\prob(Y(m,a)=t-m)f_{\old}(a)\,da\\
    =& \frac{\Gamma(t+\delta)}{\Gamma(m+\delta)(t-m)!}\int_0^1 (1-a^{1-\chi})^{t-m}(a^{1-\chi})^{m+\delta} \frac{\chi}{1-\chi} a^{-(1-\chi)}(1-a^{1-\chi})\,da\\
    =& (\tau_e-2)\frac{\Gamma(t+\delta)}{\Gamma(m+\delta)(t-m)!}\int_0^1 (1-a^{1-\chi})^{t-m+1}(a^{1-\chi})^{m+\delta-1}\,da~.
\end{split}}
Again we do the same change of variable and simplify the RHS of \eqref{for:thm:old:1} as
\eqn{\label{for:thm:old:2}
\begin{split}
    &\int_0^1\prob(Y(m,a)=t-m)f_{\old}(a)\,da\\
    =& (\tau_e-2)(\tau_e-1)\frac{\Gamma(t+\delta)}{\Gamma(m+\delta)(t-m)!}\int_0^1 (1-u)^{t-m+1}u^{m+\delta+\tau_e-3}\,da\\
    =& (\tau_e-2)(\tau_e-1)\frac{\Gamma(t+\delta)}{\Gamma(m+\delta)(t-m)!} \frac{\Gamma(m+\delta+\tau_e-2)\Gamma(t-m+2)}{\Gamma(t+\delta+\tau_e)}\\
    =& (\tau_e-2)(\tau_e-1){(t-m+1)}\frac{\Gamma(t+\delta){\Gamma(m+\delta+\tau_e-2)}}{\Gamma(m+\delta)\Gamma(t+\delta+\tau_e)}~.
\end{split}}
Therefore,
\eqan{\label{for:thm:old:3}
&\int_0^1\prob(Y(m+1,a)=t-m-1)f_{\old}(a)\,da \nn\\
=&(\tau_e-2)(\tau_e-1){(t-m)}\frac{\Gamma(t+\delta){\Gamma(m+\delta+\tau_e-1)}}{{\Gamma(m+1+\delta)}\Gamma(t+\delta+\tau_e)}.}
Now substituting this integral value in \eqref{for:appd:deg:1}, the sum in the RHS could be simplified as 
\eqan{\label{tail:old:1}
    &\prob(1+M^{(\delta)}+Y(M^{(\delta)}+1,A_{\old})=t)\nn\\
    &\qquad =(\tau_e-2)(\tau_e-1)\frac{\Gamma(t+\delta)}{\Gamma(t+\delta+\tau_e)}\sum\limits_{m=1}^{t-1}\prob(M^{(\delta)}=m){(t-m)}\frac{{\Gamma(m+\delta+\tau_e-1)}}{{\Gamma(m+1+\delta)}}\\
    &\qquad = (\tau_e-2)(\tau_e-1) t^{-\tau_e}(1+\mathcal{O}(1/t))\sum\limits_{m=1}^{t-1}L_1(m)m^{-(\tau_M-\tau_e+1)}(t-m)(1+\mathcal{O}(1/m))~.\nn}
Now if $\tau_M\leq \tau_e$, the sum in \eqref{tail:old:1} varies regularly as $t^{(\tau_e-\tau_M)+1}$. On the other hand for $\tau_M>\tau_e$, the sum varies regularly as $t$. Therefore there exists a slowly varying function $L^{({\old})}(t)$ such that
\eqn{\label{tail:old:2}
\prob(1+M^{(\delta)}+Y(M^{(\delta)}+1,A_{\old})=t) = L^{({\old})}(t) t^{-\tau_{\sss(\old)}}~,}
where ${\tau_{\sss(\old)}=\min\{ \tau_M,\tau_e \}-1}$. 
\FC{For $\tau_M>\tau_e,~L^{\sss(\old)}$ turns out to be a constant. On the other hand, for $\tau_M\leq \tau_e,$ using Karamata's Theorem \cite[Proposition~1.5.8]{bingham_goldie_teugels_1987}, $L^{\sss(\old)}$ can be shown to be asymptotically equal to $L_1/\tau_{\sss(\old)}$.}
\paragraph{{The $\Young$ child}}
The calculation for $\Young$ neighbours of the root is more involved. We first substitute the $f_{\young}(a)$ in \eqref{for:appd:deg:1}
\eqan{\label{for:thm:young:1}
    &\int_0^1\prob(Y(m,a)=t-m)f_{\young}(a)\,da\nn\\
    &\qquad= \frac{\Gamma(t+\delta)}{\Gamma(m+\delta)(t-m)!}\int_0^1(1-a^{1-\chi})^{t-m}(a^{1-\chi})^{(m+\delta)}a^{-\chi}\\
    &\hspace{4cm}\times\int_0^{a^{1-\chi}}x^{\tau_e-2}(1-x)^{-1}\,dx\,da~. \nn}
We perform the usual change of variable $u=a^{1-\chi}$ and simplify the above equation as
\eqn{\label{for:thm:young:2}
\begin{split}
    &\int_0^1\prob(Y(m,a)=t-m)f_{\young}(a)\,da\\
    &\qquad = (\tau_e-1)\frac{\Gamma(t+\delta)}{\Gamma(m+\delta)(t-m)!}\int_0^1 (1-u)^{t-m}u^{m+\delta}\int_0^u x^{\tau_e-2}(1-x)^{-1}\,dx\,du~.
\end{split}}
Using the fact that $(1-x)^{-1}=\sum\limits_{i\geq 0}x^i$ we simplify the inner integral as
\eqn{\label{for:thm:young:3}
\int_0^u x^{\tau_e-2}(1-x)^{-1}\,dx = \sum\limits_{i\geq 0}\int_0^u x^{\tau_e+i-2}\,dx=\sum\limits_{i\geq 0}\frac{u^{\tau_e+i-1}}{\tau_e+i-1}~.}
Substituting the RHS of \eqref{for:thm:young:3} in \eqref{for:thm:young:2},
\eqan{\label{for:thm:young:4}
    &\hspace{-1cm}\int_0^1\prob(Y(m,a)=t-m)f_{\young}(a)\,da\nn\\
    =& \frac{(\tau_e-1)\Gamma(t+\delta)}{\Gamma(m+\delta)(t-m)!}\sum\limits_{i\geq 0}\frac{1}{\tau_e+i-1}\int_0^1 (1-u)^{(t-m+1)-1}u^{(m+\delta+\tau_e+i)-1}\,du\nn\\
    =& (\tau_e-1)\frac{\Gamma(t+\delta)\Gamma(t-m+1)}{\Gamma(m+\delta)(t-m)!}\sum\limits_{i\geq 0}\frac{1}{(\tau_e+i-1)} \frac{\Gamma(m+\delta+\tau_e+i)}{\Gamma(t+\delta+\tau_e+i+1)}\\
    =& (\tau_e-1)\frac{\Gamma(t+\delta)}{\Gamma(m+\delta)}\sum\limits_{i\geq 0}\frac{1}{(\tau_e+i-1)} \frac{\Gamma(m+\delta+\tau_e+i)}{\Gamma(t+\delta+\tau_e+i+1)}\nn\\
    =&\frac{\Gamma(t+\delta)\Gamma(m+\delta+\tau_e)}{\Gamma(m+\delta)\Gamma(t+\delta+\tau_e+1)}\nn\\
    &\hspace{3cm}+(\tau_e-1)\frac{\Gamma(t+\delta)}{\Gamma(m+\delta)}\sum\limits_{i\geq 1}\frac{1}{(\tau_e+i-1)} \frac{\Gamma(m+\delta+\tau_e+i)}{\Gamma(t+\delta+\tau_e+i+1)}~.\nn}
Now the RHS of \eqref{for:thm:young:4} is lower bounded by the first term and hence the RHS of \eqref{for:appd:deg:1} can be lower bounded as
\eqan{\label{for:thm:young:5}
    \prob(M^{(0)}+Y(M^{(0)},A_{\young})=t)
    &\geq\sum\limits_{m=1}^{t}\prob(M^{(0)}=m) \frac{\Gamma(t+\delta)\Gamma(m+\delta+\tau_e)}{\Gamma(m+\delta)\Gamma(t+\delta+\tau_e+1)}\nn\\
    &=t^{-(\tau_e+1)}(1+\mathcal{O}(1/t))\sum\limits_{m=1}^{t}L_2(m)m^{-(\tau_M-\tau_e-1)}(1+\mathcal{O}(1/m))~.}
It can easily be shown that the RHS of \eqref{for:thm:young:5} varies regularly with $t^{-(\tau_e+1)}$ if $\tau_M>\tau_e+2$. When $\tau_M\leq \tau_e+2$, it varies regularly with $t^{-(\tau_M-1)}$. Therefore the lower bound varies regularly with $t^{-\tau_{(\young)}}$, where $\tau_{(\young)}=\min\{\tau_M-1,\tau_e+1\}$ and some slowly varying function $L_2^\prime(t)$. For $\tau_M>\tau_e+2,~L_2^\prime(t)$ turns out to be a constant and for $\tau_M=\tau_e+2,$
\eqn{
    L_2'(t)=\sum_{t=1}^m L_2(m)/m,
    }
and lastly for $\tau_M<\tau_e+2$, we use the same Karamata's Theorem as before to obtain that $L_2^\prime(t)=\Theta(L_2(t)).$
Now we move on to analyse the second term of RHS of \eqref{for:thm:young:4}.
\eqn{\label{for:thm:young:6}
\begin{split}
    &(\tau_e-1)\frac{\Gamma(t+\delta)}{\Gamma(m+\delta)}\sum\limits_{i\geq 1}\frac{1}{(\tau_e+i-1)} \frac{\Gamma(m+\delta+\tau_e+i)}{\Gamma(t+\delta+\tau_e+i+1)}\\
    =& (\tau_e-1)\frac{\Gamma(t+\delta)}{\Gamma(m+\delta)}\sum\limits_{i\geq 1}\frac{1}{(\tau_e+i-1)(m+\delta+\tau_e+i)} \frac{\Gamma(m+\delta+\tau_e+i+1)}{\Gamma(t+\delta+\tau_e+i+1)}\\
    =&(\tau_e-1)\frac{\Gamma(t+\delta)\Gamma(m+\delta+\tau_e+2)}{\Gamma(m+\delta)\Gamma(t+\delta+\tau_e+2)}\sum\limits_{i\geq 0}\frac{1}{(\tau_e+i)(m+\delta+\tau_e+i+1)}\prod\limits_{j=0}^{i-1}\frac{m+\delta+\tau_e+j+2}{t+\delta+\tau_e+j+2}~.
\end{split}}
\begin{Lemma}\label{lemma:upper:bound:young:tail}
  For any $m\geq 1,$
    \eqan{\label{eq:lemma:upper:bound:young:tail}
    &\sum\limits_{i\geq 0}\frac{1}{(\tau_e+i)(m+\delta+\tau_e+i+1)}\prod\limits_{j=0}^{i-1}\frac{m+\delta+\tau_e+j+2}{t+\delta+\tau_e+j+2}\\
     =&~ \mathcal{O}(m^{-1}(2-\log (1-2m/(t+m+1))).\nn
    }
\end{Lemma}
Subject to Lemma~\ref{lemma:upper:bound:young:tail}, there exists $J_0>0$ such that the RHS of \eqref{for:thm:young:6} can be upper bounded by $J_0 (1+\log (1-2m/(t+m+1))) t^{-(\tau_e+2)}m^{(\tau_e+1)}(1+\mathcal{O}(1/t))(1+\mathcal{O}(1/m))$. Therefore, using \eqref{for:thm:young:4}, \eqref{for:thm:young:5}, \eqref{for:thm:young:6} and Lemma~\ref{lemma:upper:bound:young:tail}, the RHS of \eqref{for:appd:deg:1} can be upper bounded as
    \eqan{
    \label{for:thm:young:7}
    &\prob(M^{(0)}+Y(M^{(0)},A_{\young})=t) \nn \\
    \leq& L_2^\prime(t)t^{-\tau_{(\young)}}+J_0t^{-(\tau_e+2)}(1+\mathcal{O}(1/t))\nn\\
    &\hspace{2cm}\times\sum\limits_{m=1}^t L_2(m)(2-\log (1-2m/(t+m+1)))m^{-(\tau_M-\tau_e-2)}(1+\mathcal{O}(1/m))~.
}
When $\tau_M>\tau_e+2,$ the second sum is $o(t)$ and hence the second term is $o(t^{-(\tau_e+1)})$. Therefore the tail degree distribution of the $\young$ child varies regularly with $t^{-\tau_{(\young)}}$ when $\tau_M>\tau_e+2$.

On the other hand, when $\tau_M\leq \tau_e+2$, the upper bound can be shown to vary regularly with $t^{-\tau_{(\young)}}$ and the slowly varying function is again $\Theta(L_2^\prime(t))$. Therefore considering $L_2^\prime(t)$ as $L^{({\young})}(t)$, we can say that 
\eqn{\label{tail:young:final}
\prob(M^{(0)}+Y(M^{(0)},A_{\young})=t)=\Theta(L^{({\young})}(t))t^{-\tau_{(\young)}}~.} 
From our calculation here we could not the exact slowly varying function in $t$ when $\tau_M\leq \tau_e+2$. We think this could also be shown by tweaking the sum in \eqref{for:thm:young:4} properly.
\medskip

We now complete the argument. Equations \eqref{tail:young:final} and \eqref{tail:old:2} complete the proof of Corollary~\ref{cor-older-younger-neighbours}(b) and \eqref{tail:root:2} proves the claims in  \eqref{degree-convergence}-\eqref{asymptotic-degree} which matches with \cite[Proposition~1.4]{Dei}.
It remains to prove Lemma~\ref{lemma:upper:bound:young:tail}:

\begin{proof}[ Proof of Lemma~\ref{lemma:upper:bound:young:tail}]
For $i\geq m-\delta-\tau_e-1$, we bound
    \eqn{
    \prod\limits_{j=0}^{i-1}\frac{m+\delta+\tau_e+j+2}{t+\delta+\tau_e+j+2}\leq 1,
    }
so that
    \eqn{
    \sum_{i\geq m-\delta-\tau_e-1}
    \frac{1}{(\tau_e+i)(m+\tau_e+i+\lfloor\delta\rfloor)}\leq 
    \sum_{i\geq m-\delta-\tau_e-1}
    \frac{1}{(\tau_e+i)(\tau_e+i+\lfloor\delta\rfloor)}=\Theta(1/m),
    }
as required.

For $i<m-\delta-\tau_e-1$, instead, we rewrite the summands in \eqref{eq:lemma:upper:bound:young:tail} as 
\eqn{\begin{split}
    &\frac{1}{(\tau_e+i)(m+\delta+\tau_e+i+1)}\prod\limits_{j=0}^{i-1}\frac{m+\delta+\tau_e+j+2}{t+\delta+\tau_e+j+2}\\
    =&\frac{1}{(\tau_e+i)(m+\delta+\tau_e+1)}\prod\limits_{j=0}^{i-1}\frac{m+\delta+\tau_e+j+1}{t+\delta+\tau_e+j+2}~.
\end{split}}
Now we bound the product terms as
    \eqn{
    \frac{m+\delta+\tau_e+j+1}{t+\delta+\tau_e+j+2}
    \leq \frac{m+\delta+\tau_e+i-1+1}{t+\delta+\tau_e+i-1+2}
    \leq 
    \frac{2m}{t+m+1}.
    }
Therefore,
    \eqn{
    \prod\limits_{j=0}^{i-1}\frac{m+\delta+\tau_e+j+1}{t+\delta+\tau_e+j+2}\leq 
    \Big(\frac{2m}{t+m+1}\Big)^i.
    }
We conclude that, also using that $\tau_e\geq 1$,
    \eqan{
    &\sum_{i=0}^{m-\delta-\tau_e-1}
    \frac{1}{(\tau_e+i)(m+\delta+\tau_e+i+1)}\prod\limits_{j=0}^{i-1}\frac{m+\delta+\tau_e+j+2}{t+\delta+\tau_e+j+2}\\
    &\qquad \leq \frac{1}{m+\delta+\tau_e+1}+\frac{1}{m}
    \sum_{i=1}^{m-\delta-\tau_e-1}
    \frac{1}{i}\Big(\frac{2m}{t+m+1}\Big)^i\nn\\
    &\qquad\leq \Theta(1/m)
    +\frac{1}{m}
    \sum_{i\geq 1}
    \frac{1}{i}\Big(\frac{2m}{t+m+1}\Big)^i\nn\\
    &\qquad=\Theta(m^{-1}(1-\log(1-2m/(t+m+1)))),\nn
    }
as required.
\end{proof}
\color{black}
\end{appendix}}


\noindent
\paragraph{\bf Acknowledgments}
The authors thank Shankar Bhamidi, Tiffany Y.\ Y.\ Lo, Haodong Zhu and Delphin S\'enizergues for their insightful suggestions on this paper. 

\noindent\parbox{0.85\textwidth}
{\begin{funding}
The work of AG, RSH, RvdH and RR is supported in part by the Netherlands Organisation for Scientific Research (NWO) through the Gravitation {\sc Networks} grant 024.002.003. The work of RvdH is further supported by the Netherlands Organisation for Scientific Research (NWO) through 
VICI grant 639.033.806. The work of RR is further supported by the European Union’s Horizon 2020 research and innovation programme under the Marie Skłodowska-Curie grant agreement no.\ 945045 and the Office of Naval Research under the Vannevar Bush Faculty Fellowship N0014-21-1-2887.
\end{funding}}
\parbox{0.15\textwidth}
{\vspace{0.5cm}

~~~\includegraphics[width=0.14\textwidth]{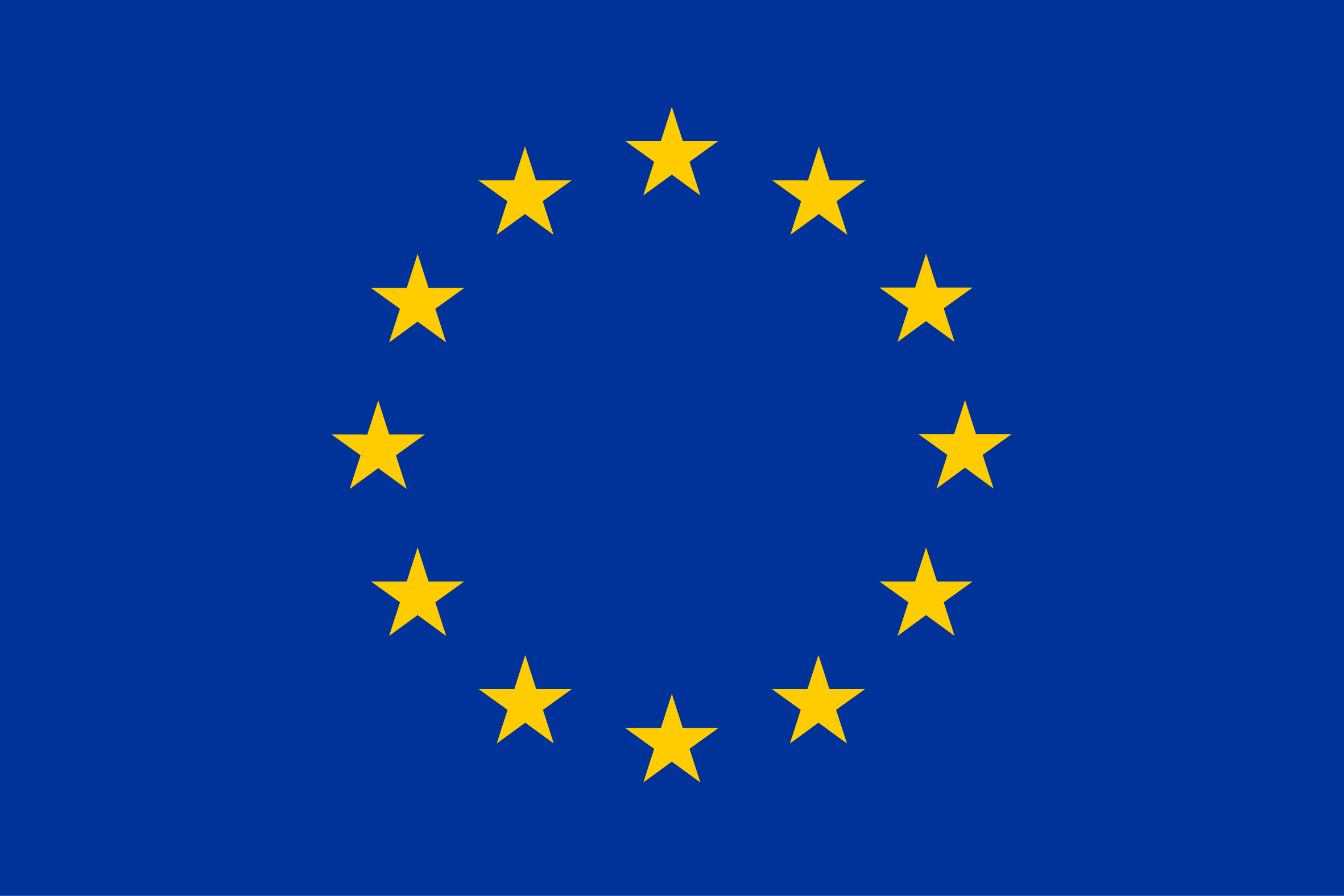}}

\bibliographystyle{imsart-number} 
\bibliography{reference}       


\end{document}